\documentclass[a4paper,11pt,reqno]{amsart}
\usepackage{hyperref}
\usepackage{amsmath}
\usepackage{amssymb}
\usepackage{amsthm}
\usepackage{amsfonts}
\usepackage{amstext}
\usepackage{amsopn}
\usepackage{amsxtra}
\usepackage{mathrsfs}
\usepackage{dsfont}
\usepackage{esint}
\usepackage{enumitem}
\usepackage{color}
\usepackage{graphicx}

\newtheorem{theorem}{Theorem}[section]
\newtheorem{lemma}[theorem]{Lemma}

\newtheorem{corollary}[theorem]{Corollary}
\newtheorem{remark}[theorem]{Remark}
\newtheorem{example}[theorem]{Example}
\newtheorem{definition}[theorem]{Definition}
\newcommand{\R}{\mathbb{R}}
\newcommand{\Z}{\mathbb{Z}}

\newcommand{\N}{\mathbb{N}}
\newcommand{\bP}{\mathbb{P}}
\newcommand{\bQ}{\mathbb{Q}}
\newcommand{\bR}{\mathbb{R}}
\newcommand{\bS}{\mathbb{S}}
\newcommand{\bG}{\mathbb{G}}
\newcommand{\dps}{\displaystyle}
\newcommand{\ii}{\infty}
\newcommand\1{{\ensuremath {\mathds 1} }}
\renewcommand\phi{\varphi}
\newcommand{\gS}{\mathfrak{S}}
\newcommand{\wto}{\rightharpoonup}
\newcommand{\cS}{\mathcal{S}}

\newcommand{\cC}{\mathcal{C}}

\newcommand{\cX}{\mathcal{X}}

\newcommand{\cF}{\mathcal{F}}

\newcommand{\cD}{\mathcal{D}}

\newcommand{\cH}{\mathcal{H}}

\newcommand{\bc}{\mathbf{c}}
\newcommand{\bx}{\mathbf{x}}
\newcommand{\by}{\mathbf{y}}
\newcommand{\bPhi}{\mathbf{\Phi}}
\newcommand{\PiGC}{\Pi_{\rm GC}}

\newcommand{\norm}[1]{ \left\| #1 \right\|}

\renewcommand{\geq}{\geqslant}
\renewcommand{\leq}{\leqslant}

\renewcommand{\tilde}{\widetilde}
\newcommand{\eps}{\varepsilon}
\newcommand{\nn}{\nonumber}
\newcommand{\rd}{{\rm d}}

\title{Grand-Canonical Optimal Transport}

\author[S. Di Marino]{Simone Di Marino}
\address{Universit\`{a} di Genova, DIMA, MaLGa, Via Dodecaneso 35, 16146 Genova, Italy}
\email{simone.dimarino@unige.it}

\author[M. Lewin]{Mathieu Lewin}
\address{CNRS \& CEREMADE, Universit\'e Paris-Dauphine, PSL University, 75016 Paris, France}
\email{mathieu.lewin@math.cnrs.fr}

\author[L. Nenna]{Luca Nenna}
\address{Universit\'e Paris-Saclay, CNRS, Laboratoire de math\'ematiques d'Orsay, ParMA, Inria Saclay, 91405, Orsay, France. }
\email{luca.nenna@universite-paris-saclay.fr}

\date{\today}

\begin{document}

 \begin{abstract}
We study a generalization of the multi-marginal optimal transport problem, which has no fixed number of marginals $N$ and is inspired of statistical mechanics. It consists in optimizing a linear combination of the costs for all the possible $N$'s, while fixing a certain linear combination of the corresponding marginals.
\end{abstract}

 \maketitle

\setcounter{tocdepth}{1}
 \tableofcontents

\section{Introduction}

The theory of optimal transport plays an important role in many applications~\cite{Villani-03,Villani-09,Santambrogio-15,PeyCut-19}. Its generalization to the \emph{multi-marginal} case, where $N\geq3$ marginals are given instead of only two, has been particularly studied in the last years. The latter has applications in many areas including Economics~\cite{CarEke-10,ChiCanNes-10,Pass-14}, Financial Mathematics~\cite{BeiHenPen-13,DolSon-14,DolSon-14b,GalHenTou-14,HenTanTou-16,HenTou-16}, Statistics~\cite{BigKle-12_ppt,CarCheGal-16}, Image Processing~\cite{RabPeyBer-12}, Tomography~\cite{AbrAbrBerCar-17}, and Quantum Physics and Chemistry in the setting of the \emph{Strictly Correlated Electron} model in \emph{Density Functional Theory} (DFT)~\cite{Seidl-99,SeiPerLev-99,SeiGorSav-07,GorSei-10,ButPasGor-12,ColPasMar-15,CotFriKlu-13,CotFriPas-15,FriMenPasCotKlu-13,MenLin-13}.

In this paper, we study a further generalization where the number of marginals $N$ is not fixed. This model takes its roots in Statistical Physics, where it usually goes under the name \emph{Grand-Canonical}~\cite{Ruelle}. It has recently been used for Coulomb and Riesz costs where it naturally occurs in the large-$N$ limit of the multi-marginal problem~\cite{LewLieSei-18,CotPet-19,LewLieSei-20,LewLieSei-23_DFT}. A truncated version has been studied in \cite{BouButChaPas-21,BouButChaPas-21b}. Its entropic regularization is well known in the literature~\cite{ChaChaLie-84,ChaCha-84,JanKunTsa-22} and plays a central role in the density functional theory of inhomogeneous classical fluids at positive temperature~\cite{Evans-79,Evans-92}. An entropic grand canonical problem has also recently appeared in~\cite{BarLav-23}, where it is interpreted as a relative entropy minimization with respect to branching Brownian motion, in the framework of regularized unbalanced optimal transport. We will come back to the entropic model at the end of the paper in Section~\ref{sec:entropy}.

In short, the Grand-Canonical Optimal Transport (GC-OT) problem can be formulated in the Kantorovich form as follows
\begin{equation}
\cC(\rho):=\inf\left\{\sum_{n=0}^\ii \int_{\Omega^n}c_n\,{\rm d}\bP_n\  :\ \sum_{n=0}^\ii \bP_n(\Omega^n)=1,\quad \sum_{n=1}^\ii n\, \bP_n(\cdot,\Omega^{n-1})=\rho\right\}
\label{eq:GC-OT_intro}
\end{equation}
where each $\bP_n$ is a symmetric measure on $\Omega^n$, $c_n$ is the symmetric cost for the $n$-marginal problem and $\bP_n(\cdot,\Omega^{n-1})$ is the first marginal of $\bP_n$. Notice the factor $n$ multiplying the marginal in the constraint involving $\rho$, which accounts for the fact that there are $n$ equal such marginals since all the $\bP_n$ are symmetric. The family $\bP=(\bP_n)_{n\geq0}$ forms a probability which describes the behavior of some agents whose number is unknown or can vary. In this interpretation $\bP_n(\Omega^n)$ is the probability that there are $n$ agents and $\bP_0\in[0,1]$ is the one that there is no agent at all. In the GC-OT problem~\eqref{eq:GC-OT_intro} only the \emph{average} quantity $\rho$ is fixed and fluctuations of the number of agents are allowed. Solving the minimization problem $\cC(\rho)$ requires in particular to determine the best way to distribute the number of agents through the measures $\bP_n$, in order to reproduce the given average $\rho$, depending on the corresponding costs $\bc=(c_n)_{n\geq0}$.

The (symmetric) $N$-marginal problem corresponds to having $\rho(\Omega)=N\in\N$ and adding the constraint that all the $\bP_n$ are zero except $\bP_N$ (the usual optimal transport problem is for $N=2$). It is not true in general that this will be the solution to the grand-canonical problem~\eqref{eq:GC-OT_intro}. The optimum will largely depend on the choice of the costs $c_n$. If $\rho(\Omega)$ is not an integer, an optimum $\bP$ will necessarily involve at least two non-trivial $\bP_n$'s. One natural question is whether this is limited to $\bP_N$ and $\bP_{N+1}$ where $N$ is the unique integer so that $N<\rho(\Omega)<N+1$, or whether higher fluctuations are favorable.

Problems of the form~\eqref{eq:GC-OT_intro} appear for instance in the semi-classical approximation of Density Functional Theory~\cite{LewLieSei-23_DFT}, in which case $\Omega=\R^3$ and
\begin{equation}
c_0=c_1=0,\qquad c_n(x_1,...,x_n)=\sum_{1\leq j<k\leq n}\frac{1}{|x_j-x_k|}
 \label{eq:Coulomb_intro}
\end{equation}
is the pairwise Coulomb repulsion between $n$ electrons. The function $\rho$ is then interpreted as the total average density of electrons in the system. The integral $\int_{\R^3}\rho$ is the average number of electrons. Solving this problem is believed to give some information on the amount of correlations in the system, depending on the density $\rho$, which could then be used in the approximation of the more precise quantum problem.

Let us mention that the symmetric constraint on the measures $\bP_n$ does not prevent us from treating systems with different kinds of agents. If we want to transport a certain amount of croissants produced in bakeries to Parisian caf\'es as in~\cite[Chap.~3]{Villani-09}, then $\Omega$ will be a finite set containing the properties (precise location, size, etc) of both bakeries and  caf\'es. The symmetry just means that the croissants are all the same and we do not want to distinguish which one is sent where. This point of view is further discussed in Example~\ref{rmk:non-symmetric} below.

In this article, we discuss several mathematical properties of the GT-OT problem~\eqref{eq:GC-OT_intro}, within the framework of optimal transport theory.
The paper is organized as follows. In Section~\ref{sec:existence} we formulate the problem and show the existence of a minimizer $\bP=(\bP_n)_{n\geq0}$, under appropriate assumptions on the costs. In Section~\ref{sec:support} we focus on the case of pairwise costs and derive some properties on the \emph{support} of $\bP$, that is, on how many of the $\bP_n$'s are non zero. We particularly discuss the Coulomb case~\eqref{eq:Coulomb_intro}. We also study the truncated problem where all the $\bP_n$ are assumed to vanish after some $N_\text{max}$ and the convergence to the true problem when $N_\text{max}\to\ii$. This is useful for doing practical computations. Section~\ref{sec:duality} is devoted to duality theory and the existence of the dual potential. In Section~\ref{sec:1D} we study the 1D case, confirming thereby some predictions on the shape of the optimal plan made in \cite{MirSeiGor-13}. Finally, in Section~\ref{sec:entropy} we study the entropic regularization of the GC-OT problem.

\section{Main properties of Grand-Canonical Optimal Transport}\label{sec:existence}

\subsection{Subsystems are often grand-canonical}\label{sec:localization}

In order to motivate the problem, let us first explain why it is natural to let the number of marginals vary, even if we start with a system which has a well defined number of agents $N$. Consider a symmetric Borel probability measure $\mathscr{P}$ over $\Omega^N$, where $\Omega\subset\R^d$ is any given Borel set\footnote{We work in $\R^d$ to avoid general spaces and because this is the situation in the applications we have in mind. Everything holds the same in a more abstract setting.} and $N\geq2$ is typically a large number. The symmetry means here that
$$\mathscr{P}(A_{\sigma(1)}\times\cdots \times A_{\sigma(N)})=\mathscr{P}(A_1\times\cdots \times A_N)$$
for any permutation $\sigma\in\gS_N$ and any Borel sets $A_1,...,A_N\subset\Omega$. In applications, $\mathscr{P}$ describes $d$ properties of $N$ agents, taking their values in $\Omega$ (e.g. their space location along the $d$ coordinates in $\R^d$), with $\mathscr{P}(A_1\times\cdots \times A_N)$ being the probability that agent 1 is in $A_1$, agent 2 is in $A_2$, etc. The symmetry of $\mathscr{P}$ simply means that our agents are all identical and indistinguishable from one another.
When $N$ is very large it seems natural to allow it to vary a little, for instance to account for the fact that it can probably not be known exactly. However, grand-canonical states also occur naturally at fixed $N$ when we look at \emph{subsystems}. Let us explain this now.

Let us fix a subset $A\subset\Omega$. Imagine that we would like to ignore what is happening outside of $A$ and only concentrate on what the agents in $A$ are doing. Of course, although there are exactly $N$ agents in total, the number of agents in the subset $A$ can vary from 0 to $N$. A simple counting argument leads us to introducing the measures $\bP_n$ on $A^n$ defined by
\begin{equation}
\begin{cases}
\bP_0=\mathscr{P}\left((\Omega\setminus A)^N\right)&\text{for $n=0$,}\\
\bP_n(B_n)={N\choose n}\mathscr{P}\Big(B_n\times (\Omega\setminus A)^{N-n}\Big)&\text{for $1\leq n\leq N-1$,}\\
\bP_N(B_N)=\mathscr{P}\left(B_N\right)&\text{for $n=N$,}\\
0&\text{for $n\geq N+1$,}
       \end{cases}
 \label{eq:geometric}
\end{equation}
for any $B_n\subset A^n$. In other words, the conditional measure $\bP_n$ describes what $n$ agents are doing in $A$, in the situation that there are exactly $n$ agents in $A$ and $N-n$ outside. The combinatorial factor in the definition of $\bP_n$ is because in the expression it was assumed, by symmetry, that the $n$ first are in $A$ whereas the $N-n$ last ones are outside. Then $\bP_n(A^n)$ is the probability that there are exactly $n$ agents in $A$ and thus
$$\bP_0+\sum_{n=1}^N\bP_n(A^n)=1.$$
On the other hand, the average number of agents in $A$ is given by
$$\sum_{n=1}^Nn\,\bP_n(A^n)\in[0,N].$$
The corresponding density of agents in $A$ is the positive measure defined by
$$\rho_\bP(B):=\bP_1(B)+\sum_{n=2}^Nn\,\bP_n(B\times A^{n-1})$$
for every $B\subset A$. It is the sum of the marginals with the multiplication factor $n$ seen before. A calculation shows that
$$\rho_\bP(B):=N\mathscr{P}(B\times \Omega^{N-1}),$$
that is, $\rho_\bP$ is just the restriction of $N$ times the first marginal of $\mathscr{P}$ to the set $A$.
Our conclusion is that although the total number of agents $N$ can be fixed, it is never fixed as soon as we look at a subsystem, which is always represented by a grand-canonical probability $\bP$. Here we have $P_n\equiv0$ for $n\geq N+1$ because we start with $N$ agents in total. But if we let $N$ be arbitrarily large with $A$ fixed, then $\bP$ can in principle have infinitely many non-trivial $\bP_n$'s, as in~\eqref{eq:GC-OT_intro}.

The above construction is well-known in statistical mechanics, where the family $\bP=(\bP_0,...,\bP_N)$ is often called the \emph{localization of $\mathscr{P}$ to the set $A$}~\cite{Ruelle-67,RobRue-67,Ruelle,Lewin-11}. Note that there is definitely some loss of information when we look at this $\bP$ instead of the big probability $\mathscr{P}$ since everything which is happening outside of $A$ has been averaged over. Even if we would cover $\Omega$ with several $A$'s and look at the corresponding localized probabilities, we would in general not be able to reconstruct $\mathscr{P}$, since the correlation between the different domains is discarded. We believe that approximating a large multi-marginal problem by a collection of smaller local grand-canonical ones is a strategy which might be helpful in practice. This is in the spirit of embedding theories~\cite{SunCha-16,LinLin-21} used for large quantum systems.

\subsection{The Grand-Canonical Optimal Transport problem}

\subsubsection{Grand-canonical probabilities}
Let again $\Omega\subset \R^d$ be any Borel set. Consider any positive Borel measure $\rho$ such that $\rho(\Omega)<\ii$. The set of grand-canonical probabilities with density $\rho$ is denoted by
\begin{multline}
\PiGC(\rho):= \Big\{ \bP=(\bP_n)_{n\geq0}\ :\ \bP_0\in[0,1],\quad \bP_n \in\mathscr{M}_{\rm sym}(\Omega^n)\\
\bP_0+\sum_{n\geq1}\bP_n(\Omega^n)=1,\quad \rho_\bP=\rho\Big\}
\label{eq:GC_set}
\end{multline}
where $\rho_\bP$ is by definition the positive measure satisfying
$$\rho_\bP(B):=\bP_1(B)+\sum_{n\geq2}n\,\bP_n(B\times\Omega^{n-1}),$$
for every $B\subset\Omega$, that is, the sum of $n$ times the first marginal of the $\bP_n$'s. In~\eqref{eq:GC_set}, $\mathscr{M}_{\rm sym}(\Omega^n)$ denotes the set of finite positive symmetric Borel measures over $\Omega^n$. The following examples show that $\PiGC(\rho)$ is never empty.

\begin{example}[Usual multi-marginal probabilities]\label{ex:usual}
If $\rho(\Omega)=N\in\N$, then we can consider any $N$-agent symmetric probability $P_N$ of first marginal $\rho/N$, for instance
$$\bP_N=\left(\frac{\rho}{N}\right)^{\otimes N}$$
for independent agents, and take $\bP_n\equiv0$ for $n\neq N$. This is usually called a \emph{canonical probability} because the number of agents does not vary.

If $\rho(\Omega)=N+t$ with $N\in\N$ and $t\in(0,1)$ we need at least two non-vanishing $\bP_n$'s. We can for instance first write $\rho=(1-t)\rho_{N}+t\rho_{N+1}$ with $\rho_N(\Omega)=\rho_{N+1}(\Omega)-1=N$ and then take $\bP_N=t\bQ_N$ and $\bP_{N+1}=(1-t)\bQ_{N+1}$ with an obvious notation.
\end{example}

\begin{example}[Poisson states]\label{ex:Poisson}
For $\rho$ any positive measure over $\Omega$ with $\rho(\Omega)<\ii$ we define the associated \emph{Poisson grand-canonical probability} $\bG_\rho=(\bG_{\rho,n})_{n\geq0}\in\PiGC(\rho)$ by
$$\bG_{\rho,0}=e^{-\rho(\Omega)},\qquad \bG_{\rho,n}=e^{-\rho(\Omega)}\frac{\rho^{\otimes n}}{n!}. $$
This state has a Poisson distribution in the number $n$ of agents, which then behave independently for every $n$.
Such Poisson states can be obtained as the localization to $\Omega$ of $N$ i.i.d. agents in the limit $N \to \infty$. To see this, let us consider any point $x^*$ outside of $\Omega$ and the \emph{canonical} probability measure associated with $N\geq \rho(\Omega)$ independent agents distributed according to $\frac{\rho+(N-\rho(\Omega))\delta_{x^*}}N$, that is,
$$\mathscr{P}^N=\Bigl(\frac{\rho + (N-\rho(\Omega)) \delta_{x^*}}N \Bigr)^{\otimes N}.$$
If we localize this probability to $\Omega$ as we did in the previous section, we obtain the grand-canonical probability
$$ \bP^N_n=\begin{cases}
\Bigl( \frac{N-\rho(\Omega)}{N}\Bigr)^{N}&\text{for $n=0$,}\\
\frac{\rho^{\otimes n}}{N^n} \Bigl( \frac{N-\rho(\Omega)}{N}\Bigr)^{N-n} { {N} \choose {N-n}}&\text{for $1\leq n\leq N$,}\\
0&\text{for $n\geq N+1.$}
\end{cases}$$
In the limit $N\to\ii$, we have $\bP^N_n \to \bG_{\rho,n}$, in total variation norm.
\end{example}

Next we discuss the corresponding Grand-Canonical Optimal Transport (GC-OT) problem, where we minimize some total average cost at fixed density $\rho_\bP=\rho$. We assume for simplicity that we have a certain cost $c_n$ for each number of agent $n$, that is, different numbers $n$ are not coupled to one another. We therefore consider a family $\bc=(c_n)_{n\geq0}$ with $c_0\in\R$ and $c_n$ a symmetric lower semi-continuous bounded-below function over $\Omega^n$. The minimization problem reads
\begin{equation}
\boxed{\cC(\rho):=\inf_{\bP\in\PiGC(\rho)}\left\{c_0\bP_0+\sum_{n\geq1}\int_{\Omega^n}c_n\,{\rm d}\bP_n\right\}.}
\label{eq:GC_OT}
\end{equation}
In the following we will often use the shorthand notation
\begin{equation}
\bP(\bc):=c_0\bP_0+\sum_{n\geq1}\int_{\Omega^n}c_n\,{\rm d}\bP_n
\label{eq:average_cost}
\end{equation}
for the average cost of a grand-canonical probability $\bP$.

Although one can allow any possible costs $c_n$ for a mathematical exercise, in practice the $c_n$ are often related with one another. A large part of the paper will be devoted to costs describing \emph{pairwise interactions between agents}, where all the $c_n$ are expressed in terms of the two-agent cost $c_2$ only.

\begin{definition}[Pairwise costs]
A \emph{pairwise grand-canonical cost} takes the form
\begin{equation}
 c_0=c_1=0,\qquad c_n(x_1,...,x_n)=\sum_{1\leq j<k\leq n}c_2(x_j,x_k)\quad\text{for $n\geq3$.}
 \label{eq:pairwise}
\end{equation}
\end{definition}

Such costs are ubiquitous in applications, but many other possibilities can of course be examined. For instance, any fixed $k$-agent cost $c_k$ induces a family of costs $c_n=\sum_{1\leq j_1<\cdots <j_k\leq n}c_k(x_{j_1},...,x_{j_k})$ for $n\geq k+1$ in the same manner. On the other hand, a cost of the form
$c_n(x_1,...,x_n)=\sum_{j=1}^nc_1(x_j)$
is of no interest in our setting since then the total average cost equals
\begin{equation}
 \sum_{n\geq1}\int_{\Omega^n}c_n\;{\rm d}\bP_n=\int_\Omega c_1\,{\rm d}\rho_\bP=\int_\Omega c_1\,{\rm d}\rho
 \label{eq:one-particle_cost}
\end{equation}
for $\bP\in\PiGC(\rho)$ and there is nothing left to optimize over.

\begin{example}[Interacting classical particles]
For systems of interacting classical particles~\cite{Ruelle} $c_n$ is of the form~\eqref{eq:pairwise} with, usually, a translation-invariant $c_2(x,y)=w(x-y)$. The probability $\bP$ represents the spatial location of the particles in $\Omega\subset\R^d$, $w$ is their pairwise interaction potential and $\rho_\bP$ is their average spatial density. In a gas of neutral atoms typical interactions are strongly repulsive at the origin and attractive with a fast decay at infinity, like for the Lennard-Jones potential
\begin{equation}
c_2(x,y)=\frac{A}{|x-y|^a}-\frac{B}{|x-y|^b},\qquad A,B>0,\qquad a>b>d.
\label{eq:Lennard-Jones}
\end{equation}
In Coulomb gases the particles are charged like electrons and then
\begin{equation}
c_2(x,y)=\begin{cases}
|x-y|^{2-d}&\text{for $d\geq3$,}\\
-\log|x-y|&\text{for $d=2$,}\\
-|x-y|&\text{for $d=1$.}
\end{cases}
\label{eq:Coulomb}
\end{equation}
The corresponding grand-canonical problem was studied in~\cite{LewLieSei-18,LewLieSei-20}. Riesz gases form a larger family of interactions comprising the Coulomb cost, defined by
\begin{equation}
c_2(x,y)=\begin{cases}
|x-y|^{-s}&\text{for $s>0$,}\\
-\log|x-y|&\text{for $s=0$,}\\
-|x-y|^{|s|}&\text{for $s<0$.}
\end{cases}
\label{eq:Riesz}
\end{equation}
The GC-OT for Riesz costs is used in~\cite{LewLieSei-18,CotPet-19,CotPet-19b,Lewin-22}.
\end{example}

\begin{remark}[Sub-additivity]
For any pairwise cost in the form~\eqref{eq:pairwise}, we have the inequality
\begin{equation}
 \cC(\rho_1+\rho_2)\leq \cC(\rho_1)+\cC(\rho_2)+\iint_{\Omega\times\Omega}c_2(x,y)\,\rd \rho_1(x)\,\rd \rho_2(y)
 \label{eq:subadditive}
\end{equation}
which is shown in~\cite[Eq.~(3.2)]{LewLieSei-18}. In other words
$$\rho\mapsto \cC(\rho)-\frac12\iint_{\Omega\times\Omega}c_2(x,y)\rd \rho(x)\,\rd \rho(y)$$
is sub-additive (on the set of $\rho$'s for which the two terms make sense).
\end{remark}

\begin{example}[Triviality of the pairwise harmonic cost]\label{ex:harmonic2}
The pairwise harmonic cost
$$c_0=c_1=0,\qquad c_n(x_1,...,x_n)=\sum_{1\leq j<k\leq n}|x_j-x_k|^2$$
is not interesting in our context because of symmetry. We obtain $\cC(\rho)=0$ for any $\rho$, with possible optimizers given by
$$\bP_0=1-\frac{\rho(\Omega)}{N},\qquad \bP_N=\frac1N\int_{\Omega} (\delta_y)^{\otimes N}\,{\rm d}\rho(y),\qquad \rho(\Omega)\leq N\in\N$$
or any convex combination of those. The same holds for any non-negative pairwise cost $c_2$ which vanishes on the diagonal.
\end{example}

\begin{example}[Mapping two densities with different masses]\label{rmk:non-symmetric}
The grand-canonical formalism allows us to transport a density $\rho_1$ to another $\rho_2$, with possibly different masses. The idea is just to work in $\Omega\times\{1,2\}$ instead of $\Omega$ and to choose $\rho(x,\sigma)=\rho_\sigma(x)$. The interpretation is that we have two different populations with the members in each groups being indistinguishable from each other, as is often the case in applications. The number of agents of each type can vary when we transport the density $\rho_1$ of the first population onto the second density $\rho_2$. Each optimal measure $\bP_N(x_1,\sigma_1,...,x_N,\sigma_N)$ describes $N_1=\sum_{j=1}^N(2-\sigma_j)$ agents of the first type which are transported to $N_2=\sum_{j=1}^N(\sigma_j-1)$ of the second. This seems a natural model which we will further investigate in future work. Note that the harmonic cost is again trivial here, by the same argument as in Example~\ref{ex:harmonic2}. The entropic regularization of this model was already considered in~\cite[Sec.~5]{ChaChaLie-84}. For other approaches like the unbalanced optimal transport problem, see~\cite{KanRub-58,ChiPeySchVia-18,LieMieSav-18,LieMieSav-16}.
\end{example}

\subsubsection{Well-posedness and existence of optimizers}
The well-posedness of the grand-canonical problem~\eqref{eq:GC_OT} requires specific assumptions on $c_n$ to avoid a collapse due to the possibility of having infinitely many agents in the system.\footnote{The situation is much easier if we add the constraint (or know in advance) that $\bP_n\equiv0$ for $n$ larger than some $N$.} This is a well known problem mentioned for instance in~\cite{ChaChaLie-84}. Note also that the problem could be well posed for some well behaved density $\rho$ (e.g.~ with compact support) and not for other ones. Here we will look for assumptions which work for all $\rho$'s. The following is inspired of statistical mechanics~\cite{Ruelle}.

\begin{definition}[Stability]
Let $\Omega\subset\R^d$. We say that the family of symmetric costs $\bc=(c_n)_{n\geq0}$ is \emph{stable} whenever there exist two constants $A,B\geq0$ such that
$c_n\geq -A-Bn$
on $\Omega^n$ for all $n\geq0$.
\end{definition}

Stability implies that
\begin{equation}
\bP(\bc)= c_0\bP_0+\sum_{n\geq1}\int_{\Omega^n}c_n\,{\rm d}\bP_n\geq -A-B\rho_\bP(\Omega)
 \label{eq:stable_lower_bound}
\end{equation}
and hence
$\cC(\rho)\geq -A-B\rho(\Omega)>-\ii.$
This makes $\cC(\rho)$ a well defined minimization problem, which is manifestly convex in $\rho$. The set of finite measures $\rho$ such that $\cC(\rho)<\ii$ is also convex.

In the case of pairwise costs as in~\eqref{eq:pairwise}, understanding the condition of stability in terms of the generating cost $c_2$ is a famous problem in statistical mechanics. Stable systems include for instance the case of
\begin{itemize}
 \item positive pairwise costs ($c_2\geq0$)
 \item positive-definite translation-invariant pairwise costs (that is, $c_2(x,y)=w(x-y)$ with $\widehat{w}\geq0$ and $\widehat{w}$ continuous at the origin)
 \item any convex combination of these two.
\end{itemize}
For instance, the Lennard-Jones-type potentials~\eqref{eq:Lennard-Jones} are always stable thanks to the strong repulsion close to the origin and the sufficiently fast decay at infinity~\cite{Dobrushin-64,FisRue-66}. On the other hand a cost satisfying $c_2(0,0)<0$ is never stable. Indeed, when $c_2$ is a continuous function, stability for the corresponding family~\eqref{eq:pairwise} is actually equivalent to the property that
\begin{equation}
\iint_{\Omega\times\Omega}c_2(x,y)\,{\rm d}\rho(x)\,{\rm d}\rho(y)\geq0\quad\text{for every finite measure $\rho\geq0$.}
\label{eq:cond_stability}
\end{equation}
Taking $\rho$ a Dirac delta gives $c_2(0,0)\geq0$. To prove the equivalence, we choose such a $\rho$ and an $\ell>0$ and we plug the Poisson state from Example~\ref{ex:Poisson} with density $\ell\rho$ into~\eqref{eq:stable_lower_bound}. After a calculation we find
$$\frac{\ell^2}2\iint_{\Omega\times\Omega}c_2(x,y)\,{\rm d}\rho(x)\,{\rm d}\rho(y)\geq -A-B\ell\rho(\Omega)$$
and this gives~\eqref{eq:cond_stability} after taking $\ell\to\ii$. Conversely, starting from~\eqref{eq:cond_stability} we can obtain the condition in the definition with $A=0$ and $B=-c_2(0,0)/2$ after taking $\rho=\sum_{j=1}^n\delta_{x_j}$.

\begin{example}[Unstability of the repulsive pairwise harmonic cost]\label{ex:harmonic}
The repulsive pairwise harmonic cost corresponds to $c_2(x,y)=-|x-y|^2$ on $\Omega=\R^d$, that is, $c_0=c_1=0$ and
\begin{equation}
c_n(x_1,...,x_n)=-\sum_{1\leq j<k\leq n}|x_j-x_k|^2=-n\sum_{j=1}^n|x_j|^2+\left|\sum_{j=1}^nx_j\right|^2.
\label{eq:harmonic_cost}
\end{equation}
In the $n$-marginal case studied for instance in~\cite{GamSwi-98,MarGerNen-17}, the first term depends on the density and it can safely be ignored. The problem is then the same as taking the attractive harmonic cost $|\sum_{j=1}^nx_j|^2$ for the center of mass of the $x_j$ (see Section~\ref{sec:center_mass} for more about this cost). In the grand canonical case the factor $n$ in the first term on the right side of~\eqref{eq:harmonic_cost} has an average against $\bP$ which does \emph{not} depend only on $\rho_\bP$. Thus the problem is \emph{not} equivalent to the harmonic cost for the center of mass. In fact, the first term makes the system unstable and we have $\cC(\rho)=-\ii$ for any $\rho$ which is supported on two points or more (and $\cC(\rho)=0$ otherwise). Consider for instance the trial state $\bP$ defined by
\begin{equation}
\bP_0=1-\frac{\rho(\Omega)}{N},\quad \bP_N=\frac{\rho(\Omega)}{N}\,\left(\frac{\rho}{\rho(\Omega)}\right)^{\otimes N},\quad  \bP_n\equiv0\quad \forall n\notin\{0,N\}
\label{eq:stupid_bP}
\end{equation}
for $N\geq \rho(\Omega)$. A calculation shows that
$$\bP(\bc)=-\frac{N-1}{2\rho(\Omega)}\iint|x-y|^2{\rm d}\rho(x)\,{\rm d}\rho(y)\underset{N\to\ii}{\longrightarrow}-\ii.$$
\end{example}

\bigskip

Although stability is a good condition for $\cC(\rho)$ to be well defined for all finite measures $\rho$, it is not sufficient to obtain the existence of optimizers. The following rather artificial examples illustrate the kind of problems which can arise from large numbers of agents with small probabilities.

\begin{example}[No-agent cost]\label{ex:artificial}
Take a cost which only favors the case with no agent: $c_0=-1$ and $c_n\equiv0$ for $n\geq1$.
Then $\cC(\rho)=-1$ for every $\rho$ but it is never attained for $\rho\neq0$.
Indeed, the average cost equals $\bP(\bc)=-\bP_0\geq-1$. This is $>-1$ as soon as one $\bP_n$ is non-zero, which ought to be the case when $\rho\neq0$. To prove that $\cC(\rho)=-1$ we can use the grand-canonical probability $\bP$ introduced in~\eqref{eq:stupid_bP}. Then we have $\rho_\bP=\rho$ as required and the average cost is
$\bP(\bc)={\rho(\Omega)}/{N}-1\to-1.$
\end{example}

\begin{example}[Too small costs]\label{ex:artificial2}
A different example is when the costs are positive but not large enough for $n\gg1$, e.g., $c_0=0$ and $c_n>0$ with $\norm{c_n}_{L^\ii(\Omega^n)}=o(n)$.
Using the same probability $\bP$ as in~\eqref{eq:stupid_bP} we see that
$0< \cC(\rho)\leq \rho(\Omega)\|c_N\|_{\ii}/N\to 0.$
Hence $\cC(\rho)=0$ is never attained for $\rho\neq0$.
\end{example}

The previous example shows that the cost $c_n$ should be at least as large as $n$ for some configurations to hope to have minimizers. Note that for a pairwise cost as in~\eqref{eq:pairwise} with $c_2(0,0)>0$ then we have
$c_n(0,...,0)=\frac{n(n-1)}{2}c_2(0,0)$
which blows up like $n^2$. The following remedy is the adaptation of another classical concept in statistical mechanics~\cite{Ruelle-70}.

\begin{definition}[Super-stability]
We say that the family of costs $\bc=(c_n)_{n\geq0}$ is \emph{super-stable} if it is stable and if for any compact set $K\subset\R^d$, there exists $\eps_K>0$ and $n_K\in\N$ such that
\begin{equation}
 c_n(x_1,...,x_n)\geq -\frac{n}{\eps_K}+\eps_K\left(\sum_{j=1}^n\1_{\Omega\cap K}(x_j)\right)^2\quad\text{on $\Omega^n$ for all $n\geq n_K$.}
 \label{eq:superstable}
\end{equation}
\end{definition}

The condition~\eqref{eq:superstable} requires that the cost $c_n$ blows up quadratically in terms of the number of agents in any fixed domain $K$. Placing all the agents in $K$, then we deduce in particular that $\|c_n\|_{L^\ii}\geq \eps_K n^2-n/\eps_K\gg n$. The usual definition of statistical mechanics~\cite{Ruelle-70} uses a partition of the space into cubes $\R^d=\cup_{z\in\Z^d} C_z$ and the condition
$$c_n(x_1,...,x_n)\geq -Bn+\eps\sum_{z\in\Z^d}\left(\sum_{j=1}^n\1_{\Omega\cap C_z}(x_j)\right)^2,\qquad\forall n\geq1,$$
which is stronger when $\Omega$ is unbounded. This is a more global condition which provides a better uniform control. In some sense our definition~\eqref{eq:superstable} is more local since all the constants can depend on $K$. This is sufficient in our setting where the total density is anyway fixed, hence there will never be too many agents in average far away. Note also that the power 2 in~\eqref{eq:superstable} is for convenience. Any power strictly larger than 1 will do.

A two-agent cost satisfying $c_2(x,y)\geq c\1(|x-y|\leq \eps)$ with $\eps,c>0$ provides a superstable pairwise family $\bc$ through~\eqref{eq:pairwise}.
Hence when $c_2=c_2'+c_2''$ with $c_2'(x,y)\geq c\1(|x-y|\leq \eps)$ and $c_2''$ stable the corresponding $\bc$ is super-stable. This includes for instance Lennard-Jones potentials as in~\eqref{eq:Lennard-Jones} and Riesz costs~\eqref{eq:Riesz} in any dimension $d\geq3$, for $s>0$ and $\Omega=\R^d$.


\begin{theorem}[Existence of optimizers]\label{thm:existence}
Let $\Omega\subset\R^d$ be any Borel set. Let $\bc=(c_n)_{n \geq 0} $ be a superstable family of lower semi-continuous costs. Then any finite $\cC(\rho)$ admits a minimizer $\bP^*$. Moreover $\rho\mapsto \cC(\rho)$ is convex and lower semi-continuous for the tight convergence of measures.\end{theorem}


\begin{proof}
After extending all our probabilities by zero we may assume that $\Omega=\R^d$.
Let $\bP^k= ( \bP^k_n )_{n \geq 0}$ be a minimizing sequence. Then, since $\rho_{\bP^k_n} \leq \rho$, we see that for every $n \geq0 $ the sequence $\{\bP^k_n\}_{k \geq 1}$ is tight. By a diagonal argument, we can assume after extracting a subsequence that $\bP^k_n \stackrel{*}{\rightharpoonup} \bP^*_n$ narrowly (that is, in duality with bounded continuous functions) for every $n$. 

First we show that $\bP^* \in \PiGC(\rho)$. We set $\lambda^k_n:= \bP^k_n( \R^{dn})$ and $\lambda^*_n:=\bP^*_n( \R^{dn})$. Since $\bP^k_n \in \Pi_{GC}(\rho)$, we have $ \sum_{n = 0}^{\infty} \lambda^k_n = 1$ and  $\sum_{n = 0}^{\infty} n\lambda^k_n = \rho(\R^d)$.
Since $\bP^k_n \stackrel{*}{\rightharpoonup} \bP^*_n$ narrowly, we have $\lambda^k_n \to \lambda^*_n$. The weights $n$ imply that $(\lambda^k_n)_n$ converges strongly in $\ell^1$, and hence $ \sum_{n = 0}^{\infty} \lambda^*_n = 1$.
Thus $\bP^*$ is a grand-canonical probability and it only remains to show that $\rho_{\bP^*}=\rho$. This is where the super-stability is useful. We find $\rho_{\bP^*}\leq\rho$ by passing to the weak limit but have to show equality (think again of our example~\eqref{eq:stupid_bP} which has the limiting $\bP^*=(1,0,...)$ and $\rho_{\bP^*}=0$). We estimate the tail of the sum using the superstability of $(c_n)_{n \geq 0}$ and the finiteness of the average cost. We fix a compact set $K$ and we know that for $k$ large enough and $M\geq n_K$
\begin{align*}
2 \cC(\rho)  \geq \bP^k(\bc) & \geq -A-C_K\rho(\R^d) + \eps_K\sum_{n \geq M} \int \left(\sum_{j=1}^n\1_{K\cap\Omega}(x_j)\right)^2 \, {\rm d} \bP_n^k  \\
  & \geq -A-C_K\rho(\R^d) + \eps_K\sum_{\substack{n \geq M\\ \bP^k_n\neq0}}  \frac1{\lambda_n^k}\left(\int\bigg(\sum_{j=1}^n\1_{K\cap\Omega}(x_j)\bigg) \, {\rm d} \bP_n^k \right)^2 \\
 & =-A-C_K\rho(\R^d) + \eps_K\sum_{\substack{n \geq M\\ \bP^k_n\neq0}}  \frac{\rho_{\bP_n^k}(K)^2}{\lambda_n^k} \\
 & \geq -A-C_K\rho(\R^d) + \eps_K\frac{ \left(\sum_{n \geq M} \rho_{\bP_n^k}(K) \right)^2 }{\sum_{n \geq M} \lambda_n^k},
\end{align*}
with $C_K=\max(B,1/\eps_K)$. In the second line we have used Jensen's inequality. Using now the inequality $\sum_{n \geq M} \lambda_n^k \leq \rho (\R^d)/ M$ this proves that
\begin{equation}\label{eqn:estiss}
\sum_{n \geq M} \rho_{\bP_n^k}(K) \leq  \frac{C'_K}{\sqrt{M}}
\end{equation}
for some constant $C'_K$. Thus we obtain from the tightness of each of the $\rho_{\bP^k_n}$
\begin{align*}
\rho_{\bP^*} (K) \geq \sum_{n=0}^{M-1} \rho_{\bP^*_n} (K) &= \limsup_{k\to \infty} \sum_{n=0}^{M-1} \rho_{\bP^k_n} (K)  \\
& =\limsup_{k\to \infty} \left(\rho(K) - \sum_{n\geq M} \rho_{\bP^k_n} (K)\right) \geq \rho(K) - \frac{C'_K}{\sqrt{M}}.
\end{align*}
Letting $M \to \infty$ and using $\rho_{\bP^*}\leq\rho$, we obtain that $\rho_{\bP^*}$ and $\rho$ agree on $K$. Since this holds for every $K$ they are equal everywhere.

Next we show that the energy of $\bP^*$ is optimal. It is sufficient to prove that $\bP \mapsto \bP(\bc)$ is lower semicontinuous over $\PiGC(\rho)$, for the weak convergence used above in the proof. Since $(c_n)_{n \geq 0}$ is stable we can consider $\tilde{c}_n = c_n +n B+A\geq0$ which is also l.s.c. Noticing that $\bP(\widetilde{\bc})=\bP(\bc)+B\rho(\R^d)+A$ for all $\bP \in \PiGC(\rho)$,
we see that it suffices to prove the lower semi-continuity for $\widetilde{\bc}$. But, for each $n$, $\bP_n\mapsto \bP_n(\widetilde{c}_n)$  is lower semi-continuous with respect to the narrow convergence and the sum is also lower semi-continuous by Fatou's lemma in $\ell^1$. The lower semi-continuity of $\cC(\rho)$ for the narrow convergence of densities is proved by following step by step the previous arguments.
\end{proof}

In general $\cC(\rho)$ is \emph{not} lower semi-continuous for the weak convergence of measures (without the tightness condition).

\begin{example}[Non weakly lower semi-continuous]
Consider a pairwise cost of the form~\eqref{eq:pairwise} with $c_2(x,y)=A|x-y|^{-a}-B|x-y|^{-b}$ a Lennard-Jones potential with $a>b>d$ and $A,B>0$. Let $\rho_1\in C^\ii_c(\R^d,\R_+)$ and $\rho_2=\delta_{R_1}+\delta_{R_2}$ with $R_1$ and $R_2$ chosen so that $c_2(R_1,R_2)<0$. Then $\cC(\rho_2)\leq c_2(R_1,R_2)<0$ and one can even prove that $\cC(\rho_2)= c_2(R_1,R_2)$. Let $\rho_n:=\rho_1+\rho_2(\cdot-n\tau)$ with $\tau\neq0$ a fixed vector. Then $\rho_n\wto\rho_1$ vaguely (that is, in duality with continuous functions tending to 0 at infinity) but it is not tight. By~\eqref{eq:subadditive} we have
$\limsup_{n\to\ii}\cC(\rho_n)\leq \cC(\rho_1)+\cC(\rho_2)<\cC(\rho_1)$
(the first inequality is actually an equality). Hence $\cC$ is not wlsc for the weak topology in duality with $C_0(\R^d)$.
\end{example}

Lower semi-continuity for non-tight sequences holds under the additional condition that the cost is increasing with $n$, in an appropriate sense.

\begin{theorem}[Weak lower semi-continuity]\label{thm:wlsc}
Let $\Omega\subset\R^d$ be a Borel set. Let $\bc=(c_n)_{n \geq 0} $ be a superstable family of lower semi-continuous costs, with
\begin{equation}
c_n(x_1,...,x_n)\geq c_{n-1}(x_1,...,x_{n-1}),\qquad\forall x_1,...,x_n\in\Omega
\label{eq:condition_wlsc}
\end{equation}
for every $n\geq1$.
Then $\rho\mapsto \cC(\rho)$ is weakly lower semi-continuous for the convergence in duality with $C_0(\R^d)$.
\end{theorem}

By symmetry, any other set of $n-1$ points among the $x_j$ can be chosen on the right of~\eqref{eq:condition_wlsc}. The condition means that the cost always increases with the number of agents in the system. For a pairwise cost as in~\eqref{eq:pairwise} this is satisfied under the simple condition that $c_2\geq0$, since for a lower bound we can simply neglect all the $c_2(x_j,x_n)$.

\begin{proof}
Let $\rho^k\stackrel{*}{\wto}\rho^*$ be a sequence which converges vaguely but not necessarily narrowly, so that $\cC(\rho^k)$ is bounded. We follow~\cite{Lewin-11} and rewrite the weak convergence into the two successive narrow convergences
$$\rho^k\1_{B_R}\underset{k\to\ii}{\wto}\rho^*\1_{B_R},\qquad\qquad \rho^*\1_{B_R}\underset{R\to\ii}{\wto}\rho^*$$
by localizing first the density to a finite ball $B_R$ of radius $R$, taking $k\to\ii$ and only at the end taking $R\to\ii$. Next we claim that
\begin{equation}
 \cC\left(\rho\1_{B_R}\right)\leq \cC\left(\rho\right)
 \label{eq:increasing_energy_under_loc}
\end{equation}
for any $\rho$, which intuitively says that the cost for the agents in $B_R$ is lower than that of the full system. This will just follow from the condition~\eqref{eq:condition_wlsc}, as we will see. Admitting~\eqref{eq:increasing_energy_under_loc} we have for every $R$
$$\liminf_{k\to\ii}\cC\left(\rho^k\right)\geq \liminf_{k\to\ii}\cC\left(\rho^k\1_{B_R}\right)\geq \cC\left(\rho^*\1_{B_R}\right)$$
by Theorem~\ref{thm:existence} and the tightness of $(\rho^k\1_{B_R})_k$. Taking $R\to\ii$ and using now the tightness of $\rho^*\1_{B_R}$ gives the stated lower semi-continuity.

In order to prove~\eqref{eq:increasing_energy_under_loc} we have to go back to grand-canonical probabilities. The localization into $B_R$ naturally brings in the concept of subsystems introduced in Section~\ref{sec:localization}. Consider any grand-canonical probability $\bP\in\PiGC(\rho)$ with $\bP(\bc)<\ii$.
For any fixed ball $B_R$ we define a new grand-canonical probability $\bP_{|B_R}$ on $B_R\cap\Omega$ using the formulas
\begin{align*}
\bP_{|B_R,0}=&\bP_0+\sum_{\ell\geq 1}\bP_\ell\left((\Omega\setminus B_R)^{\ell}\right),\\
\bP_{|B_R,n}(A_1\times\cdots\times  A_n)=&\bP_n(A_1\times\cdots\times  A_n)\\&+\sum_{\ell\geq n+1}{\ell\choose n}\bP_\ell\left(A_1\times\cdots\times  A_n\times (\Omega\setminus B_R)^{\ell-n}\right)
\end{align*}
for every $A_1,...,A_n\subset B_R\cap\Omega$ and $n\geq1$. This is just the extension to grand-canonical probabilities of the construction in~\eqref{eq:geometric} for each $\bP_n$, by linearity. The interpretation is that $\bP_{|B_R}$ describes the subsystem consisting of all the agents in $B_R$. By linearity of the construction, one has again $\rho_{\bP_{|B_R}}=\rho_{\bP}\1_{B_R}$, that is, the density of the localization is just the restriction of the total density. Next we look at the cost and write
$$\int_{\Omega^n} c_n{\rm d}\bP_n=\sum_{m=0}^n{n\choose m}\int_{(B_R)^m\times(\Omega\setminus B_R)^{n-m}} c_n\,{\rm d}\bP_n,$$
that is, we look at all the possible ways to split the $n$ agents between $B_R$ and $\Omega\setminus B_R$. The combinatorial factor is again by symmetry of $\bP_n$ and $c_n$. Using $c_n(x_1,...,x_n)\geq c_m(x_1,...,x_\ell)$ by~\eqref{eq:condition_wlsc} and summing over $n$, we find
$\bP(\bc)\geq \bP_{|B_R}(\bc)\geq \cC\left(\rho\1_{B_R}\right).$
We obtain~\eqref{eq:increasing_energy_under_loc} after optimizing over $\bP$.
\end{proof}

\begin{remark}[Monge grand-canonical states]\label{rmk:Monge}
In optimal transport, the concept of \emph{Monge states} plays an important role~\cite{Villani-09}. A grand-canonical probability $\bP=(\bP_n)_{n\geq0}$ is called a \emph{Monge grand-canonical probability} whenever all the $\bP_n$ are Monge for $n\geq2$. This means that there exists a transport map $T_n:\Omega\to\Omega$ with $(T_n)^{\circ n}={\rm Id}$ such that $\bP_n=({\rm Id},T_n,\cdots,T_n^{\circ(n-1)})_\#\rho_{\bP_n}/n$. It is well known that for $N\geq3$, the $N$-marginal optimal transport problem does not always admit Monge minimizers~\cite{ColStr-16,SeiMarGerNenGieGor-17,MarGerNen-17,FriVog-18}, on the contrary to the classical $N=2$ case. We expect the same for $\cC(\rho)$ in the grand-canonical case but will not study this question further in this article. We will see in Section~\ref{sec:1D} that Grand-Canonical Monge states are optimal in dimension $d=1$ for a convex cost.
\end{remark}

\subsubsection{Relation with the multi-marginal problem}
The $n$-marginal problem can be written in the form
$$\cC_n(\rho)=\inf_{(0,...,0,\bP_n,0,...)\in\PiGC(\rho)}\int_{\Omega^n}c_n\,{\rm d}\bP_n$$
when $\rho(\Omega)=n\in\N$. The following result inspired of~\cite{LewLieSei-23_DFT} states that $\cC(\rho)$ is the convex hull of the $\cC_n$, if we decompose $\rho$ into any possible convex combination of densities with integer masses. It follows from the existence in Theorem~\ref{thm:existence}.

\begin{corollary}[Convex hull]
Let $\rho$ be a positive measure with $\rho(\Omega)<\ii$ such that $\cC(\rho)<\ii$. Then, under the same assumptions as in Theorem~\ref{thm:existence}, we have
\begin{equation}
 \cC(\rho)=\min_{\substack{\rho=\sum_{n\geq1}\alpha_n\rho_n\\ \rho_n(\Omega)=n\\ \sum_{n\geq0}\alpha_n=1}}\;\sum_{n\geq0}\alpha_n\,\cC_n(\rho_n).
 \label{eq:convex_hull}
\end{equation}
\end{corollary}

\begin{proof}
We first prove the equality with an infimum instead of a minimum on the right side of~\eqref{eq:convex_hull}.
Consider some weights $\alpha_n$ and densities $\rho_n$ as in the infimum. If for some $\alpha_n>0$ we have $\cC_n(\rho_n)=+\ii$ then there is nothing to prove. Hence we may assume that $\cC_n(\rho_n)<\ii$ for all $\alpha_n>0$ and take a corresponding minimizer $\bP_n$.
We then obtain the upper bound after introducing the grand-canonical probability $\bP=(\alpha_n\bP_n)_{n\geq0}$. For the lower bound we consider a minimizer $\bP^*=(\bP^*_n)_{n\geq0}$ for the grand-canonical problem $\cC(\rho)$ and obtain
$$\cC(\rho)=\bP^*(\bc)=\sum_{\bP_n^*\neq0}\bP_n^*(c_n)\geq \sum_{\bP_n^*\neq0}\bP_n^*(\Omega^n)\,\cC_n\left(\frac{\rho_{\bP^*_n}}{\bP^*_n(\Omega^n)}\right).$$
Letting $\alpha_n=\bP_n^*(\Omega^n)$ and $\rho_n=\rho_{\bP^*_n}/ \bP^*_n(\Omega^n)$ (when $\bP_n^*=0$ we take any $\rho_n$), we see that the right side is larger than the infimum in the statement. We obtain in addition that each non-vanishing $\bP_n^*$ has to be an optimizer of the $n$-marginal problem corresponding to its own density:
$$\frac{\bP^*_n(c_n)}{\bP^*_n(\Omega^n)}=\cC_n\left(\frac{\rho_{\bP^*_n}}{\bP^*_n(\Omega^n)}\right)$$
and thus the infimum is really a minimum.
\end{proof}

The fact that $\cC$ is a kind of convex hull let us suspect that it is also the weak closure of the $\cC_n$. The following theorem is an adaptation of a similar result in the quantum case in~\cite{LewLieSei-23_DFT} and it is somewhat the reciprocal to the discussion in Section~\ref{sec:localization}.

\begin{theorem}[Weak lower semi-continuous envelope]\label{thm:wlsc2}
Take $\Omega=\R^d$. Let $\bc=(c_n)_{n \geq 0} $ be a superstable family of lower semi-continuous costs, such that $c_0=0$ and
\begin{equation}
\lim_{\substack{\min_k|y_k|\to\ii\\\min_{k\neq \ell}|y_k-y_\ell|\to\ii}}c_{n+m}(x_1,...,x_n,y_1,...,y_m)= c_{n}(x_1,...,x_{n}),
\label{eq:condition_negligible}
\end{equation}
for every $n\geq1$. Let $\rho$ be any finite measure so that $\cC(\rho)<\ii$. Then there exists a sequence $\rho^k$ such that $N_k:=\rho^k(\R^d)\in \N$,
$$\rho^k\wto\rho\quad\text{locally and}\qquad \lim_{k\to\ii}\cC_{N_k}\big(\rho^k\big)=\cC(\rho).$$
\end{theorem}

The limit in~\eqref{eq:condition_negligible} is meant in the weak sense, that is,
\begin{multline*}
\lim_{\substack{\min_k|y_k|\to\ii\\\min_{k\neq \ell}|y_k-y_\ell|\to\ii}}\int_{\R^{dn}} c_{n+m}(x_1,...,x_n,y_1,...,y_m)\,{\rm d}\bP_n(x_1,...,x_n)\\= \int_{\R^{dn}}c_{n}(x_1,...,x_{n})\,{\rm d}\bP_n(x_1,...,x_n),
\end{multline*}
for every probability $\bP_n$ such that the right side is finite. Under the other condition~\eqref{eq:condition_wlsc}, the proof of Theorem~\ref{thm:wlsc} carries over to the \emph{local} weak convergence of measures. Thus if we add~\eqref{eq:condition_wlsc} to~\eqref{eq:condition_negligible}, we really obtain that $\cC$ is the closure of the $\cC_n$'s for the local weak convergence of measures.  The fact that we only have local convergence is due to the possible unboundedness of the number $N_k$ of agents in $\rho^k$. This is because we need infinitely many agents to be able to reproduce a $\bP=(\bP_n)_{n\geq0}$ with $\bP_n\neq0$ for arbitrarily large $n$. If we know that there exists a minimizer which satisfies $\bP_n\equiv0$ for $n$ large, then the local convergence can be replaced by weak convergence in duality with $C_0(\R^d)$.

\begin{proof}
Let $\bP=(\bP_n)_{n\geq0}$ be a minimizer for $\cC(\rho)<\ii$. Then
$\sum_{n\geq0}\bP_n(\R^{dn})=1$ and $\sum_{n\geq0}n\,\bP_n(\R^{dn})=\rho(\R^{d})<\ii$.
In addition, the stability gives $c_n+Bn+A\geq0$ and
$$\sum_{n\geq0}\bP_n(c_n+Bn+A)=\cC(\rho)+B\rho(\R^{d})+A<\ii.$$
This proves that $\sum_{n\geq0}|\bP_n(c_n)|<\ii$. Our first task will be to cut-off the large number of agents in $\bP$. We define
$$\bP^M_n:=\begin{cases}
\bP_0+\sum_{n\geq M+1}\bP_n(\R^{dn})&\text{for $n=0$,}\\
\bP_n&\text{for $1\leq n\leq M$,}\\
0&\text{for $n\geq M+1$.}
\end{cases}
$$
Then $\rho_{\bP^M}=\sum_{n=1}^M\rho_{\bP_n}\leq \rho$ converges narrowly to $\rho$ and
$$\bP^M(\bc)=\cC(\rho)+c_0\sum_{n\geq M+1}\bP_n(\R^{dn})-\sum_{n\geq M+1}\bP_n(c_n)\underset{M\to\ii}\longrightarrow\cC(\rho).$$
From the narrow convergence we have by Theorem~\ref{thm:existence}
$\lim_{M\to\ii}\bP^M(\bc)\geq \liminf_{M\to\ii}\cC(\rho_{\bP^M})\geq \cC(\rho)$
and this proves that
$$\lim_{M\to\ii}\bP^M(\bc)=\lim_{M\to\ii}\cC(\rho_{\bP^M})= \cC(\rho).$$
Thus we may replace $\bP$  by $\bP^M$, which only yields a small error in the density and in the average cost.

Next we take $M$ points $R_1,...,R_M\in\R^d$ and consider the $M$-particle probability
$$\bQ_R:={\rm Sym.}\bigg(\bP^M_0\delta_{R_1}\otimes\cdots \otimes\delta_{R_M}+\bP_1\otimes\delta_{R_2}\otimes\cdots \otimes\delta_{R_1}+\cdots +\bP_M\bigg)$$
symmetrized in the usual way. The density of $\bQ_R$ equals
$$\rho_{\bQ_R}=\rho_{\bP^M}+\bP^M_0\sum_{m=1}^M\delta_{R_j}+\bP_1(\R^d)\sum_{m=2}^M\delta_{R_j}+\cdots +\bP_{M-1}\big(\R^{d(M-1)}\big)\delta_{R_M}$$
and it converges to $\rho^M$ locally when $|R_j|\to\ii$ for all $j$. On the other hand, the $M$-marginal cost is
\begin{multline*}
\bQ_R(c_M)=\bP^M_0c_M(R_1,...,R_M)+\int_{\R^d}c_M(x_1,R_2,...,R_M)\,{\rm d}\bP_1(x_1)\\
+\int_{\R^{2d}}c_M(x_1,x_2,R_3...,R_M)\,{\rm d}\bP_2(x_1,x_2)+\cdots  +\int_{\R^{Md}}c_M\,{\rm d}\bP_M.
\end{multline*}
and it converges to $\bP^M(\bc)$ when $|R_j|\to\ii$ in such a way that $|R_j-R_k|\to\ii$ for $j\neq k$, due to our assumption~\eqref{eq:condition_negligible}. We have therefore
$$\cC(\rho_{\bP^M})\leq \liminf\cC_{M}(\rho_{\bQ_R})\leq \bQ_R(c_M)\longrightarrow \bP^M(\bc).$$
Thus for each $M$ we can find some positions $R_1^M,...,R_M^M$ such that
$$|R_j^M|\geq M,\qquad \forall j=1,...,M$$
$$\left|\cC_M(\rho_{\bQ_{R^M}})-\cC(\rho)\right|\leq \frac{1}{M}+\bP^M(\bc)-\cC(\rho_{\bP^M})+\left|\cC(\rho_{\bP^M})-\cC(\rho)\right|.$$
The densities $\rho_{\bQ_{R^M}}$ make up the sought-after sequence.
\end{proof}

\section{Support for pairwise repulsive costs}\label{sec:support}

\subsection{Support and truncation}\label{sec:support_def}

When $\cC(\rho)$ has a minimizer $\bP$, a natural question is to ask how many agents are necessary to minimize the given grand-canonical cost $\bc$, that is, how many of the $\bP_n$'s are non zero. This is important for numerical purposes because if we know that not too many $\bP_n$'s are different from zero, we can then reduce the number of unknowns.

\begin{definition}[Support]
We call
$${\rm Supp}(\bP)=\left\{n\geq0\ :\ \bP_n\neq0\right\}$$
the \emph{support} in $n$ of a grand-canonical probability $\bP=(\bP_n)_{n\geq0}$ and we say that $\bP$ has a \emph{compact support} whenever ${\rm Supp}(\bP)$ is bounded.
\end{definition}

If minimizers for $\cC(\rho)$ are known to have a compact support, this has the consequence that one can rewrite the grand-canonical optimal transport problem as a usual multi-marginal problem, at the expense of adding one variable. We explain this now. Assume for instance that
$$\cC(\rho)=\inf_{\substack{\bP\in\PiGC(\rho)\\ {\rm supp}(\bP)\subset [0,N]}}\bP(\bc)$$
for some $N$ which might depend on $\rho$. Let us then introduce the new cost on $\widetilde\Omega^N$ with $\widetilde\Omega:=\Omega\times\{0,1\}$ defined by
$$\widetilde{c}(x_1,\sigma_1,...,x_N,\sigma_N):=c_{n}(x_{i_1},...,x_{i_n})$$
where $n=\sum_{j=1}^N\sigma_j$ and $\{i_1<\cdots <i_n\}=\{i\ :\ \sigma_i=1\}$,
that is, we only retain the $x_i$ for which $\sigma_i=1$. What we are doing here is to associate to each agent a new variable $\sigma_i$ which determines whether the agent is in the system ($\sigma_i=1$) or not ($\sigma_i=0$). Then, a symmetric probability $\widetilde\bP_N$ exactly corresponds to a grand-canonical probability $\bP=(\bP_n)_{n\leq N}$ of support in $[0,N]$ through the relations
\begin{align*}
\bP_n(A_1\times\cdots\times A_n)=&{N\choose n}\widetilde\bP_N\Big((A_1\times\{1\})\times\cdots\times(\Omega\times\{0\})^{N-n}\Big),\\
\bP_0=&\widetilde\bP_N\Big((\Omega\times\{0\})^N\Big),\\
\bP_N(A_1\times\cdots \times A_N)=&\widetilde\bP_N\Big((A_1\times\{1\})\times\cdots\times (A_N\times\{1\})\Big).
\end{align*}
In other words, the grand-canonical $\bP$ is the \emph{localization to $\Omega\times\{1\}$} of the canonical $\widetilde\bP_N$, as defined in Section~\ref{sec:localization}. The total cost coincides with the grand-canonical cost
$$\int_{(\Omega\times\{0,1\})^N}\widetilde c_N\,{\rm d}\widetilde\bP_N=\sum_{n=0}^N\int_{\Omega^n}c_n\,{\rm d}\bP_n$$
and the density equals
$\rho_{\bP}(x)=\rho_{\widetilde\bP_N}(x,1)$.
We see that the grand-canonical problem can always be rewritten as a symmetric $N$-marginal problem, with the difference that not all the first marginal is fixed. Only its projection to $\Omega\times\{1\}$ is given.

\begin{remark}
For a pairwise cost as in~\eqref{eq:pairwise} we can rewrite the new $N$-particle cost in the simple form
$$\widetilde{c}(x_1,\sigma_1,...,x_N,\sigma_N):=\sum_{1\leq j<k\leq N}\sigma_j\sigma_k\,c_2(x_j,x_k).$$
\end{remark}

In numerical simulations, we of course always have to truncate the support of the multi-plan $\bP$ to some $N$ and thus can rewrite the problem as the above $N$-marginal optimization. Knowing the size of the support of exact minimizers is important to suppress or diminish the approximation due to the truncation. Good quantitative estimates are then useful and they will often depend on $\rho$.

\begin{definition}[Truncated Grand-Canonical problem]
The truncated Grand-Canonical problem is defined by
\begin{equation}
\label{pb:truncated}
\cC^{\leq N}(\rho):=\inf_{\substack{\bP\in\PiGC(\rho)\\ {\rm supp}(\bP)\subset[0,N]}}\bP(\bc).
\end{equation}
\end{definition}

This is the same as replacing $c_n$ by $+\ii$ for every $n\geq N+1$. The truncated problem~\eqref{pb:truncated} was studied in \cite{BouButChaPas-21} where it was obtained as the weak closure of $\cC_N$ in a similar spirit as Theorems~\ref{thm:wlsc} and~\ref{thm:wlsc2}. The following provides the convergence of $\cC^{\leq N}(\rho)$ towards $\cC(\rho)$ when $N\to\ii$.

\begin{theorem}[Convergence of the truncated problem]\label{thm:truncation}
Let $\Omega\subset\R^d$ be any Borel set. Let $\bc=(c_n)_{n \geq 0} $ be a stable family of lower semi-continuous costs with $c_1\in L^1(\Omega,\rd \rho)$. Let $\rho$ be a finite measure on $\Omega$ such that
\begin{equation}
\cC\big((1+\eps)\rho\big)<\ii\qquad\text{for some $\eps>0$.}
\label{eq:condition_rho_not_extreme}
\end{equation}
Then we have
\begin{equation}
 \lim_{N\to\ii}\cC^{\leq N}(\rho)=\cC(\rho).
 \label{eq:limit_truncation}
\end{equation}
If in addition $\bc=(c_n)_{n \geq 0} $ is superstable, then any associated sequence of optimizers $\bP^N=(\bP_0^N,...,\bP_N^N,0,...)$ for $\cC^{\leq N}(\rho)$ converges narrowly to a minimizer $\bP^*$ for $\cC(\rho)$, up to subsequences.
\end{theorem}

The second part of the result could also be rephrased in the setting of Gamma-convergence.

The main condition~\eqref{eq:condition_rho_not_extreme} used in the theorem is probably too strong, but it was chosen on account of its simplicity. We claim that it implies $\cC(\rho)<\ii$. Indeed, consider a $\bP=(\bP_n)_{n\geq0}\in \PiGC((1+\eps)\rho)$ such that $\bP(\bc)<\ii$. Then we can introduce the modified state $\bQ$ defined by
\begin{equation}
\begin{cases}
\bQ_0=\bP_0+\frac{1-\bP_0}{1+\eps}&\text{for } n=0,\\
\bQ_n=\frac{\bP_n}{1+\eps}&\text{for } n\geq1.
  \end{cases}
 \label{eq:decreasing_rho}
\end{equation}
It satisfies $\rho_\bQ=\rho$ and
$$\bQ(\bc)=\left(\bP_0+\frac{1-2\bP_0}{1+\eps}\right)c_0+\frac{\bP(\bc)}{1+\eps}<\ii,$$
thus proving the claim that $\cC(\rho)<\ii$.
The reader should interpret~\eqref{eq:condition_rho_not_extreme} as the assumption that $\rho$ must be in the interior of the convex set of densities for which $\cC(\rho)<\ii$. The condition is inspired of~\cite{ChaCha-84} and will re-appear several times in this work, in particular in Section~\ref{sec:entropy}. It could be replaced by the more complicated condition that $\rho=t_1\rho_1+t_2\rho_2$ for two arbitrary densities such that $\max(\rho_1(\Omega),\rho_2(\Omega))>\rho(\Omega)$, $t_1+t_2<1$ and $\cC(\rho_1),\cC(\rho_2)<\ii$.

By using the same construction~\eqref{eq:decreasing_rho} to increase the density instead of decreasing it, we can also see that if there exists $\bP\in \PiGC(\rho)$ such that $\bP(\bc)<\ii$ and $\bP_0\neq0$, then the condition~\eqref{eq:condition_rho_not_extreme} is automatically satisfied. Indeed, similarly as in~\eqref{eq:decreasing_rho} we can introduce the deformed state defined for $0\leq \eps\leq\bP_0/(1-\bP_0)$ by
\begin{equation}
\bP_{\eps,n}:=\begin{cases}
\bP_0-\eps(1-\bP_0)&\text{for $n=0$},\\
(1+\eps)\bP_n&\text{for $n\geq1$.}
\end{cases}
 \label{eq:deformed_Chayes}
\end{equation}
In other words we increase all the $\bP_n$ with $n\geq1$ and compensate by decreasing $\bP_0$.
We have
$$\bP_\eps(\bc)=-\eps\bP_0c_0+(1+\eps)\bP(\bc),$$
which is therefore finite. In addition, we have $\rho_{\bP_\eps}=(1+\eps)\rho$ and thus conclude that $\cC(\eta\rho)$ is finite for every $0\leq \eta<\frac{1}{1-\bP_0}$. This proves the claim that if there exists a $\bP\in\PiGC(\rho)$ such that $\bP_0\neq0$ and $\bP(\bc)<\ii$, then~\eqref{eq:condition_rho_not_extreme} is satisfied.

Our conclusion is that~\eqref{eq:condition_rho_not_extreme} is only an assumption for the case that all states $\bP\in\PiGC(\rho)$ of finite cost $\bP(\bc)<\ii$ must satisfy $\bP_0=0$. Note that for repulsive pairwise costs we will prove later in Lemma~\ref{lem:repulsive_N1} that $\bP^*_0$ always vanishes for a minimizer $\bP^*$ when $\rho(\Omega)\geq1$.

\begin{proof}
When the limit~\eqref{eq:limit_truncation} holds, any minimizer $\bP^{\leq N}$ for $\cC^{\leq N}(\rho)$ forms a minimizing sequence for $\cC(\rho)$. When $\bc$ is superstable, its convergence to a minimizer of $\cC(\rho)$ follows from the proof of Theorem~\ref{thm:existence}. We thus only have to discuss the validity of the limit~\eqref{eq:limit_truncation}.
Let $\bP\in\PiGC(\rho)$ be any state such that $\bP(\bc)<\ii$. We will show that $\limsup_{N\to\ii}\cC^{\leq N}(\rho)\leq\bP(\bc)$. Since it is clear from the definition that $\cC(\rho) \leq \cC^{\leq N}(\rho)$, we immediately obtain the claimed limit~\eqref{eq:limit_truncation} after optimizing over~$\bP$. We have here two possibilities.

\medskip

\noindent\textsl{Case 1: $\bP_0\neq0$.} We introduce the probability $\bP^{\leq N}$ defined by
$$\begin{cases}
\bP^{\leq N}_0=\bP_0-\sum_{n\geq N+1}\rho_{\bP_n}(\Omega)+\sum_{n\geq N+1}\bP_n(\Omega^n)\\
\bP^{\leq N}_1=\bP_1+\sum_{n\geq N+1}\rho_{\bP_n}\\
\bP^{\leq N}_n=\bP_n\qquad \text{for $n=2,...,N$,}\\
\bP^{\leq N}_n=0\qquad \text{for $n\geq N+1$,}
\end{cases}$$
which has density $\rho_{\bP^{\leq N}}=\rho$ as required. Note that
$$\sum_{n\geq N+1}\rho_{\bP_n}(\Omega)=\sum_{n\geq N+1}n\bP_n(\Omega^n)\geq (N+1)\sum_{n\geq N+1}\bP_n(\Omega^n)$$
so that $\bP^{\leq N}_0\leq \bP_0$. We thus need $\bP_0>0$ to make sure that $\bP^{\leq N}_0\geq0$ for $N$ large enough. Then we have
\begin{align*}
\cC^{\leq N}(\rho)\leq \bP^{\leq N}(\bc)&=\bP(\bc)-\sum_{n\geq N+1}\bP_n(c_n)+\sum_{n\geq N+1}\rho_{\bP_n}(c_1) \\
&\qquad +c_0\sum_{n\geq N+1}(1-n)\bP_n(\Omega^n)\\
&\leq \bP(\bc)+\sum_{n\geq N+1}(2A+(A+B)n)\bP_n(\Omega^n)+\sum_{n\geq N+1}\rho_{\bP_n}(c_1)
\end{align*}
which converges to $\bP(\bc)$ as claimed, since we have assumed that $c_1\in L^1(\Omega,\rd \rho)$, thus $\sum_{n\geq 1}|\rho_{\bP_n}(c_1)|<\ii$.

\medskip

\noindent\textsl{Case 2: $\bP_0=0$.} We have to first slightly modify $\bP$ by adding a small component in the vacuum before we apply the previous argument. Since we need to do this at fixed density without generating a too large error in the total cost, this is where the condition~\eqref{eq:condition_rho_not_extreme} becomes useful.
Let $\bQ\in\PiGC((1+\eps_0)\rho)$ with $\eps_0>0$ be such that $\bQ(\bc)<\ii$ and consider the state
$$\bQ_\eps:=(1-\eps)\bP+\frac{\eps}{1+\eps_0}\bQ+\frac{\eps\eps_0}{1+\eps_0}\delta_0$$
where $\delta_0=(1,0,...)$ is the vacuum. Then we have $\rho_{\bQ_\eps}=\rho$, $\bQ_{\eps,0}>0$ and
$$\bQ_\eps(\bc)=(1-\eps)\bP(\bc)+\frac{\eps}{1+\eps_0}\bQ(\bc)+\frac{\eps\eps_0 c_0}{1+\eps_0}\underset{\eps\to0}{\longrightarrow}\bP(\bc).$$
Truncating $\bQ_\eps$ as in Step 1 concludes the proof of the limit.
\end{proof}

\subsection{$\bc$--monotonicity}
Our main tool for establishing properties of the support will be the $\bc$-monotonicity of minimizers, which expresses the optimality with regard to displacements and variations of the number of agents at fixed total density.
For a vector $X=(x_1,...,x_N)\in(\R^d)^N$ and a set of indices $I=\{i_1<i_2<\cdots<i_K\}\subset\{1,...,N\}$ with $|I|:=K$ we denote $X_I:=(x_{i_1},...,x_{i_K})$. The following is an adaptation of a classical fact in multi-marginal optimal transport~\cite{Villani-09,Pass-12}.

\begin{lemma}[Grand-canonical $\bc$-monotonicity]\label{lem:cmonot}
Let $\Omega\subset\R^d$ be any Borel set. Let $\bc=(c_n)_{n \geq 0} $ be a family of lower semi-continuous costs. Assume that $\cC(\rho)<\ii$ admits a minimizer $\bP^*$ which satisfies $\bP_N\neq0$ and $\bP_K\neq0$ for some $N, K\geq0$. Then we have
\begin{equation}\label{eq:cmonot}
c_{|I|+|J|}(X_I,Y_J) + c_{N+K-|I|-|J|}(X_{I^c},Y_{J^c})\\
\geq c_N(X)+c_K(Y),
\end{equation}
for every $I\subset\{1,...,N\}$ and $J\subset\{1,...,K\}$, $\bP_N\otimes \bP_K$--almost everywhere on $\Omega^{N+K}$. In particular, there exist $X\in\Omega^N$ and $Y\in\Omega^K$ such that~\eqref{eq:cmonot} holds.
\end{lemma}

The inequality~\eqref{eq:cmonot} states that, on the support of $\bP_N\otimes\bP_K$, exchanging the position of $|I|$ agents in $\bP_N$ with that of $|J|$ agents in $\bP_K$ must not decrease the cost.

\begin{proof}
Consider any Borel sets $A_1,...,A_N,B_1,...,B_K\subset\Omega$ such that
$q_N:=\bP_N(A_1\times\cdots\times A_ N)>0$ and $r_K:=\bP_K(B_1\times\cdots\times B_K)>0$.
We prove that~\eqref{eq:cmonot} is valid when integrated against $\bP_N\otimes\bP_K$ on $A_1\times\cdots\times A_N\times B_1\times\cdots\times B_K$. We assume for simplicity of notation that $I=\{1,...,N'\}$ and $J=\{1,...,K'\}$. We define the two marginals by
$$\bQ_{N'}(C):=\bP_N(C\times A_{N'+1}\times\cdots \times A_N),\qquad \forall C\subset A_1\times\cdots A_{N'},$$
$$\bQ^c_{N-N'}(C):=\bP_N(A_1\times\cdots\times A_{N'}\times C),\qquad \forall C\subset A_{N'+1}\times\cdots A_{N}.$$
as well as $\bR_{K'}$ and $\bR^c_{K-K'}$ for $\bP_K$ by similar formulas. We then introduce
\begin{multline*}
\bP_n'=\bP_n+\\\begin{cases}
\dps-\eps\, r_K\,\bP_N\bigotimes_{j=1}^N\1_{A_j}&\text{$n=N$,}\\
\dps-\eps\, q_N\,\bP_K\bigotimes_{j=1}^K\1_{B_j}&\text{$n=K$,}\\
\dps\eps \left(\bQ_{N'}\bigotimes_{j=1}^{N'}\1_{A_j}\right)\otimes \left(\bR_{K'}\bigotimes_{j=1}^{K'}\1_{B_j}\right) &\text{$n=N'+K'$,}\\
\dps\eps \left(\bQ^c_{N-N'}\bigotimes_{j=N'+1}^{N}\1_{A_j}\right)\!\otimes\! \left(\bR^c_{K-K'}\bigotimes_{j=K'+1}^{K}\1_{B_j}\right) &\text{$n=N-N'+K-K'$.}
\end{cases}
\end{multline*}
This is a grand-canonical probability for $0<\eps<\min(q_N^{-1},r_K^{-1})$.
Note that $\bP'$ is not symmetric but it can be symmetrized at no cost since the $c_n$'s are symmetric. For a non-symmetric measure, the density is defined to be the sum of all its marginals. Noticing that
$$\rho_{\bP_N\bigotimes_{j=1}^N\1_{A_j}}=\rho_{\bQ_{N'}\bigotimes_{j=1}^{N'}\1_{A_j}}+\rho_{\bQ_{N-N'}\bigotimes_{j=N'+1}^N\1_{A_j}}$$
we obtain that $\rho_{\bP'}=\rho$, hence $\bP'\in\PiGC(\rho)$. The average cost equals
\begin{multline*}
\bP'(\bc)-\bP(\bc) = \eps\int_{A_1\times\cdots\times A_N}\int_{B_1\times\cdots \times B_K}\Big(-c_N(X)-c_K(Y)\\
+c_{N'+K'}(X_I,Y_J)+c_{N-N'+K-K'}(X_{I^c},Y_{J^c})\Big){\rm d}\bP_N(X)\,{\rm d}\bP_K(Y)
\end{multline*}
and it must be non-negative due to the optimality of $\bP$. This concludes the proof of the lemma.
\end{proof}

\subsection{Repulsive pairwise costs}
Now we discuss implications for pairwise costs as in~\eqref{eq:pairwise} which, we recall, take the form
\begin{equation}
c_0=c_1=0,\qquad c_n(x_1,...,x_n)=\sum_{1\leq j<k\leq n}c_2(x_j,x_k).
\label{eq:pairwise2}
\end{equation}
We assume throughout the section that $c_2$ is a \emph{strictly positive (repulsive) lower semi-continuous} function on $\Omega^2$. For any compact set $K\subset\R^d$, we have $\eps_K:=\min_{K\times K}c_2>0$ since $c_2$ attains its minimum. This gives
$$c_n(x_1,...,x_n)\geq \frac{\eps_K}2 N_K(N_K-1)\geq \frac{\eps_K}2N_K^2-\frac{\eps_K}2n,\qquad N_K:=\sum_{j=1}^n\1_K(x_j)$$
and proves that $\bc=(c_n)$ is super-stable. By Theorem~\ref{thm:existence} $\cC(\rho)$ admits a minimizer $\bP^*$ when it is finite.

In this section we discuss several possible conditions on $c_2>0$ which imply that the support of minimizers is always compact. There will be some overlap between these conditions but we have not found a general simple theory which covers everything. To be more precise, we will give a quantitative estimate on the support of minimizers
\begin{itemize}
 \item when $\Omega$ is a bounded domain,
 \item when $1/c_2$ satisfies a triangle-type inequality, which is for instance the case of Riesz interactions $c_2(x,y)=|x-y|^{-s}$ with $s>0$ in $\R^d$,
 \item when $c_2$ is asymptotically doubling.
\end{itemize}
In the Coulomb case $c_2(x,y)=|x-y|^{-1}$ we will show that the length of the support is controlled by $\sqrt{\rho(\Omega)}$ but it can in principle grow with $\rho(\Omega)$.

Before turning to the general case, we start by solving the GC-OT problem when there is only one agent or less in average, $\rho(\Omega)\leq1$. The support is then always in $\{0,1\}$. We also show that whenever there are more than one agent ($\rho(\Omega)>1$), the probability $\bP_0$ that there is no agent at all must vanish. Recall that the positivity of $\bP_0$ played a role through the condition~\eqref{eq:condition_rho_not_extreme} discussed after Theorem~\ref{thm:truncation}.

\begin{lemma}[The case of zero or one agent]\label{lem:repulsive_N1}
Let $\Omega\subset \R^d$ and $\bc=(c_n)_{n\geq0}$ be a pairwise cost  as in~\eqref{eq:pairwise2} with $c_2$ a \emph{strictly positive} lower semi-continuous function on $\Omega^2$.
\begin{itemize}
\item If $\rho(\Omega)\leq1$, then $\cC(\rho)=0$ with the unique minimizer
\begin{equation}
 \bP^*=\big(1-\rho(\Omega)\,,\, \rho\,,\,0\,,...\big).
 \label{eq:Nleq1}
\end{equation}
\item If $\rho(\Omega)>1$ and $\cC(\rho)<\ii$ then all the minimizers satisfy $\bP_0^*=0$.
\end{itemize}
\end{lemma}

\begin{proof}

If $\rho(\Omega)\leq1$, then we have $\cC(\rho)=0$ for the mentioned probability, which is the only possible one in $\PiGC(\rho)$ supported on $\{0,1\}$. Probabilities with $\bP_n\neq0$ for some $n\geq2$ all have a positive cost since $c_2>0$ by assumption.
If $\rho(\Omega)>1$ then a minimizer has at least one $\bP_N^*\neq0$ with $N\geq2$. But for a pairwise potential as in~\eqref{eq:pairwise2} with $c_2>0$, we have
$$c_N(X)=c_{|I|}(X_I) + c_{N-|I|}(X_{I^c})+\sum_{i\in I}\sum_{j\in I^c}c_2(x_i,x_j)>c_{|I|}(X_I) + c_{N-|I|}(X_{I^c}).$$
Thus~\eqref{eq:cmonot} cannot hold for $N\geq2$ and $K=0$ and we deduce that $\bP^*_0=0$.
\end{proof}

Next we discuss the interpretation of the $\bc$-monotonicity~\eqref{eq:cmonot} in spirit of charged particles. This will guide our analysis in the Coulomb case and allow us to adapt some tools developed in this context~\cite{Lieb-84,FraKilNam-16,FraNamBos-18}. For a pairwise cost as in~\eqref{eq:pairwise2}, the inequality~\eqref{eq:cmonot} can be rewritten in the form
\begin{multline}
 \sum_{i\in I}\sum_{k\in J}c_2(x_i,y_k)+\sum_{i\in I^c}\sum_{k\in J^c}c_2(x_i,y_k)\\
 \geq \sum_{i\in I}\sum_{j\in I^c}c_2(x_i,x_j)+\sum_{k\in J}\sum_{\ell\in J^c}c_2(y_k,y_\ell).
 \label{eq:cmonot_c2}
\end{multline}
Let us introduce a new cost for two kinds of agents (with positions $x_i$ and $y_k$) with a flipped sign for their mutual interaction
\begin{equation}
W_{N+K}(X;Y):=\sum_{1\leq i<j\leq N}c_2(x_i,x_j)+\sum_{1\leq k<\ell\leq K}c_2(y_k,y_\ell)-\sum_{i=1}^N\sum_{k=1}^Kc_2(x_i,y_k).
\label{eq:cost_charged}
\end{equation}
We should think here that the $x_i$ (resp. $y_k$) repel each other with the interaction $c_2$, but are attracted to the $y_k$ with the interaction $-c_2$. If $c_2$ is the Coulomb cost, then this is exactly the situation when we have $N$ electrons of charge $-1$ located at $x_1,...,x_N$ and $K$ nuclei of charge $+1$ located at $y_1,...,y_K$. After exchanging $J$ with $J^c$, the $\bc$-monotonicity condition~\eqref{eq:cmonot} can then be rewritten in the form
\begin{equation}
W_{N+K}(X;Y)\leq  W_{N'+K'}(X_I;Y_J)+W_{N-N'+K-K'}(X_{I^c};Y_{J^c})
\label{eq:cmonot_pairwise}
\end{equation}
with $N'=|I|$ and $K'=|J|$. This means that the cost should always increase when the system is split in two independent subsystems, on the support of an optimal multi-plan. In physical terms, $\bP_N\otimes \bP_K$ must always be supported on configurations where the two kinds of particles are bound together.

Our guiding principle in the rest of the section will be to contradict~\eqref{eq:cmonot_pairwise} for $N\gg K$. That is, we will show that for any $K\geq0$ there exists an $N(K)$ so that for any $N> N(K)$ and any $X=(x_1,...,x_N)\in\Omega^N$, $Y=(y_1,...,y_K)\in\Omega^K$,~\eqref{eq:cmonot_pairwise} fails for one choice of $I,J$. In this case, we deduce immediately that all the minimizers $\bP^*$ for $\cC(\rho)<\ii$ and $\rho(\Omega)>1$ satisfy
$${\rm supp}(\bP^*)\subset \left[1,\max_{0\leq K\leq \lfloor\rho(\Omega)\rfloor}N(K)\right].$$
This is because we know that $\bP_K\neq0$ for at least one $K\leq\lfloor\rho(\Omega)\rfloor$, by definition of $\rho$.

\subsubsection{Bounded repulsive cost}
One simple situation is when the two-agent cost is bounded both from above and below, for instance a continuous positive cost over a bounded domain.

\begin{theorem}[Support for bounded repulsive costs]\label{thm:support_bounded}
Let $\Omega\subset \R^d$ and $\bc=(c_n)_{n\geq0}$ be a pairwise cost  as in~\eqref{eq:pairwise2} with $c_2$ a lower semi-continuous function such that
$0<m\leq c_2\leq M<\ii$
on $\Omega^2$. Then $\cC(\rho)$ is finite for any non-negative measure $\rho$. When $\rho(\Omega)>1$, a minimizer $\bP^*$ satisfies
\begin{equation}
 {\rm supp}(\bP^*)\subset
\left[\frac{m}M \lfloor\rho(\Omega)\rfloor\;,\; 1+\frac{M}{m}(\lceil\rho(\Omega)\rceil-1)\right].
 \label{eq:support_bounded_cost}
\end{equation}
When $\rho(\Omega)\leq1$, we have $\cC(\rho)=0$ with the unique minimizer given by~\eqref{eq:Nleq1}.
\end{theorem}

\begin{proof}
Using a Poisson state as in Example~\ref{ex:Poisson} we find
$$0\leq \cC(\rho)\leq \frac12\iint_{\Omega\times\Omega}c_2(x,y)\,{\rm d}\rho(x)\,\,{\rm d}\rho(y)\leq \frac{M}2\rho(\Omega)^2.$$
Thus $\cC(\rho)$ is always finite and admits minimizers. From~\eqref{eq:cmonot_c2}, we have
\begin{align*}
 \sum_{i\in I}\sum_{k\in J}c_2(x_i,y_k)+\sum_{i\in I^c}\sum_{k\in J^c}c_2(x_i,y_k)&\leq M\big(N'K'+(N-N')(K-K')\big) \\
\sum_{i\in I}\sum_{j\in I^c}c_2(x_i,x_j)+\sum_{k\in J}\sum_{\ell\in J^c}c_2(y_k,y_\ell)&\geq m\big(N'(N-N')+K'(K-K')\big)
\end{align*}
with $|I|=N'$ and $|J|=K'$. Thus $\bP_N,\bP_K\neq0$ implies
$$N'(N-N')+K'(K-K')\leq \frac{M}{m}\big(N'K'+(N-N')(K-K')\big)$$
for all $N'\in\{0,...,N\}$ and $K'\in\{0,...,K\}$. To get a simple bound we can for instance choose $K'=K$ and $N'=1$, which gives
$$N\leq \frac{M}{m}K+1.$$
This shows that~\eqref{eq:cmonot_c2} can only hold for such $K$ and $N$.

Let $\bP^*$ be a minimizer for $\cC(\rho)$ with $\rho(\Omega)>1$ and, in the case that $\rho(\Omega)\in\N$, assume in addition that $\bP^*$ has a non trivial support (see Example~\ref{ex:usual}). Then we must have $\bP^*_K,\bP^*_N\neq0$ for some smallest $1\leq K\leq\lceil\rho(\Omega)\rceil-1$ and some $N\geq \lfloor\rho(\Omega)\rfloor+1$.
Recall that $K\geq1$ by Lemma~\ref{lem:repulsive_N1}. From the previous bound we have
$$\lfloor\rho(\Omega)\rfloor+1\leq N\leq \frac{M}{m}K+1\leq \frac{M}{m}(\lceil\rho(\Omega)\rceil-1)+1$$
and this gives the stated estimate.
\end{proof}

\begin{example}[Constant pairwise cost]
When $c_2$ is a positive constant,
\begin{equation*}
\bP(\bc)=\frac{c_2}{2}\sum_{n\geq0}n^2\bP_n(\Omega^n)-\frac{c_2}{2}\rho(\Omega)\geq \frac{c_2}{2}\rho(\Omega)(\rho(\Omega)-1),
\end{equation*}
for every $\bP\in\PiGC(\Omega)$, with equality if and only if $\bP$ is canonical, that is, $\bP_n\equiv0$ for all $n$ but one. Thus when $\rho(\Omega)=N$ is an integer, minimizers are all canonical probabilities. This follows also from Theorem~\ref{thm:support_bounded}. If $\rho(\Omega)=N+t$ is not an integer, the result tells us that the optimizers are exactly given by $\bP=(0,...,\bP_N,\bP_{N+1},0,...)$ with $\bP_N(\Omega^N)=1-t$ and $\bP_{N+1}(\Omega^{N+1})=t$ and, of course, $\rho_{\bP_N}+\rho_{\bP_{N+1}}=\rho$. Thus
$$\cC(\rho)=\frac{\rho(\Omega)(\rho(\Omega)-1)}{2}-\frac{t(t-1)}{2},\qquad t=\rho(\Omega)-\lfloor\rho(\Omega)\rfloor.$$
\end{example}

\subsubsection{Unbounded repulsive cost with a control on $1/c_2$}
Our next goal is to handle costs $c_2$ which are not necessarily bounded from above and below, but are still positive. The main point is to control the speed at which it goes to zero at infinity. In this section we use a kind of triangle inequality for $1/c_2$, up to a multiplicative constant whereas in Section~\ref{sec:asymp_doubling} we will use the different concept of ``asymptotically doubling'' functions. In Section~\ref{sec:Coulomb} we give better quantitative estimates for the special case of the Coulomb potential $c_2(x,y)=|x-y|^{-1}$ in any dimension.

\begin{theorem}[Support for unbounded pairwise positive costs]\label{thm:support_triangle}
Let $\Omega\subset \R^d$ and $\bc=(c_n)_{n\geq0}$ be a pairwise cost  as in~\eqref{eq:pairwise2} with $c_2$ a positive lower semi-continuous function such that $1/c_2$ satisfies the triangle-type inequality
\begin{equation}
 \frac1{c_2(x,y)}\leq Z\left(\frac1{c_2(x,z)}+\frac1{c_2(z,y)}\right),\qquad \forall x,y,z\in\Omega
 \label{eq:triangle}
\end{equation}
for some constant $Z>0$. Then for any non-negative measure $\rho$ such that $\cC(\rho)<\ii$ and $\rho(\Omega)>1$, a minimizer $\bP^*$ has a compact support in
\begin{equation}
{\rm Supp}(\bP^*)\subset\left[\frac{\lfloor\rho(\Omega)\rfloor+1}{2Z+1}\,,\,(2Z+1)(\lceil\rho(\Omega)\rceil-1)\right].
\label{eq:support_triangle}
\end{equation}
\end{theorem}

The result is inspired of a celebrated paper by Lieb~\cite{Lieb-84} on the maximum number of electrons that a molecule can bind. Theorem~\ref{thm:support_triangle} can be applied to the 3D Coulomb potential where $c_2(x,y)=|x-y|^{-1}$ in $d=3$, or more generally to Riesz costs where $c_2(x,y)=|x-y|^{-s}$, $s>0$ in any dimension $d$. We obtain the inequality~\eqref{eq:support_triangle} with $Z=\max\left(1,2^{s-1}\right)$.

On the other hand, for a cost satisfying $0<m\leq c_2\leq M<\ii$ as in Theorem~\ref{thm:support_bounded} the inequality~\eqref{eq:triangle} holds with $Z=M/(2m)$
and we obtain a bound slightly worse than~\eqref{eq:support_bounded_cost}.

\begin{proof}[Proof of Theorem~\ref{thm:support_triangle}]
Let $N,K\geq1$ and $X=(x_1,...,x_N)\in\Omega^N$, $Y=(y_1,...,y_K)\in\Omega^K$. Let us introduce the potential
$V(x)=\sum_{k=1}^{K} c_2(x,y_{k}).$
The $\bc$-monotonicity~\eqref{eq:cmonot} for $I=\{i\}$ and $J=\{1,...,K\}$ can be rewritten similarly to~\eqref{eq:cmonot_c2} in the form
$\sum_{j \neq i} c_2(x_i,x_j) \leq V(x_i).$
Dividing by $V(x_i)$ and summing over $i$, we obtain the relation
$$ \sum_{i=1}^{N} \frac 1{V(x_i)} \sum_{j \neq i} c_2(x_i,x_j)=\sum_{1\leq i < j\leq N} \frac{ V(x_i) + V(x_j)}{V(x_i)V(x_j)} c_2(x_i,x_j) \leq N.$$
We bound this term using the triangle inequality~\eqref{eq:triangle} as follows:
\begin{align}
& \sum_{1\leq i < j\leq N} \frac{ V(x_i) + V(x_j)}{V(x_i)V(x_j)}c_2(x_i,x_j)\nonumber\\
&\qquad  = \sum_{1\leq i < j\leq N}  \frac {c_2(x_i,x_j)}{V(x_i)V(x_j)} \sum_{k=1}^K c_2(x_i,y_{k})c_2(x_j,y_{k})\left(\frac{1}{c_2(x_i,y_{k})}+\frac1{c_2(y_{k},x_j)}\right) \nonumber\\
&\qquad  \geq \frac 1Z \sum_{k=1}^K  \sum_{1\leq i < j\leq N} \frac {c_2(x_i,y_{k})c_2(x_j,y_{k})} { V(x_i) V(x_j) } \label{eq:use_triangle}\\
&\qquad  = \frac 1{2Z}  \sum_{k=1}^K  \left[ \left( \sum_{i=1}^N \frac{c_2(x_i,y_{k})}{ V(x_i)} \right)^2 - \sum_{i=1}^N \frac{c_2(x_i,y_{k})^2}{V(x_i)^2} \right].\nonumber
\end{align}
By convexity we have
$$ \sum_{k=1}^K  \left( \sum_{i=1}^N \frac{c_2(x_i,y_{k})}{ V(x_i)} \right)^2 \geq \frac 1{K} \left(\sum_{k=1}^K \sum_{i=1}^N \frac{c_2(x_i,y_{k})}{ V(x_i)} \right)^2 = \frac {N^2}{K}.$$
On the other hand, we have $ V(x_i)^2 \geq \sum_{k=1}^K c_2(x_i,y_k)^{2}$ with a strict inequality when $K>1$, and thus
$$\sum_{k=1}^K \sum_{i=1}^N \frac {c_2(x_i,y_k)^2}{V(x_i)^2} \leq N. $$
So, we obtain
$$ N \geq \frac 1{2Z} \left( \frac {N^2}{K} - N \right),$$
that is, $N \leq (2Z+1)K$. When $K>1$, the inequality is strict. The rest of the proof is identical to that of Theorem~\ref{thm:support_bounded}.
\end{proof}

\subsubsection{Unbounded repulsive asymptotically doubling cost}\label{sec:asymp_doubling}
In this section we discuss a different condition on $c_2$ which controls the speed at which it goes to zero at infinity and is inspired of the work~\cite{ButChaPas-18,ColMarStr-19}.

\begin{definition}[Asymptotically doubling] \label{def:asymp_doubling}
We consider
\begin{equation}
 m(r)=\inf \{ c_2(x,y) \; : \; |x-y| \leq  r\}, \quad M(r)= \sup \{ c_2(x,y) \; : \; |x-y| \geq r \}.
 \label{eq:def_m_M}
\end{equation}
We say that $c_2$ is \emph{asymptotically doubling} if there exists $C>0$ such that, for some large enough $R$ we have
$m(2r) \geq C  M(r)$ for all $r\geq R$.
\end{definition}

The last condition is a way to quantify how fast $c_2$ can converge to zero at infinity. Indeed, if $c_2(x,y)=w(|x-y|)$ is translation-invariant, radial-decreasing and continuous, then we have simply $m(r)=M(r)=w(r)$. The condition is then that $\inf_{r\geq R} w(2r)/w(r)>0$, that is, $w$ should not converge faster than a polynomial at infinity. For instance, a Riesz cost $c_2(x,y)=|x-y|^{-s}$ with $s>0$ is asymptotically doubling with $C=2^{-s}$. The following is a generalization to the grand-canonical case of ideas from~\cite{ColMarStr-19}.

\begin{theorem}[Support for asymptotically doubling pairwise costs]\label{thm:support_asymp_doubling}
Let $\Omega\subset \R^d$ and $\bc=(c_n)_{n\geq0}$ be a pairwise cost  as in~\eqref{eq:pairwise2} with $c_2$ a positive lower semi-continuous function.  Let $\rho$ be a non-negative measure such that $\cC(\rho)<\ii$, $\rho(\Omega)>1$ and so that there exists $r>0$ with
$\kappa:=\sup_{x\in\Omega}\rho(B_r(x))<1$
(for instance $\rho$ has no atom).
Let $\bP^*=(\bP_n^*)_{n\geq0}$ be any minimizer for $\cC(\rho)$.

\medskip

\noindent $\bullet$ \emph{(Diagonal estimate).} We have for all $n\geq2$
\begin{equation}\label{eq:diagonal_estimate}
\max_{1\leq i<j\leq n}c_2(x_i,x_j) \leq \frac{\rho(\Omega)}{1-\kappa} M(r),\qquad \text{$\bP^*_n$--a.e.},
\end{equation}
where $M(r)$ is defined in~\eqref{eq:def_m_M}. In particular, if $c_2$ diverges at the origin,
$\lim_{r\to0}M(r)=+\ii,$
then $\bP^*_n$ concentrates outside of the diagonal:
$$\bP_n\left(\min_{1\leq i<j\leq n}|x_i-x_j|<\eps\right)=0,\qquad \eps=M^{-1}\left(\frac{\rho(\Omega)}{1-\kappa} M(r)\right).$$

\medskip

\noindent $\bullet$ \emph{(Support).} If in addition $c_2$ is asymptotically doubling, then $\bP^*$ has compact support:
\begin{equation}
{\rm Supp}(\bP^*)\subset\left[0\,,\, 1+\rho(\Omega)\max\left(\frac2{C^2},\frac{M(r)}{(1-\kappa)m(2R_0)}\right)\right]
\label{eq:support_asymp_doubling}
\end{equation}
where $R_0$ is the smallest radius such that $\rho(\Omega\setminus B_{R_0})\leq1/2$ and $R_0\geq R$. The parameters $R$ and $C$ are the ones from Definition~\ref{def:asymp_doubling}.
\end{theorem}

In the Coulomb case $c_2(x,y)=|x-y|^{-1}$, the estimate~\eqref{eq:diagonal_estimate} gives that all the $\bP_n^*$ are supported on
$$\left\{\min_{i\neq j}|x_i-x_j|\geq \frac{r(1-\kappa)}{\rho(\Omega)}\right\}.$$
A similar estimate was proved in the canonical case in~\cite{ButChaPas-18,ColMarStr-19}. Note that the support estimate~\eqref{eq:support_asymp_doubling} is less precise than in our other previous results.

\begin{proof}
Our proof is split into several intermediate lemmas.

\begin{lemma}\label{lem:law_large_numbers}
Let $\rho$ be a non-negative measure and $\mathbb{P} \in \PiGC(\rho)$ such that $\bP_0=0$ (this implies $\rho(\Omega)\geq1$). Let us consider a non-negative measurable function $V : \Omega \to [0, +\infty]$ such that $ \int_\Omega V(x) \, {\rm d} \rho(x) <\infty$. Then there exists $n\geq1$ with $\bP_n\neq0$ such that
\begin{equation}
 \sum_{i=1}^nV(x_i)\leq \int_\Omega V(x) \, d \rho(x)\quad\text{on a $\bP_n$--non negligible set. }
 \label{eq:law_large_numbers}
\end{equation}
\end{lemma}

\begin{proof}[Proof of Lemma~\ref{lem:law_large_numbers}]
We have
$$
\sum_{n\geq1} \int_{\Omega^n}\left(\sum_{i=1}^nV(x_i)\right){\rm d}\,\bP_n= \int_{\Omega} V(x) \, {\rm d} \rho(x)=\left(\sum_{n\geq1}\bP_n(\Omega^n)\right)\int_{\Omega} V(x) \, {\rm d} \rho(x)
$$
and therefore the statement follows.
\end{proof}

\begin{lemma}\label{lem:law_large_numbers_outside_ball}
Let $\rho$ be a non-negative measure and $\mathbb{P} \in \PiGC(\rho)$ such that $\bP_0=0$ (this implies $\rho(\Omega)\geq1$). Assume that $B$ is a Borel set such that $\rho(B)<1$. Let us consider a positive lower semi-continuous function $V : \Omega \to (0, +\infty]$ such that $ \int_\Omega V(x) \, {\rm d} \rho(x) <\infty$. Then there exists $n\geq1$ with $\bP_n\neq0$ such that
\begin{equation}
 \sum_{i=1}^nV(x_i)\leq \frac1{1-\rho(B)}\int_{\Omega\setminus B} V(x) \, {\rm d} \rho(x)\quad\text{on a $\bP_n$--non negligible set.}
 \label{eq:law_large_numbers_outside_ball}
\end{equation}
\end{lemma}

\begin{proof}[Proof of Lemma~\ref{lem:law_large_numbers_outside_ball}]
Define the truncated potential $V_A=V\1_{\Omega\setminus B}+A\1_B$. From Lemma~\ref{lem:law_large_numbers} there exists $n$ with $\bP_n\neq0$ such that
\begin{equation}
 \big|B\cap\{x_1,...,x_n\}\big|+\frac1A\sum_{j=1}^nV(x_j)\1_{\Omega\setminus B}(x_j)\leq \frac1A\int_{\Omega\setminus B} V(x)\,{\rm d}\rho(x)+\rho(B)
 \label{eq:apply_Lemma_LLN}
\end{equation}
on some $\bP_n$--non-negligible set in $\Omega^n$. For
$$A=\frac1{1-\rho(B)}\int_{\Omega\setminus B} V(x)\,{\rm d}\rho(x)$$
the right side of~\eqref{eq:apply_Lemma_LLN} is equal to 1. Since $\big|B\cap\{x_1,...,x_n\}\big|$ is an integer and the second term is positive when $\big|B\cap\{x_1,...,x_n\}\big|=1$, this shows that there is no point in $B$ for this $A$ and thus, on this set,
$$\sum_{j=1}^nV(x_j)\leq \int_{\Omega\setminus B} V(x)\,{\rm d}\rho(x)+A\rho(B)=\frac1{1-\rho(B)}\int_{\Omega\setminus B} V(x)\,{\rm d}\rho(x).$$
\end{proof}

\begin{lemma}\label{lem:law_large_numbers_interaction}
Assume that $\bc$ is a pairwise cost as in~\eqref{eq:pairwise2} with $c_2$ a positive lower semi-continuous function. Let $\rho$ be a non-negative measure such that $\cC(\rho)<\ii$ and $\rho(\Omega)\geq1$. Let $\mathbb{P}^* \in \PiGC(\rho)$ be a minimizer. Assume that there exists $r>0$ such that
$\kappa:=\sup_{x\in\Omega}\rho\big(B_r(x)\big)<1.$
For every $n\geq2$ and every $i\in\{1,...,n\}$, we have $\bP_n^*$--almost everywhere
\begin{align}
 \sum_{j\neq i}c_2(x_i,x_j)&\leq \frac1{1-\rho\big(B_r(x_i)\big)}\int_{\Omega\setminus B_r(x_i)} c_2(x_i,y) \, {\rm d} \rho(y)\nonumber\\
 &\leq \frac{\rho\big(\Omega\setminus B_r(x_i)\big)M(r)}{1-\kappa}.
 \label{eq:law_large_numbers_interaction}
\end{align}
\end{lemma}

\begin{proof}[Proof of Lemma~\ref{lem:law_large_numbers_interaction}]
By Lemma~\ref{lem:repulsive_N1} we know that $\bP_0^*=0$.
Assume that $\bP_n^*\neq0$. Applying Lemma~\ref{lem:law_large_numbers_outside_ball} to $V(x)=c_2(x,x_i)$ we deduce that there exists $n_0$ such that for every $x_i\in\Omega$
 $$\sum_{k=1}^{n_0} c_2(x_i,y_k)\leq \frac1{1-\rho\big(B_r(x_i)\big)}\int_{\Omega\setminus B_r(x_i)} c_2(x_i,y) \, {\rm d} \rho(y)$$
for all $Y=(y_1,...,y_{n_0})\in\Omega^{n_0}$ in a $\bP_{n_0}^*$--non negligible set (depending on $x_i$).
But on the support of $\bP_n^*\otimes\bP_{n_0}^*$ we have by Lemma~\ref{lem:cmonot}
\begin{multline*}
\sum_{j\neq i}c_2(x_i,x_j)=c_n(X)-c_{n-1}(X_{\{1,...,n\}\setminus\{i\}})\\
\leq c_{n_0+1}(Y,x_i)-c_{n_0}(Y)=\sum_{k=1}^{n_0}c_2(x_i,y_k)
\end{multline*}
and the result follows.
\end{proof}

Now we are able to conclude the proof of Theorem~\ref{thm:support_asymp_doubling}. For
$n\geq2$ and $1\leq i\neq j\leq n$, we have by Lemma~\ref{lem:law_large_numbers_interaction}
$$c_2(x_i,x_j)\leq \sum_{k\neq i}c_2(x_i,x_k)\leq \frac1{1-\kappa}\int_{\Omega\setminus B_r(x_i)} c_2(x_i,y) \, {\rm d} \rho(y)\leq \frac{\rho(\Omega)M(r)}{1-\kappa},$$
which is the diagonal estimate~\eqref{eq:diagonal_estimate}.

Next we discuss the support of $\bP^*$. We consider a sufficiently large $R_0$ so that the asymptotically doubling property holds for $r\geq R_0$ and
$\rho(\Omega\cap B_{R_0/2})\geq \rho(\Omega)-\frac12.$
We assume that $\bP_n^*\neq0$ and consider two cases on the support of $\bP_n^*$. The first case is when $\max(|x_i|)> R_0$. By symmetry we can assume that $|x_1|=\max(|x_i|)=R'> R_0$. Then, by the asymptotically doubling property used twice, we have since $|x_i-x_j|\leq 2R'$
$$\sum_{j=2}^nc_2(x_1,x_j)\geq (n-1)m(2R')\geq (n-1)C^2M(R'/2).$$
On the other hand, we have
$\rho(B_{R'/2}(x_1))\leq \rho(\Omega\setminus B_{R_0/2})\leq\frac12$
since $B_{R'/2}(x_1)\subset \R^d\setminus B_{R_0/2}$. Using Lemma~\ref{lem:law_large_numbers_interaction} we obtain
$$ \sum_{j=2}^nc_2(x_1,x_j)\leq \frac1{1-\rho\big(B_{R'/2}(x_1)\big)}\int_{\Omega\setminus B_{R'/2}(x_1)} c_2(x_1,y) \, {\rm d} \rho(y)\leq 2\rho(\Omega)M(R'/2)$$
$\bP_n^*$--almost everywhere. Thus on this set we get the bound
$n\leq 1+\frac{2\rho(\Omega)}{C^2}$.
On the other hand, if $\max|x_i|\leq R_0$ we can use the bound
$$(n-1)\, m(2R_0)\leq \sum_{j=2}^nc_2(x_1,x_j)\leq \frac{\rho(\Omega)M(r)}{1-\kappa}$$
and we obtain~\eqref{eq:support_asymp_doubling}.
\end{proof}

\subsection{The 3D Coulomb cost}\label{sec:Coulomb}

\subsubsection{A better estimate on the support}
Next we discuss in detail the 3D Coulomb potential
$$c_2(x,y)=\frac1{|x-y|}$$
in any dimension $d\geq1$. The bound~\eqref{eq:support_triangle} gives
$${\rm Supp}(\bP^*)\subset\left[\frac{N_0+1}{3}\,,\,3N_0-3\right],\qquad \rho(\Omega)=N_0\in\N\setminus\{1\}.$$
A natural question is to ask whether the length of the support is really of the order of $N_0$, or smaller.

The \emph{ionization conjecture} is a celebrated problem in quantum mathematical physics~\cite{Simon-00,LieSei-09} which states that a nucleus of charge $Z$ can bind at most $N\leq Z+M$ electrons. This conjecture is known to fail for bosons~\cite{BenLie-83,Solovej-90} and it should thus deeply rely on the Pauli principle, that is, on the fermionic nature of the electrons. For classical results in the same spirit, see for instance~\cite[Thm.~2.1~\&~3.1]{LieSigSimThi-88}. A natural similar conjecture in our setting would be that the support of any minimizer for $\cC(\rho)$ is included in $[\rho(\R^d)-M,\rho(\R^d)+M]$ for a universal constant $M$. This happens to be \emph{true} in dimension $d=1$ with $M=1$, as we will see in Section~\ref{sec:1D}, but \emph{wrong} in dimensions $d\geq2$, as discussed later in this section. We can however prove that the length of the support is much smaller than the average $\rho(\R^d)$, at most of the order of $\sqrt{\rho(\R^d)}$ in any dimension $d\geq1$.

\begin{theorem}[Support for Coulomb cost]\label{thm:support_Coulomb_3D}
Let $\Omega=\R^d$, $d\geq1$, and
$$c_0=c_1=0,\qquad c_n(x_1,...,x_n)=\sum_{1\leq j<k\leq n}\frac{1}{|x_j-x_k|}.$$
Then for all finite non-negative measure $\rho$ such that $\cC(\rho)<\ii$ and $\rho(\R^d)>1$, any minimizer $\bP^*$ for $\cC(\rho)$ satisfies
\begin{multline}
{\rm supp}(\bP^*)\subset\\
\left[\lfloor\rho(\Omega)\rfloor-\frac12\sqrt{8\lfloor\rho(\Omega)\rfloor+9}+\frac32\;,\;
\lceil\rho(\Omega)\rceil+\frac12\sqrt{8\lceil\rho(\Omega)\rceil-7}-\frac12\right].
 \label{eq:support_Coulomb}
\end{multline}
If $\rho(\R^d)= 2$, then we have
${\rm supp}(\bP^*)=\{2\}.$
\end{theorem}

\begin{proof}
Our proof uses a method introduced in~\cite{FraKilNam-16,FraNamBos-18}. Let $N,K\geq1$ and $X=(x_1,...,x_N)\in\R^{dN}$, $Y=(y_1,...,y_K)\in\R^{dK}$ satisfying the $\bc$-monotonicity property~\eqref{eq:cmonot}. For $\nu \in \mathbb{S}^{d-1}$ and $t \in \R$, we introduce the two sets
$ A^-= \{ x \in \R^d \; : \; x \cdot \nu < t\}$ and $A^+ = \{ x \in \R^d \; : \; x \cdot \nu \geq t\}$.
Let
$I:=\big\{i\in\{1,...,N\}\; :\; x_i\in A^{+}\big\}$
and denote $X^+:=X_{I}$ as well as $X^-:=X_{I^c}$. Do the same for the $y_j$'s.
The $\bc$-monotonicity formula~\eqref{eq:cmonot} tells us that
$$ \sum_{\substack{ x_i \in A^+,\\ y_{k} \in A^-}}  \frac {1}{|x_i - y_{k}|} + \sum_{\substack{x_i \in A^- ,\\ y_{k} \in A^+ }} \frac 1{|x_i - y_{k}|} \geq \sum_{\substack{ x_i \in A^+,\\ x_j \in A^-}}  \frac 1{|x_i - x_j|} + \sum_{\substack{y_{k} \in A^- ,\\ y_{\ell} \in A^+ }} \frac 1{|y_{k} - y_{\ell}|}. $$
Integrating this inequality with respect to $t$ using
$$\int_\R \1(x\cdot \nu\geq t)\1(y\cdot \nu< t)\,dt=\big((x-y)\cdot \nu\big)_+$$
we find
$$ \sum_{i=1}^N\sum_{k=1}^K \frac { | (x_i- y_{k})\cdot \nu |}{|x_i -y_{k}|} \geq \sum_{1\leq i<j\leq N} \frac { | (x_i- x_{j})\cdot \nu |}{|x_i -x_j|} + \sum_{1\leq k<\ell\leq K} \frac { | (y_{k}- y_{\ell})\cdot \nu |}{|y_{k} -y_{\ell}|}. $$
Finally, averaging over $\nu\in\bS^{d-1}$ we find
$$ N K \geq  \frac{N(N-1)}{2}+\frac{K(K-1)}{2}$$
which is the same as
$(N-K)^2 \leq N+K$ .
If $N\geq K$ this gives
$$N\leq K+\frac12 \sqrt{1+8K}+\frac12,$$
which is better than the bound $N\leq 3K$ obtained in the proof of Theorem~\ref{thm:support_triangle}. The rest of the argument is the same as before.

When $\rho(\R^d)=2$, our estimates~\eqref{eq:support_triangle} and~\eqref{eq:support_Coulomb} both give ${\rm supp}(\bP^*)\subset[1,3]$ and we have to turn one large inequality to a strict one. Going back for instance to the proof of Theorem~\ref{thm:support_triangle} with $K=1$ and $N=3$, we notice that the triangle inequality in~\eqref{eq:use_triangle} is strict unless $y_1$ lies on the segment joining $x_i$ and $x_j$ for all $i\neq j$. This can only happen if two $x_j$'s coincide, but then the cost is infinite. Thus we get the strict inequality for $K=1$ and $N=3$ and this shows that ${\rm supp}(\bP^*)=\{2\}$ for $\rho(\R^d)=2$.
\end{proof}

\subsubsection{The case of 6 points at half filling}
When $\rho(\R^d)=3$ the bound of Theorem~\ref{thm:support_Coulomb_3D} gives ${\rm supp}(\bP^*)\subset\{2,3,4\}$ and this is actually optimal. In this section we present an example in dimension $d=2$ where ${\rm supp}(\bP^*)=\{2,4\}$. The same applies in higher dimensions by working on any plane. Note that the one-dimensional case is different, as we will see later in Section~\ref{sec:1D}.

Let us consider 6 points $x_1,...,x_6$ in the plane $\R^2$ and the uniform measure
$$\rho=\frac12\sum_{i=1}^6\delta_{x_i}.$$
That each point is occupied with the probability $1/2$ (half filling) will  play an important role. To make things explicit, we place our points as displayed in Figure~\ref{fig:losange}: Four points form a diamond of side length 1 and diagonal equal to $2t$, whereas the other two points are placed outside at a distance 1. The only parameter is the length $t\in(0,1)$. For $t=0.7$ we found that
\begin{equation}
 \bP^*_2=\frac 12 \delta_{x_1}\otimes_s\delta_{x_2},\qquad \bP^*_4=\frac 12 \delta_{x_3}\otimes_s\delta_{x_4}\otimes_s\delta_{x_5}\otimes_s\delta_{x_6}.
 \label{eq:example_N3}
\end{equation}
is the \emph{unique minimizer of the grand-canonical problem}, where $x_1,x_2$ are the two points indicated on the figure.
We explain in the forthcoming work~\cite{MarLewNen-24} that this example can be used to construct a counter-example to a famous open problem in chemistry, namely the convexity of the ground state energy with respect to the number of particles in the system~\cite[Question~7 p.~263]{Lieb-83b}.

\begin{figure}[t]
\centering
\includegraphics{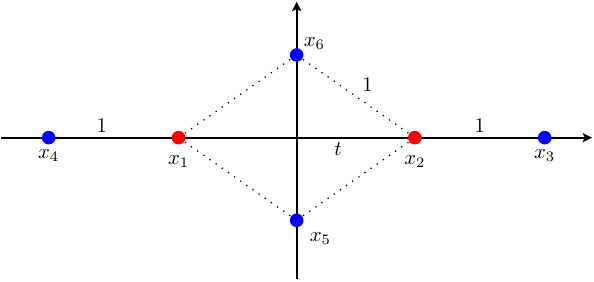}
\caption{For $t=0.7$, the measure $\rho = \frac 12 \sum_{i=1}^6 \delta_{x_i}$ has the unique optimal grand canonical probability $\bP^*_2=\delta_{(x_1,x_2)}/2$, $\bP^*_4= \delta_{(x_3,x_4,x_5,x_6)})/2$. In particular, $\cC(\rho)<\cC_3(\rho)$.\label{fig:losange}}
\end{figure}

Let us explain how we found $\cC(\rho)$ and its optimizer $\bP^*$. Since $\rho$ is supported on the 6 points, our optimal probabilities $\bP_n$ must all be supported on $\{x_1,...,x_6\}^n$. In addition, two particles can never be at the same location otherwise the cost is infinite. By symmetry, we conclude that any $\bP_n$ is a combination of the $6\choose n$ elementary probabilities consisting of putting one particle at each of the $x_{i_1},...,x_{i_n}$ with $1\leq i_1<\cdots<i_n\leq 6$. Let $p_{\sigma_1,...,\sigma_{6}}$ be the probability that there is a particle at each of the $x_j$  for $\sigma_j=1$ and none at the other points. These numbers satisfy the constraints
\begin{equation}
 \begin{cases}
\dps \sum_{\sigma_1,...,\sigma_{6}\in\{0,1\}}p_{\sigma_1,...,\sigma_{6}}=1\\[0.3cm]
\dps\sum_{\substack{\sigma_1,...,\sigma_{k-1},\\ \sigma_{k+1},...,\sigma_{6}\in\{0,1\}}}p_{\sigma_1,...,\sigma_{k-1},\sigma,\sigma_{k+1},...,\sigma_{6}}=\frac12,\qquad \forall k=1,...,6,\ \forall \sigma\in\{0,1\}.
 \end{cases}
 \label{eq:constraints_36}
\end{equation}
In other words, the 6 marginals are all Bernoulli. The total average cost is
\begin{equation}
 \bP(\bc)=\sum_{\sigma_1,...,\sigma_{6}\in\{0,1\}}p_{\sigma_1,...,\sigma_{6}}\,c_{\sigma_1,...,\sigma_{6}}=:p(c),
 \label{eq:cost_6}
\end{equation}
where
$$c_{\sigma_1,...,\sigma_{6}}:=\sum_{1\leq j<k\leq 6}\frac{\sigma_j\sigma_k}{|x_j-x_k|}.$$
Solving the minimization problem $\cC(\rho)$ is thus equivalent to finding the probabilities $p_{\sigma_1,...,\sigma_{6}}$. Note that these are not symmetric, in this way of writing. The cost $c_{\sigma_1,...,\sigma_6}$ is also not symmetric. Hence we obtain a non-symmetric $6$-marginal problem in $\{0,1\}^{6}$ with the constraints that the marginals are all Bernoulli. A similar situation has already been discussed for the Coulomb problem in~\cite{FriVog-18,Friesecke-19,KhoLinLinLex-20}. Here we know in addition that $p_{\sigma_1,...,\sigma_6}=0$ if $\sum_{j=1}^6\sigma_j\in\{0,1,5,6\}$ for a minimizer.

We can further simplify the problem by using the fact that we are working at half filling. To a probability $p=(p_{\sigma_1,...,\sigma_6})_{\sigma_1,...,\sigma_6\in\{0,1\}}$ we associate its complement $\widetilde{p}$ defined by
$$\widetilde{p}_{\sigma_1,...,\sigma_6}=p_{1-\sigma_1,...,1-\sigma_6},\qquad\forall\sigma_1,...,\sigma_6\in\{0,1\}.$$
In other words, we replace each particle by a hole and each hole by a particle. Then we have
\begin{align*}
\widetilde{p}(c)&=\sum_{\sigma_1,...,\sigma_6\in\{0,1\}}p_{\sigma_1,...,\sigma_6}\sum_{1\leq j<k\leq 6}\frac{(1-\sigma_j)(1-\sigma_k)}{|x_j-x_k|} \\
&=\sum_{1\leq j<k\leq 6}\frac{1}{|x_j-x_k|}+p(c)\\
&\qquad-2\sum_{\sigma_1,...,\sigma_6\in\{0,1\}}p_{\sigma_1,...,\sigma_6}\sum_{1\leq j< k\leq 6}\frac{\sigma_j}{|x_j-x_k|}.
\end{align*}
At half-filling the last term equals
\begin{align*}
&2\sum_{\sigma_1,...,\sigma_6\in\{0,1\}}p_{\sigma_1,...,\sigma_6}\sum_{1\leq j< k\leq 6}\frac{\sigma_j}{|x_j-x_k|}\\
&\qquad =\sum_{j=1}^6\sum_{\sigma_j\in\{0,1\}}\sigma_j\sum_{k\neq j}\frac{1}{|x_j-x_k|}\underbrace{\sum_{\sigma_1,...,\sigma_{j-1},\sigma_{j+1}\in\{0,1\}}p_{\sigma_1,...,\sigma_6}}_{=\frac12}\\
&\qquad=\sum_{1\leq j<k\leq 6}\frac{1}{|x_j-x_k|}.
\end{align*}
Therefore we have proved that the problem is invariant under the particle-hole symmetry:
$p(c)=\widetilde{p}(c).$
By linearity we can restrict our minimization to particle-hole symmetric probabilities $\widetilde{p}=p$, without changing the value of the minimum. Those symmetric probabilities all have the required density $1/2$. They form a convex set, whose extreme points are given by the elementary probabilities
\begin{equation}
 p^{I}:=\frac12\left(\prod_{i\in I}\delta_1(\sigma_i)\prod_{i\notin I}\delta_0(\sigma_i)+\prod_{i\notin I}\delta_1(\sigma_i)\prod_{i\in I}\delta_0(\sigma_i)\right)
 \label{eq:extremal_pt}
\end{equation}
consisting of placing particles in and outside a given set $\{x_i,\ i\in I\}$ indexed by $I\subset\{1,...,6\}$ with $|I|\leq 3$. In our case, we also know that we can restrict to $|I|\in\{2,3\}$. Since we are minimizing a linear function over a convex set, the minimum is attained at an extreme point and all what is left is to compute the value of the energy of each such extreme point
$$p^I(c)=\frac14\left(\sum_{j\neq k\in I}\frac{1}{|x_j-x_k|}+\sum_{j\neq k\notin I}\frac{1}{|x_j-x_k|}\right).$$
For our six particles placed as in Figure~\ref{fig:losange}, we simply computed the energy cost of all the previous extreme points~\eqref{eq:extremal_pt} as a function of $t$. The curves are displayed in Figure~\ref{fig:curves_losange}. We found that $p^{\{1,2\}}$ had a strictly lower energy than all the other extreme points for $t$ in a neighborhood of $0.7$. This shows that it is the \emph{unique minimizer among particle-hole symmetric states}.

The exact same symmetry argument applies when we restrict our attention to canonical states, which form the convex hull of the $p^I$ with $|I|=3$. The fact that $p^I(c)>p^{\{1,2\}}(c)$ for all $|I|=3$ implies that the grand canonical problem does not coincide with the canonical problem: $\cC_3(\rho)>\cC(\rho)$.

Next we prove that $p^{\{1,2\}}$ is actually the unique \emph{global} minimizer of $\cC(\rho)$. Let $p$ be any other optimizer. Then $(p+\widetilde{p})/2$ is a particle-hole symmetric minimizer and thus we must have
$\frac{p+\widetilde{p}}{2}=p^{\{1,2\}}.$
Since $p,\widetilde{p}\geq0$ and $p^{\{1,2\}}$ vanishes outside of $(1,1,0,0,0,0)$ and $(0,0,1,1,1,1)$ we conclude that $p$ must be a combination of $\delta_{(x_1,x_2)}$ and $\delta_{(x_3,x_4,x_5,x_6)}$. The constraint that the marginals of $p$ are all Bernoulli give $p=p^{\{1,2\}}$.

\begin{figure}[t]
\centering
\includegraphics[width=6cm]{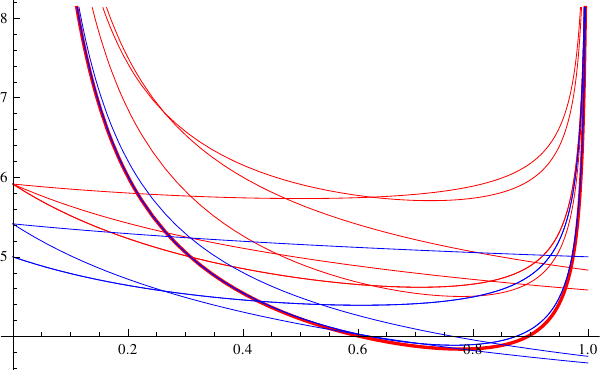}\hfill \includegraphics[width=6cm]{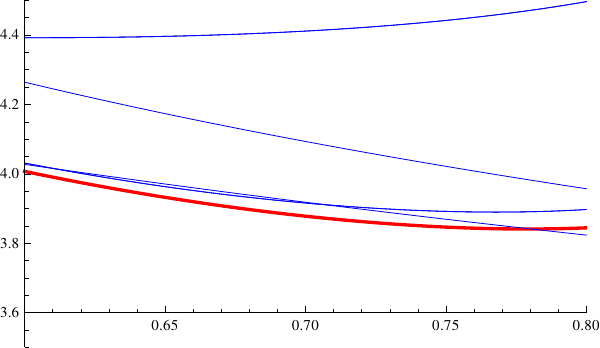}
\caption{Coulomb cost of all the extreme points~\eqref{eq:extremal_pt} for the $x_j$ as in Figure~\ref{fig:losange}, as functions of the length $t\in(0,1)$ (left), with a zoom around the value $t=0.7$ (right). The grand-canonical extreme points $(|I|=2$) are in red whereas the canonical ones $(|I|=3$) are in blue. For $t\simeq0.7$, $p^{\{1,2\}}$ has the lowest possible Coulomb cost. Its curve has been made thicker in the drawing.\label{fig:curves_losange}}

\medskip

 \includegraphics[scale=0.6]{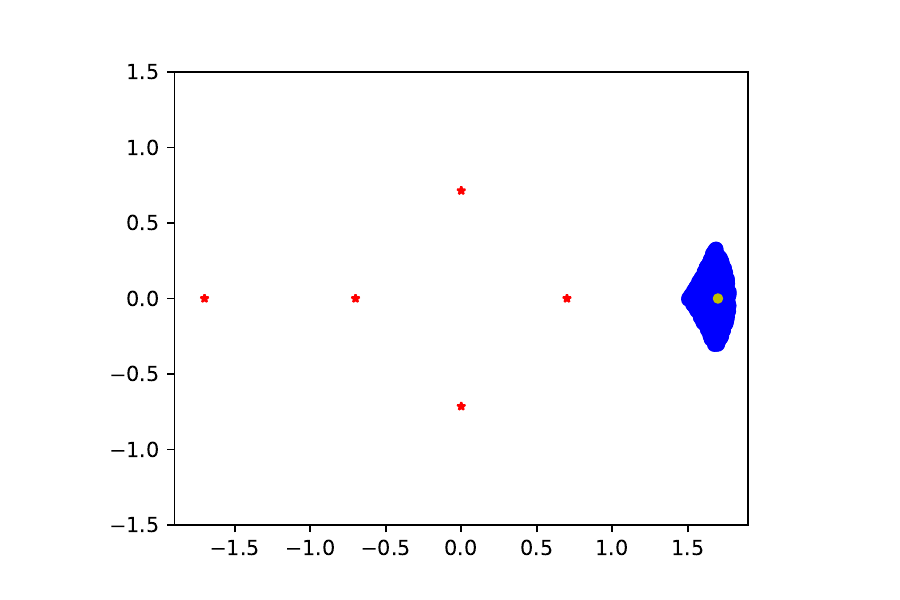}

\caption{The blue area represents represents all the positions of $x_3$ for which the configuration is grand-canonical.\label{fig:perturbation}}

\end{figure}

In Figure~\ref{fig:curves_losange} it appears that the grand-canonical optimizer $p^{\{1,2\}}$ has an energy lower but still rather close to the corresponding canonical optimizer. In addition, for most values of the half diagonal $t$ the minimizer will actually be canonical. To further illustrate the range of validity of the grand-canonical problem, we display in Figure~\ref{fig:perturbation} a different experiment. Taking $t=0.7$, we fixed 5 points and moved only the point $x_3$ further to the right. The shaded area represents all the positions of $x_3$ for which the configuration is grand-canonical.

\subsubsection{Length of the support in the Coulomb case}
Next we explain how to use the previous example to generate another example which has a support of length of the order $\rho(\Omega)^\alpha$ with $\alpha=\frac {\ln(2)}{\ln (6) } \sim 0.38 $.

\begin{theorem}[Length of the support for Coulomb]\label{thm:large_support_Coulomb}
Let $\Omega=\R^d$ and $c_2(x,y)=|x-y|^{-1}$. Let $x_1,...,x_6\in\R^2$ be such that $\cC(\sum_{j=1}^6\delta_{x_j}/2)$ admits~\eqref{eq:example_N3} as unique minimizer. From these points we can inductively construct a sequence $(y_j^{(k)})_{j=1}^{6^k}\subset \R^d$ so that, at half filling
$$\rho^{(k)}=\frac1{2}\sum_{j=1}^{6^k}\delta_{y_j^{(k)}},\qquad \rho^{(k)}(\R^d)=\frac{6^k}{2},$$
the grand-canonical problem $\cC(\rho^{(k)})$ admits a unique minimizer $\bP^{(k)}$, which satisfies
$${\rm Supp}(\bP^{(k)})= \left\{\frac{6^k- 2^k}{2}\,,\, \frac{6^k+ 2^k}{2}\right\}.$$
\end{theorem}

\begin{proof}
Upon placing all the points in a plane, we can always assume that $d=2$. Our proof is a multiscale inductive construction illustrated in Figure~\ref{fig:multiscale}.

We denote by $y_i^{(1)}=x_i$ the 6 points of our example and start by explaining the first step of the procedure, that is, the construction of the $y^{(2)}_k$. We consider $6$ distinct points $X_0,...,X_5$ to be chosen later, and introduce the $36$ points
$ y^{(2)}_{6i+ j} = \ell_2 X_{i} + x_j$ for $i=0, \ldots, 5$ and $j=1, \ldots, 6$.
In other words, we place $6$ copies of our previous example, placed very far away at a distance $\ell_2\gg1$  to each other. We again work at half filling with the total density
$\rho^{(2)}:=\frac12\sum_{k=1}^{36}\delta_{y^{(2)}_{k}}.$
Our goal is to get some information on the support of an optimal grand-canonical minimizer in the limit $\ell_2\to\ii$.

Like in the previous section, we can rewrite this problem as a non-symmetric multi-marginal problem on $\{0,1\}^{36}$. We introduce the probabilities $p_{\sigma_1,...,\sigma_{36}}$ that there is a particle at $y_k^{(2)}$ for $\sigma_k=1$ and none for  $\sigma_k=0$ and rewrite our total cost as
\begin{equation}
 \bP(\bc)=p(c^{(2)})=\sum_{\sigma_1,...,\sigma_{36}\in\{0,1\}}p_{\sigma_1,...,\sigma_{36}}\,c^{(2)}_{\sigma_1,...,\sigma_{36}}
 \label{eq:cost_36}
\end{equation}
where
\begin{align*}
c^{(2)}_{\sigma_1,...,\sigma_{36}}&=\sum_{1\leq k<\ell\leq 36}\frac{\sigma_k\sigma_\ell}{|y_k^{(2)}-y_\ell^{(2)}|}\\
&=\sum_{i=0}^5\sum_{1\leq j<j'\leq 6}\frac{\sigma_{6i+j}\sigma_{6i+j'}}{|y_{6i+j}^{(2)}-y_{6i+j'}^{(2)}|}\\
&\qquad +\sum_{0\leq i< i'\leq 5}\sum_{1\leq j,j'\leq 6}\frac{\sigma_{6i+j}\sigma_{6i'+j'}}{|\ell_2(X_{i}-X_{i'})+x_j-x_{j'}|}.
\end{align*}
On the right of the last equality, the first term is the interaction between the particles in each cluster, whereas the second term is the lower order interaction between the clusters. In this formulation of the problem, our main goal is to derive an estimate on the smallest and largest value of the number of particles $\sum_{k=1}^{36}\sigma_k$ on the support of a minimizer $(p_{\sigma_1,...,\sigma_{36}})$. Its average is the total average number of particles 18 but we would like to prove that it fluctuates quite a bit around this value.

Since we are at half filling we know from the previous section that the problem is reduced to computing the energies of the particle-hole symmetric extreme points $p^I$ in~\eqref{eq:extremal_pt}, with $I\subset\{1,...,36\}$ and $|I|\leq 18$. If the lowest energy is attained at only one such point, our proof will be finished by the same arguments as in the case of $6$ points.

In the limit $\ell_2\to\ii$, our cost behaves as
\begin{multline}
c^{(2)}(\sigma_1,...,\sigma_{36})=\sum_{i=0}^5\sum_{1\leq j<j'\leq 6}\frac{\sigma_{6i+j}\sigma_{6i+j'}}{|y_{6i+j}^{(2)}-y_{6i+j'}^{(2)}|}\\
+\frac1{\ell_2}\sum_{0\leq i< i'\leq 5}\frac{\Sigma_i\Sigma_{i'}}{|X_i-X_{i'}|}+O\left(\frac1{(\ell_2)^2}\right)
\label{eq:expansion_large_ell}
\end{multline} with $\Sigma_i:=\sum_{j=1}^6\sigma_{6i+j}$,
the total number of particles in the cluster $i$. It is therefore rather easy to compute $p^I(c^{(2)})$ in this limit. The first term is minimized when
$$(\sigma_1,...,\sigma_6)\mapsto \sum_{\sigma_7,...,\sigma_{36}\in\{0,1\}}p^I_{\sigma_1,...,\sigma_{36}}$$
and the other 5 consecutive 6-marginals all coincide with our unique $6$-point minimizer
\begin{multline}
 p^*=\frac{\delta(\sigma_1=\sigma_2=1)\delta(\sigma_3=\cdots=\sigma_6=0)}{2}\\
 +\frac{\delta(\sigma_1=\sigma_2=0)\delta(\sigma_3=\cdots=\sigma_6=1)}{2}=:p^*_2+p^*_4.
 \label{eq:unique_min_6_reformulation}
\end{multline}
This simply means that
\begin{equation}
 I=I_0\cup\cdots\cup I_5,\qquad
 I_i=\begin{cases}
6i+\{1,2\}&\text{(hence $\Sigma_i=2$), or}\\
6i+\{3,4,5,6\}&\text{(hence $\Sigma_i=4$).}
     \end{cases}
 \label{eq:I_i}
\end{equation}
There are $2^6$ such extreme points. Some are canonical when three $\Sigma_i$ equal 2 and the three others equal 4. All the $p^I$ not of the form~\eqref{eq:I_i} have a higher energy.

In order to discriminate the $2^6$ extreme points, we have to look at the next order in~\eqref{eq:expansion_large_ell}. Its average value in the state $p^{I}$ is just
$$\frac12\sum_{0\leq i< i'\leq 5}\frac{\Sigma_i\Sigma_{i'}+(4-\Sigma_i)(4-\Sigma_{i'})}{|X_i-X_{i'}|}.	$$
Upon letting $\Sigma_i=2+2\tau_i$ with $\tau_i\in\{0,1\}$, this equals
$$4\sum_{0\leq i< i'\leq 5}\frac{1}{|X_i-X_{i'}|}+4\sum_{0\leq i< i'\leq 5}\frac{\tau_i\tau_{i'}}{|X_i-X_{i'}|}.$$
The second term is exactly our initial problem with 6 points. Thus we choose $X_i=x_{i+1}$ and conclude that for $\ell_2$ large enough our problem has the unique minimizer $p^I$ with
$I=I_0\cup\cdots \cup I_5$, $I_0=I_1-6=\{3,4,5,6\}$ and $I_i=6i+\{1,2\}$ for $i=2,4,5$.
A different way of writing the same is
$$p=2^5(p_2^*)^{\otimes 2}\otimes (p_4^*)^{\otimes 4}+2^5(p^*_4)^{\otimes 2}\otimes (p_2^*)^{\otimes 4}.$$

Next we repeat the argument with a similar construction at larger scales. By induction we construct a sequence $y_1^{(k)},...,y_{6^k}^{(k)}$ with
$\rho^{(k)}=\frac12\sum_{j=1}^{6^k}\delta_{y_j^{(k)}}$ and $\rho(\R^2)=\frac{6^k}{2}$
so that, in the successive limits $\ell_1,...,\ell_k\to\ii$, $\bP^{(k)}$ is unique and has its support in the sectors of $n^{(k)}_-$ and $n^{(k)}_+$ particles, with the relations
$n_-^{(k)}=2n_+^{(k-1)}+4n_-^{(k-1)}$ and $n_+^{(k)}=4n_+^{(k-1)}+2n_-^{(k-1)}$,
that is,
$$n_+^{(k)}=\frac{6^k}2+\frac{2^k}{2},\qquad n_+^{(k)}=\frac{6^k}2-\frac{2^k}{2}.$$
This concludes our proof of Theorem~\ref{thm:large_support_Coulomb}.
\end{proof}

\begin{figure}[t]
\centering
\includegraphics[width=6.3cm]{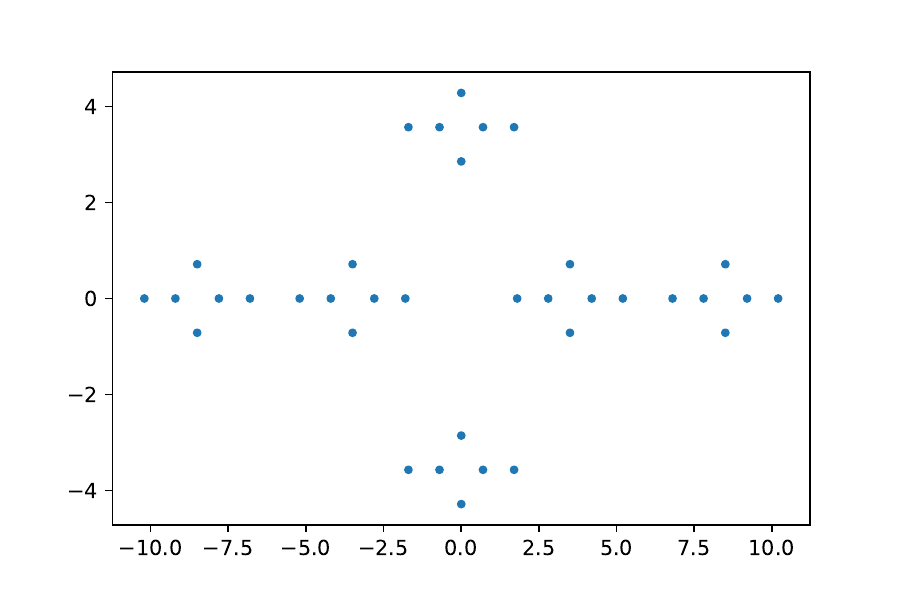}\hfill\includegraphics[width=6.3cm]{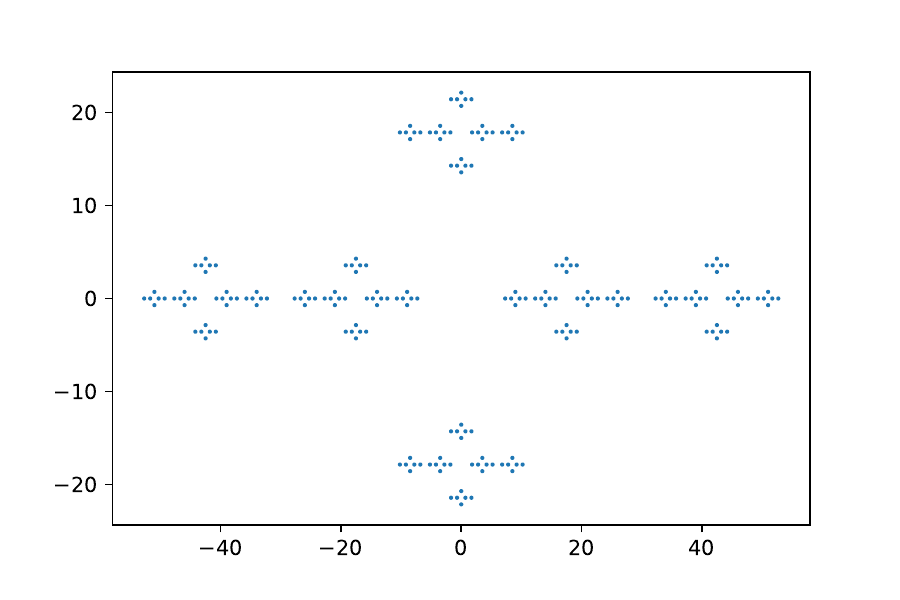}
\caption{Examples of the points $\{y_1^k , \ldots y_{6^k}^k\}$ for $k=2,3$ and the values $\ell_2=5$ and $\ell_3=25$, when we start from the configuration in Figure~\ref{fig:losange} with $t=0.7$. \label{fig:multiscale}}
\end{figure}

\subsection{Convex functions of the center of mass}\label{sec:center_mass}
We discuss a last example in the spirit of~\cite{MarGerNen-17}. Let $\Omega=\R^d$ and $\rho$ be any finite non-negative measure admitting a first moment:
$\int_{\R^d}|x|\,{\rm d}\rho(x)<\ii.$

\begin{definition}[Fixed center of mass]
We say that a $\bP=(\bP_n)_{n\geq0}\in\PiGC(\rho)$ has a \emph{fixed center of mass} whenever $\bP_n$ concentrates on the set
\begin{equation}
\cX_n(\rho):= \left\{(x_1,...,x_n)\in(\R^d)^n\ :\ \sum_{j=1}^nx_j=\int_{\R^d}x\,{\rm d}\rho(x)\right\}
 \label{eq:support_P_n_center_of_mass}
\end{equation}
for all $n\geq1$. By extension we say that a canonical $N$-particle probability $\bP_N$ has a fixed center of mass when $\bP:=(0,...,0,\bP_N,0,...)$ does.
\end{definition}

The interpretation of the definition is that the center of mass of the particles is deterministic, equal to the average against the measure $\rho$, this being true $\bP_n$--a.e.~for all $n\geq1$. The expectation against $\bP_n$ gives
$\int_{\R^d}x\,{\rm d}\rho_{\bP_n}=\bP_n(\R^{dn})\int_{\R^d}x\,{\rm d}\rho(x)$. Thus either $\bP_n=0$ or $\rho_{\bP_n}/\bP_n(\R^{dn})$ has for all $n\geq1$ the same center of mass as the total density $\rho$.

In general not all $\rho$'s admit canonical or grand-canonical multi-plans $\bP\in\PiGC(\rho)$ with a fixed center of mass. Here is a counter example adapted from~\cite[Rmk.~4.4]{MarGerNen-17}.

\begin{example}[Fixed center of mass for two points]
Consider the density $\rho=3(\delta_1+\delta_{-1})/2$ in $\Omega=\R$, which satisfies $\rho(\R)=3$ and $\int_{\R}x{\rm d}\rho(x)=0$. Then there does not exist any canonical $3$-probability of density $\rho$ with a fixed center of mass. Those are actually all of the form
$$\bP_3=p_{30}(\delta_{-1})^{\otimes 3}+p_{03}(\delta_1)^{\otimes 3}+p_{21}(\delta_{-1})^{\otimes 2}\otimes_s \delta_{1}+p_{12}\delta_{-1}\otimes_s( \delta_{-1})^{\otimes 2}$$
with $\sum p_{jk}=1$ and none of the elementary probabilities has the fixed center of mass $0$. The situation is different if we allow grand canonical probabilities. For instance the state satisfying
$\bP_0=\frac14$, $\bP_4=\frac{3}{4}(\delta_{-1})^{\otimes 2}\otimes_s (\delta_{1})^{\otimes 2}$ and $\bP_n=0$ otherwise,
belongs to $\PiGC(\rho)$ and has the fixed center of mass $0$.
However, the slightly asymmetric density $\rho=(7\delta_1+5\delta_{-1})/4$ admits no grand-canonical state $\bP\in\PiGC(\rho)$ of fixed center of mass. Indeed, we have $\int_{\R}x\,{\rm d}\rho(x)=1/2$ but any $\bP_n$ is a convex combination of the elementary probabilities
$(\delta_{-1})^{\otimes k}\otimes_s(\delta_{1})^{\otimes (n-k)}$ for $k=0,...,n$,
which has the center of mass $n-2k\in\Z$. It can never be equal to $1/2$.
\end{example}

The existence of a $\bP\in\PiGC(\rho)$ of fixed center of mass is related to a grand-canonical minimization problem.

\begin{theorem}[Convex functions of the center of mass]
Let $\Omega=\R^d$ and $\rho$ a finite measure with a finite first moment. Define
$X:=\int_{\R^d}x\,{\rm d}\rho(x)$.
Let $h:\R^d\to\R$ be a differentiable strictly convex function. Define the grand-canonical cost
\begin{equation}
c_0=h(X)-X\cdot\nabla h(X),\qquad
c_n(x_1,...,x_n)=h\left(\sum_{j=1}^nx_j\right),\quad n\geq1.
\label{eq:h_center_mass}
\end{equation}
There exists $\PiGC(\rho)$ with fixed center of mass if and only if $\cC(\rho)=h(X)$
and admits minimizers. In this case, the minimizers are exactly the $\bP\in \PiGC(\rho)$ with a fixed center of mass.
\end{theorem}

Note the need to appropriately choose $c_0$ in~\eqref{eq:h_center_mass}, a subtlety which does not occur in the canonical case considered in~\cite{MarGerNen-17}.

\begin{proof}
Let us introduce $\widetilde{h}(x):=h(x)-h(X)-(x-X)\cdot\nabla h(X)$, which is non-negative and vanishes only at $X$.
Let
$\widetilde{c}_0=0$ and $\widetilde{c}_n(x_1,...,x_n)=\widetilde{h}(\sum_{j=1}^nx_j)$ for $n\geq1$.
Then we have $\bP(\widetilde{\bc})=\sum_{n\geq1}\int\widetilde{h}(\sum_{j=1}^nx_j)\,{\rm d}\bP_n$ for any  $\bP=(\bP_n)_{n\geq0}\in\PiGC(\rho)$. This is positive and vanishes exactly when $\sum_{j=1}^nx_j=X$, $\bP_n$--almost surely. This is our definition of having a fixed center of mass. Remark then that
\begin{align*}
\bP(\widetilde{\bc})&=\sum_{n\geq1}\int_{\R^{dn}}h\!\left(\sum_{j=1}^nx_j\right)\,{\rm d}\bP_n(x_1,...,x_n) -\big(h(X)-X\cdot\nabla h(X)\big)(1-\bP_0)\\
&\qquad -\nabla h(X)\cdot\underbrace{\sum_{n\geq1}\int_{\R^{dn}}\sum_{j=1}^nx_j\,{\rm d}\bP_n(x_1,...,x_n)}_{=X}\\
&=\big(h(X)-X\cdot\nabla h(X)\big)\bP_0+\sum_{n\geq1}\int_{\R^{dn}}h\!\left(\sum_{j=1}^nx_j\right)\,{\rm d}\bP_n(x_1,...,x_n) -h(X)\\
&=\bP(\bc)-h(X),
\end{align*}
which concludes the proof.
\end{proof}

Next we show on an example that the support in $n$ of an optimal multi-plan for $\cC(\rho)$ can be very large in some situations.

\begin{example}[Finitely many points with the harmonic cost]\label{ex:square}
We consider the four corners $x_1=(1,-1)$, $x_2=(1,1)$, $x_3=(-1,1)$ and $x_4=(-1,-1)$ of a square centered at the origin and place the system at half-filling: $\rho=\frac12\sum_{j=1}^4\delta_{x_j}$.
Then $\int_{\R^2}x\,{\rm d}\rho(x)=0$ and $\cC(\rho)=0$ for
$$c_0=0,\qquad c_n(x_1,...,x_n)=\left|\sum_{j=1}^nx_j\right|^2.$$
There is a unique canonical minimizer, which consists of pairing the particles along the diagonal:
\begin{equation}
 \bP_2=\frac{\delta_{x_1}\otimes_s \delta_{x_3}+\delta_{x_2}\otimes_s \delta_{x_4}}2.
 \label{eq:canonical_square}
\end{equation}
On the other hand, there are many grand-canonical probabilities giving the same cost $\bP(\bc)=0$ with the same density. For instance we can fill the square uniformly with $4N$ particles with a probability $1/(2N)$
\begin{equation}
 \bP_0=1-\frac{1}{2N},\qquad\bP_{4N}=\frac1{2N}(\delta_{x_1})^{\otimes N}\otimes_s(\delta_{x_2})^{\otimes N}\otimes_s (\delta_{x_3})^{\otimes N}\otimes_s (\delta_{x_4})^{\otimes N}.
 \label{eq:grand_canonical_square}
\end{equation}
Since any convex combination is also optimal, we see that optimizers of $\cC(\rho)$ can have an arbitrarily large support.

It is possible to allow $c_0$ to be different from 0, while keeping $c_n=|\sum_{j=1}^nx_j|^2$ for $n\geq1$. For $c_0>0$ we find $\cC(\rho)=0$ with unique minimizer the canonical probability in~\eqref{eq:canonical_square}, since this is the unique minimizer which has $\bP_0=0$. For $c_0<0$ we can use the sequence~\eqref{eq:grand_canonical_square} and obtain $\cC(\rho)=c_0$ with no minimizer.
\end{example}

\section{Duality}\label{sec:duality}

In this section we study the dual problem. From~\eqref{eq:one-particle_cost} we know that the variables dual to the density $\rho$ are one-agent costs of the form $\bPhi=(\Phi_n)_{n\geq0}$ with $\Phi_0=0$ and $\Phi_n=\sum_{j=1}^n\phi(x_j)$ for a given $\phi\in C^0_b(\R^d)$.
We should however not forget the other constraint that $\bP$ forms a probability, $\bP_0+\sum_{n\geq1}\bP_n(\Omega^n)=1$, which requires the introduction of an additional Lagrange multiplier $\beta$. This constraint is independent of the density constraint, on the contrary to the usual multi-marginal problem. This leads us to the following \emph{dual problem}
\begin{multline}
\cD(\rho):=\sup\bigg\{\beta+\int_\Omega \phi(x)\,{\rm d}\rho(x)\ :\ \beta\leq c_0,\quad \phi\in C^0_b(\R^d),\\ \beta+\sum_{j=1}^n\phi(x_j)\leq c_n(x_1,...,x_n),\quad \forall n\geq1\bigg\}.
\label{eq:dual}
\end{multline}
If we take any $\beta$ and $\phi\in C_b^0(\R^d)$ satisfying the above constraints and any $\bP\in\PiGC(\rho)$, then we have
$$\bP(\bc)=c_0\bP_0+\sum_{n\geq1}\int_{\Omega^n}c_n\,{\rm d}\bP_n \geq \beta+\int_{\Omega}\phi(x)\,{\rm d}\rho(x)$$
which proves that $\cC(\rho)\geq \cD(\rho)$ in all situations. We would like to have equality.
It is possible to rewrite the dual problem in a slightly different manner. Let us introduce the grand-canonical ground state energy in the potential $\phi$
$$E(\phi):=\inf_{\substack{n\geq0\\ x_1,...,x_n\in\Omega^n}}\left(c_n(x_1,...,x_n)-\sum_{j=1}^n\phi(x_j)\right)=\inf_{\bP}\bP(\bc-\Phi).$$
The last infimum taken over all grand-canonical probabilities $\bP$ so that $\bP(\bc)<\ii$ and $\rho_\bP(\Omega)<\ii$ (without any other constraint on $\rho_\bP$). Like in Density Functional Theory~\cite{Lieb-83b,LewLieSei-23_DFT}, we rewrite the infimum over $\bP$ as an infimum over $\rho$ and then an infimum over all $\bP$ having this density $\rho$:
$$E(\phi):=\inf_{\rho}\left\{\cC(\rho)-\int_\Omega\phi(x)\,{\rm d}\rho(x)\right\}$$
with the infimum taken over all finite non-negative measures. We see that $-E$ is nothing but the Legendre-Fenchel transform of $\cC$. On the other hand, in~\eqref{eq:dual} the largest possible $\beta$ at fixed $\phi$ is indeed equal to $E(\phi)$, and therefore we can rewrite~\eqref{eq:dual} in the form
\begin{equation}
\cD(\rho)=\sup_{\phi\in C^0_b(\R^d)}\left\{\int_{\Omega}\phi(x)\,{\rm d}\rho(x)+E(\phi)\right\}.
\label{eq:dual_Legendre}
\end{equation}
Thus, $\cD$ is the Legendre-Fenchel transform of $-E$. From Theorem~\ref{thm:existence} and the Fenchel duality theorem for convex lower semi-continuous functions~\cite{Simon-11}, we conclude the following, which is an extension of a well known result in the multi-marginal case~\cite{Kellerer-84}.

\begin{theorem}[Duality]\label{thm:duality}
Let $\Omega\subset\R^d$ be any Borel set. Let $\bc=(c_n)_{n \geq 0} $ be a superstable family of lower semi-continuous costs. Then we have $\cC(\rho)=\cD(\rho)$.
\end{theorem}

Next we turn to the more complicated question of the existence of an optimal pair $(\beta,\phi)$ for the dual problem. As a first step we relax a bit the notion of dual potentials and assume that $\phi$ is in $L^\ii(\rd\rho)$ instead of $C^0_b$:
\begin{multline}
\label{eq:dualinf}
\widetilde\cD(\rho):=\sup\bigg\{\beta+\int_\Omega \phi(x)\,{\rm d}\rho(x)\ :\ \beta\leq c_0,\quad \phi\in L^{\infty}(\rd\rho),\\ \beta+\sum_{j=1}^n\phi(x_j)\leq c_n(x_1,...,x_n) \quad \rho^{\otimes n}\text{--a.e.}\quad \forall n\geq1\bigg\}.
\end{multline}
The following is a rather simple result which provides the existence of the dual pair $(\beta,\phi)$ under strong assumptions on the cost.

\begin{theorem}[Existence of a dual potential]\label{thm:existence_dual}
Let $\Omega$ be any open set in $\R^d$ and $\rho$ be any finite measure on $\Omega$. Let $\bc=(c_n)_{n \geq 0} $ be a superstable family of lower semi-continuous costs such that

\noindent $(i)$ $c_1\in L^\ii(\Omega,\rd\rho)$;

\noindent $(ii)$ for $n\geq2$, $\{c_n<\ii\}$ is an open subset of $\Omega$ on which $c_n$ is continuous;

\noindent $(iii)$ $c_{n+1}(x_1,...,x_{n+1})\geq c_n(x_1,...,x_n)-A$ for some $A\in\R$, all $n\geq0$ and all $x_1,...,x_{n+1}\in\Omega$.

\noindent Then we have
$\cD(\rho)=\widetilde\cD(\rho)=\cC(\rho).$
If $\cC((1+\eps)\rho) < +\infty$ for some $\eps>0$. Then there exists $(\beta^*,\varphi^*)\in \R\times L^\ii(\Omega,\rd\rho)$ which is optimal for $\widetilde\cD(\rho)$ in~\eqref{eq:dualinf}. For any optimizer $\bP^*=(\bP_n^*)_{n\geq0}$ for $\cC(\rho)$, we have
\begin{equation}
 \beta^*+\sum_{j=1}^n\phi^*(x_j)= c_n(x_1,...,x_n) \quad \bP^*_n\text{--a.e., for all $n\geq0$.}
 \label{eq:optimality_dual}
\end{equation}
\end{theorem}

The monotonicity condition $(iii)$ on $c_n$ appeared before in~\eqref{eq:condition_wlsc} (without the constant $A$). The assumption $\cC((1+\eps)\rho) < +\infty$ was used in Theorem~\ref{thm:truncation} and is inspired from~\cite{ChaCha-84}, which also dealt with the existence of optimal potentials. The same assumption was recently used in~\cite[Thm.~4.2]{BouButChaPas-21}.

We do not expect that the continuity assumption $(ii)$ for $c_n$ is at all necessary (up to changing the dual, see Remark~\ref{rmk:dualitesci} below), but it is a reasonable assumption for applications. It is used in our proof to simplify some measure-theoretic technicalities. Later we will even ask that $c_n$ is differentiable.

\begin{remark}\label{rmk:dualitesci}
To obtain a similar result with weaker regularity assumptions on the cost, we would need to further relax the dual problem, for instance by enforcing the inequality a.e.~with respect to every grand-canonical plan:
\begin{multline}
\label{eq:dualinf2}
\overline{\cD}(\rho):=\sup\bigg\{\beta+\int_\Omega \phi(x)\,{\rm d}\rho(x)\ :\ \beta\leq c_0,\quad \phi\in L^{\infty}(\rd\rho),\\ \beta+\sum_{j=1}^n\phi(x_j)\leq c_n(x_1,...,x_n) \quad \bP_n\text{--a.e.}\quad \forall n\geq1, \; \forall \bP \in \PiGC(\rho)\bigg\}.
\end{multline}
In fact let us consider the costs $c_n(x_1,\ldots, x_n)= 1-\1_A$ where $A=\{x_1=x_2=\ldots=x_n\}$, which does not satisfy $(ii)$. Then we have $\cC(\rho)=\cD(\rho)=0$ while if $\rho$ has no atoms $\widetilde\cD(\rho)=1$ since $c_n =1 $ $\rho^{\otimes n}$--a.e.
 \end{remark}

\begin{proof}[Proof of Theorem~\ref{thm:existence_dual}]
Upon changing $c_n$ into $c_n+An$ and $\phi$ into $\phi+A$, we can always assume that $A=0$.
First we make some comments on the constraint
\begin{equation}
 \beta+\sum_{j=1}^n\phi(x_j)\leq c_n(x_1,...,x_n)
 \label{eq:constraint_interpretation}
\end{equation}
which we have imposed in the definition of $\widetilde\cD(\rho)$, for $\phi\in L^\ii(\rd\rho)$. Let $B_\phi$ be the set of the $\rho$--Lebesgue points of $\phi$ and $B_{c_1}$ that of $c_1$. Let finally $B:=B_\phi\cap B_{c_1}$. Then we have $\rho(\Omega\setminus B)=0$. Let finally $\bar x_1,...,\bar x_n\in B$. If $c_n(\bar x_1,...,\bar x_n)=+\ii$, then the inequality~\eqref{eq:constraint_interpretation} of course holds at those points. If $c_n(\bar x_1,...,\bar x_n)<+\ii$, then $c_n$ is bounded in a neighborhood and we can integrate~\eqref{eq:constraint_interpretation} over $\otimes_{j=1}^n\1_{B_r(\bar x_j)}/\rho(B_r(\bar x_j))\,\rd\rho^{\otimes n}$. After passing to the limit $r\to0$ using the continuity of $c_n$ on the open set $\{c_n<\ii\}$, we obtain that~\eqref{eq:constraint_interpretation} holds at $\bar x_1,...,\bar x_n$. This proves that, for all $n\geq1$,~\eqref{eq:constraint_interpretation} is valid on the particular product set $B^n$, which has full $\rho^{\otimes n}$--measure.

Consider now any $\bP=(\bP_n)_{n\geq0}\in\PiGC(\rho)$ such that $\bP(\bc)<\ii$. We have
$$\bP_n(\Omega^n)\leq \bP_n(B^n)+n\bP_n\big((\Omega\setminus B)\times\Omega^{n-1}\big)=\bP_n(B^n)+\rho_{\bP_n}(\Omega\setminus B)=\bP_n(B^n)$$
since $\rho_{\bP_n}\leq \rho$ and $\rho(\Omega\setminus B)=0$. Hence $B^n$ also has full $\bP_n$--measure and thus~\eqref{eq:constraint_interpretation} also holds $\bP_n$-almost everywhere. Integrating over $\bP_n$ and summing over $n$, we find
$\beta+\int_{\R^d}\phi\,\rd\rho\leq \bP(\bc)$.
After minimizing over $\bP$, this proves that $\widetilde\cD(\rho)\leq\cC(\rho)$ and since it is clear that $\widetilde\cD(\rho)\geq \cD(\rho)$, we conclude from Theorem~\ref{thm:duality} that there is equality.

Next we prove that we can restrict the supremum in $\widetilde \cD(\rho)$ to non-negative functions $\phi$. Let $\phi\in L^\ii(\Omega,\rd\rho)$ satisfying the constraint~\eqref{eq:constraint_interpretation} and denote $\phi_+=\max\{\phi,0\}$ its positive part. We work again on the set $B^n$ of Lebesgue points introduced before. The claim is that $\phi_+$ also satisfies the constraint~\eqref{eq:constraint_interpretation} on $B^n$. In fact, if for instance $\phi(x_1),...,\phi(x_k)\geq0$ whereas $\phi(x_{k+1}),...,\phi(x_n)<0$, we have
$$\sum_{j=1}^n\phi_+(x_j)=\sum_{j=1}^k\phi(x_j)\leq c_k(x_1,...,x_k)\leq c_n(x_1,...,x_n)$$
due to our monotonicity assumption $(iii)$ on $c_n$ (recall that $A=0$). Since $\int\phi_+\,\rd\rho\geq \int\phi\,\rd\rho$ this proves that we can restrict the supremum to non-negative potentials $\phi$'s in the dual problem $\widetilde\cD(\rho)$.

Next we assume that $\cC((1+\eps)\rho)<\ii$ and consider a maximizing sequence $(\beta_k,\phi_k)$, that is,
\begin{equation}
\lim_{k\to\ii}\left(\beta_k+\int\phi_k\,\rd\rho\right)=\widetilde\cD(\rho)=\cC(\rho)<\ii.
 \label{eq:maximizing_sequence}
\end{equation}
We will prove that $(\beta_k,\phi_k)$ is bounded in $\R\times L^\ii(\Omega,\rd\rho)$, following an argument from~\cite{ChaChaLie-84,ChaCha-84}. The constraint~\eqref{eq:constraint_interpretation} at $n=0$ gives $\beta_k\leq c_0$. To see that $\beta_k$ is bounded from below, we use $(\beta_k,\phi_k)$ as a competitor for the problem $\widetilde\cD((1+\eps)\rho)=\cC((1+\eps)\rho)$ and get
\[\beta_k +(1+\eps) \int\phi_k\,\rd\rho \leq \widetilde \cD\big((1+\eps)\rho\big).
\]
This gives
$$-\eps\beta_k\leq \widetilde \cD((1+\eps)\rho)-(1+\eps)\left(\beta_k +\int\phi_k\,\rd\rho\right)=\widetilde \cD((1+\eps)\rho)-(1+\eps)\widetilde \cD(\rho)+o(1)$$
and shows that $\beta_k$ is bounded. Finally, the constraint~\eqref{eq:constraint_interpretation} for $n=1$ provides $0\leq \phi_k\leq \|c_1\|_{L^\ii(\rd\rho)}-\beta_k$ $\rho$--a.e., and thus $\phi_k$ is bounded in $L^\ii(\Omega,\rd\rho)$.

After extraction of a subsequence, we can assume that $\beta_k\to\beta$ and $\phi_k\wto \phi$ weakly--$\ast$ in $L^\ii(\Omega,\rd\rho)$. Then
$$\tilde\cD(\rho)=\lim_{k\to\ii}\left(\beta_k +\int\phi_k\,\rd\rho\right)=\beta +\int\phi\,\rd\rho$$
and it only remains to show that $\phi$ satisfies the constraint~\eqref{eq:constraint_interpretation}. The argument goes as before, working on the set $B$ constructed from the Lebesgue points of $\phi$ and integrating again the constraint for $(\beta_k,\phi_k)$ over the measure $\otimes_{j=1}^n\1_{B_r(\bar x_j)}/\rho(B_r(\bar x_j))\,\rd\rho^{\otimes n}$. One takes first $k\to\ii$ and then $r\to0$.

Finally, if $\bP^*=(\bP^*_n)_{n\geq0}$ and $(\beta^*,\phi^*)$ are optimizers for $\cC(\rho)$ and $\widetilde\cD(\rho)$ respectively, then we have
$$\sum_{n\geq0}\int \Big(c_n-\beta^*-\sum_{j=1}^n\phi^*(x_j)\Big)\,\rd\bP^*_n=\bP(\bc)-\beta^*-\int\phi^*\,\rd\rho=\cC(\rho)-\widetilde\cD(\rho)=0.$$
The function in the parenthesis is non-negative $\bP_n^*$--a.e.~due to the constraint~\eqref{eq:constraint_interpretation} being valid on a set of full $\bP_n^*$--measure. This shows~\eqref{eq:optimality_dual}.
\end{proof}

In the following result we improve the regularity of the Kantorovich potential $\phi$, in the case that $\bc$ is a pairwise cost $c_2$ which satisfies the following regularity and growth assumption
\begin{equation}\label{eqn:assgro}
|\nabla_x c_2(x,y) | \leq F\big( c_2(x,y)\big),\ \text{$F$ convex-increasing, such that $F(0)=0$}.
\end{equation}
This is the case for example if $c_2(x,y)=\frac 1{|x-y|^s}$ choosing $F(t)=s\, t^{\frac{s+1}s}$.

\begin{theorem}[Existence of a Lipschitz dual potential]\label{thm:existence_dual_Lip}
Assume that $\bc=(c_n)_{n\geq0}$ is a pairwise cost as in~\eqref{eq:pairwise2}, where $c_2$ is positive, diverges on the diagonal, $\lim_{x\to y}c_2(x,y)=+\ii=c_2(x,x)$, is finite and differentiable on $\{x \neq y\}$, and satisfies~\eqref{eqn:assgro}. We also assume that $c_2$ is bounded outside of the origin, namely
$\sup_{|x-y|\geq r}c_2(x,y)<\ii$ for all $r>0$. Let $\rho$ be a nonnegative measure such that there exists $r>0$ with $\kappa:=\sup_{x\in\Omega}\rho\big(B_r(x)\big)<1$. Then for every optimal dual pair $(\beta^*,\varphi^*)$ for $\widetilde\cD(\rho)$, we have that $\varphi^*$ coincides with an $L$-Lipschitz function $\rho$--almost everywhere, where $L$ depends only on $\rho$. In particular if ${\rm supp} (\rho)=\R^d$, $\phi \in C_b^0(\R^d)$ and it is also an optimal potential for $\cD(\rho)$.
\end{theorem}

\begin{proof}
We are in the setting of Theorem~\ref{thm:existence_dual} and can consider an optimal pair $(\beta^*,\phi^*)$ for $\widetilde\cD(\rho)$. We work on the same set $B$ of Lebesgue points as in the proof of the latter result. Then we have
\[ \begin{cases} \beta^*+\phi^*(x_1) + \phi^*(y_2) + \ldots + \phi^*(y_n) \leq c_n(x_1, y_2, \ldots, y_n) \\
\beta^*+\phi^*(y_1) + \phi^*(x_2) + \ldots + \phi^*(x_m) \leq c_m(y_1, x_2 , \ldots, x_m) \end{cases}\]
on $B^{n}\times B^{m}$. Let $\bP^* \in \PiGC(\rho)$ be an optimal transport plan for $\cC(\rho)$. Thanks to the optimality condition~\eqref{eq:optimality_dual}, we have
\[ \begin{cases}\beta^*+ \phi^*(y_1) + \phi^*(y_2) + \ldots + \phi^*(y_n) = c_n(y_1, y_2, \ldots, y_n) \\
\beta^*+\phi^*(x_1) + \phi^*(x_2) + \ldots + \phi^*(x_m) = c_m(x_1, x_2 , \ldots, x_m) \end{cases} \ \bP_m^* \otimes \bP_n^* \text{--a.e. }\]
The two systems of equations give
\begin{equation}\label{eq:ineqduality}
A^{\bx}(x_1)-A^{\bx}(y_1)\leq \phi^*(x_1) -\phi^*(y_1)\leq A^{\by}(x_1)-A^{\by}(y_1) \quad \bP_m^* \otimes \bP_n^* \text{--a.e.,}
\end{equation}
where $A^{\bx}(x):= \sum_{i=2}^n c_2(x,x_i)$. By assumption we know that $M(r):=\sup_{|x-y|\geq r}c_2(x,y)$ is finite, hence by Lemma~\ref{lem:law_large_numbers_interaction} we conclude that $A^{\bx}(x_1) \leq K$ for some $K$. On the other hand, the $x_i$'s cannot be too close by Theorem~\ref{thm:support_asymp_doubling}. Hence $A^{\bx}$ is differentiable in a neighborhood of $x_1$. The assumption~\eqref{eqn:assgro} on $c_2$ implies
\[
|\nabla_x A^{\bx}(x)| \leq  \sum_{i=2}^n F(c_2(x,x_i)) \leq F \left( \sum_{i=2}^n c_2(x,x_i)\right)= F(A^{\bx}(x)).
\]
In particular $A^{\bx}(x)\leq K+1$ in $B_r(x_1)$, where $r$ is uniform and depends only on $K$ and $F$. In the same ball we have also $|\nabla A^{\bx} (x)| \leq F(K+1)$.
Using this information in \eqref{eq:ineqduality} we obtain
\[|\phi^*(x_1)-\phi^*(y_1)|\leq F(K+1) |x_1-y_1|  \]
for $\bP_m^* \otimes \bP_n^*$--a.e.~$(\bx,\by)$ with $|x_1-y_1|<r$.
We deduce that $\phi$ coincides $\rho$--almost everywhere with an $L$-Lipschitz function, where $L= \max \{ F(K+1) , 2\|\phi\|_{L^\infty(\rd\rho)}/r\}$.
\end{proof}

\section{The 1D problem}\label{sec:1D}

In this section we extend the results of~\cite{ColPasMar-15} to the grand canonical framework, confirming in particular the shape of the optimal plans considered in~\cite{MirSeiGor-13}. In short, we prove that for $n < \rho(\R) < n+1$ we have that ${\rm supp}(\bP) = \{ n , n+1\}$, whereas for $\rho(\R)\in\N$ the grand-canonical optimal solution is actually the canonical one. In addition, we show that the points are ``strictly correlated'' on the support of the optimal plan, that is, is given by a Monge state. When the mass of $\rho$ is not an integer, only one point is removed in a some region of the support.

The running hypothesis in this section is that $\bc=(c_n)_{n\geq0}$ is a pairwise cost as in~\eqref{eq:pairwise2} with $c_2(x,y)$ satisfying
\begin{equation}
\label{eqn:Hp1D}
c_2(x,y)= w(|x-y|),\; w: \R \to [0, +\infty] \text{ convex and decreasing.}
\end{equation}
First, we have to introduce the Monge optimal plan. For any atomless $\rho \in \mathcal{M}(\R)$ let us consider the unique $n \in \N$ and $\eta \in [0,1)$ such that $\rho(\R)=n+\eta$. We then split $\R$ into consecutive intervals where $\rho$ has the alternating masses $\eta$ and $1-\eta$:
$x_0=-\infty \leq x_1 \leq x_2 \leq \ldots \leq x_{2n} \leq x_{2n+1}=+\infty$
where
$$\rho\big((x_{2i},x_{2i+1})\big)=\eta,\qquad \rho\big((x_{2i+1},x_{2i+2})\big)=1-\eta,\qquad\forall i=0, \ldots, n-1.$$
Of course, we  automatically obtain $\rho((x_{2n},x_{2n+1}))=\eta$ too. We now define $T:(x_0,x_{2n-1}) \to (x_2, x_{2n+1})$ as the $\rho$-a.e. unique increasing function such that $\rho((x,T(x)))=1$. We have $T((x_i,x_{i+1})) \subseteq (x_{i+2}, x_{i+3})$. Iteratively we then define $T_1(x)=T(x)$ and $T_{i+1}(x)=T(T_i(x))$. Next we define $\bP^*$ by
\begin{equation}
 \bP^*_i=\begin{cases}
0&\text{for $i \notin \{n,n+1\}$},\\
\text{Sym.}\big[(\text{Id}, T_1, T_2, \ldots, T_{n-1} )_{\sharp}\rho|_{(x_1,x_2)}\big]&\text{for $i=n$},\\
\text{Sym.}\big[(\text{Id}, T_1, T_2, \ldots, T_n )_{\sharp}\rho|_{(x_0,x_1)}\big]&\text{for $i=n+1$.}
\end{cases}
 \label{eq:Monge_1D}
\end{equation}
We have $\bP^* \in \Pi_{GC}(\rho)$. If $\rho(\R)\in\N$, then $\bP^*$ is the canonical optimal plan considered in~\cite{ColPasMar-15}. Note that $\bP^*$ does not depend on the cost, it is only a function of the density $\rho$.

\begin{theorem}\label{thm:main1D}
Let $\Omega= \R$ and $\bc=(c_n)_{n\geq0}$ be a pairwise cost  as in~\eqref{eq:pairwise2} with $c_2$ satisfying~\eqref{eqn:Hp1D}. Let us consider $\rho \in \mathcal{M}(\R)$ and let $n$ be such that $n \leq \rho(\R) < n+1$. If $\mathcal{C}(\rho)< \infty$ and $c_2>0$, any optimal plan $\bP$ should satisfy ${\rm supp}(\bP) \subseteq \{n,n+1\}$. Moreover the Monge plan $\bP^*$ in~\eqref{eq:Monge_1D} is an optimal plan. If the function $w$ in~\eqref{eqn:Hp1D} is strictly convex, then it is the unique one.
\end{theorem}

In order to prove this result, we first need to generalize~\cite[Prop.~2.4]{ColPasMar-15}, which states that when we split some configurations $\{x_1, \ldots, x_m\}$ into two subsystems, the lowest energy we can get is obtained when the two sets are interlaced.

\begin{lemma}\label{lem:1D} Let $\Omega= \R$ and $\bc=(c_n)_{n\geq0}$ be a pairwise cost  as in~\eqref{eq:pairwise2} with $c_2$ which satisfies \eqref{eqn:Hp1D}. Let $x_1 \leq x_2 \leq \ldots \leq x_m \in \R$ be an ordered $m$-tuple and let $X=(x_1, \ldots, x_m)$. Let moreover $X_o$ and $X_e$ be the sub-vectors of $X$ made with the odd and even coordinates respectively, that is $X_o=(x_1, x_3, \ldots, x_{m_o})$ and $X_e=(x_2, x_4, \ldots, x_{m_e})$. Let us consider $\mathcal{R}=\bigcup_{n \geq 1} \R^n$ the set of collections of points, and the function $R : \mathcal{R} \to \mathcal{M}_+ (\R)$ given by $R(x_1, \ldots, x_k)= \sum_{i=1}^k \delta_{x_i}$. Then, for every $Y \in \R^k, Z \in \R^\ell$ such that $R(Y)+R(Z)=R(X)(=R(X_o)+R(X_e))$, we have that
\begin{equation}\label{eqn:1Dineqnei} c_k(Y)+c_\ell(Z) \geq c_{m_o}(X_o) + c_{m_e}(X_e). \end{equation}
If the function $w$ in~\eqref{eqn:Hp1D} is strictly positive, equality in~\eqref{eqn:1Dineqnei} implies $\{k,\ell\}=\{m_o,m_e\}$. If moreover $w$ is strictly convex, equality holds in \eqref{eqn:1Dineqnei} if and only if $(Y,Z)\in \{ (X_o,X_e), (X_e,X_o)\}$.
\end{lemma}

\begin{proof}
Since $\bc$ is a symmetric cost, we can assume without loss of generality that the coordinates of $Y=(y_1, \ldots, y_k)$ and $Z=(z_1, \ldots, z_\ell)$ are increasingly ordered. We then divide the contribution of the interaction energy into $1$-neighbors, $2$-neighbors and so on:
$$c_k(Y)= \sum_{n=1}^{k-1} \sum_{i=1}^{k-n} c_2(y_i,y_{i+n}) = \sum_{n=1}^{k-1} c^n_k(Y),\qquad c^n_k(Y):=\sum_{i=1}^{k-n} c_2(y_i,y_{i+n}).$$
We claim that $c^n_k(Y)+c^n_\ell(Z) \geq c^n_{m_o}(X_o)+c^n_{m_e}(X_e)$ for every $n\geq 1$. Then the lemma follows by summing up these inequalities.
We know that the number of terms in $c^n_k(Y)+c^n_\ell(Z)$ is $k_n=(k-n)_+ + (\ell-n)_+$ which is always greater than the number of terms on the right hand side $k_n^*=(m_o-n)_+ + (m_e-n)_+$. Let us now consider the collection of \emph{left} $n$-neighbors and \emph{right} $n$-neighbors for $Y$ and $Z$, and let us order  according to their indexes in $x$:
$$L_n=\{y_1, \ldots, y_{(k-n)_+}, z_1, \ldots, z_{(\ell-n)_+}\}=\{ x_{i_1} , x_{i_2} ,\ldots, x_{i_{k_n}}\},$$
$$R_n=\{y_{n+1}, \ldots, y_{k}, z_{n+1}, \ldots, z_{\ell}\}=\{ x_{j_1} , x_{j_2} ,\ldots, x_{j_{k_n}}\},$$
where in $L_n$ and $R_n$ we list the points with multiplicities.
Of course we will have that $i_1 \geq 1$, $i_2 \geq 2$, \ldots $i_{k_n} \geq k_n$ and similarly $j_{k_n} \leq m, j_{k_n -1} \leq m-1, \ldots, j_1 \leq m-k_n+1$. We now consider the cost function $c(s,t)= w(s-t)$. Since $w$ is convex, classical results in $1$-dimensional optimal transport with cost function $c$ and measures $\mu= \sum_{r=1}^{k_n} \delta_{x_{i_r}}$, $\nu=\sum_{i=1}^{k_n} \delta_{x_{j_r}}$, yields that $c^n_k(Y)+c^n_\ell(Z) \geq \sum_{r=1}^{k_n} c(x_{i_r}, x_{j_r})$. Now we can use the fact that $w$ is decreasing to show that
$c(x_{i_r}, x_{j_r}) \geq c(x_{r}, x_{m-k_n+r})\geq c(x_{r}, x_{m-k^*_n+r}).$
Summing up these inequalities we finally have
\begin{align*}
c^n_k(Y)+c^n_l(Z) &\geq  \sum_{r=1}^{k_n} c(x_{i_r}, x_{j_r}) \geq  \sum_{r=1}^{k^*_n} c(x_{i_r}, x_{j_r}) \\
& \geq  \sum_{r=1}^{k_n} c(x_r, x_{m-k^*_nn+r}) = c^n_{m_o}(X_o)+c^n_{m_e}(X_e).
\end{align*}
As for the equality cases, notice that if $w >0$ then we need to have $k_n=k^*n$ for every $n$ (otherwise the second inequality would be strict), which happens only if $\{\ell,k\}=\{n_e,n_o\}$.
Moreover similarly to \cite[Lemmas~3.4 and~3.5]{ColPasMar-15}, we can see that whenever $w$ is strictly convex, the equality case happens only when $Y=X_e$ and $Z=X_o$ or vice-versa.
\end{proof}

\begin{proof}[Proof of Theorem~\ref{thm:main1D}]
We start by observing that if $\cC(\rho)< \infty$ we have that any optimal plan $\bP$ is $c$-monotone in the sense of Lemma \ref{lem:cmonot}. However, for every $Y \in {\rm supp}(\bP_k)$ and $Z \in {\rm supp} (\bP_\ell)$, we can find $X \in \R^{k+\ell}$ with ordered coordinates and $I \subseteq \{1, \ldots, k\}$, $J \subseteq \{1, \ldots, \ell\}$ such that $X_o=(Y_I,Z_J)$ and $X_e=(Y_{I^c},Z_{J^c})$. Lemma \ref{lem:cmonot} gives $c_k(Y)+c_\ell(Z) \leq c_{m_o}(X_o)+c_{m_e}(X_e)$ and Lemma \ref{lem:1D} provides the reverse inequality. Thus there is equality. From the equality conditions in Lemma \ref{lem:1D} we then obtain that $\{k,\ell\}=\{m_e,m_o\}$, in particular $|k-\ell|\leq 1$. This means that ${\rm supp}(\bP)$ can only have at most two elements, which are next to each other. Therefore ${\rm supp}(\bP)\subseteq\{n,n+1\}$, with equality if and only if $\rho(\R) \notin \N$. This observation and the mass condition ensure also that if $\rho(\R)\in \N$ then necessary ${\rm supp}(\bP)=\{ \rho(\R)\}$, that is, every grand-canonical optimal plan is actually a canonical (optimal) plan. If $w$ is strictly convex we can follow the same reasoning as in the proof of \cite[Thm.~1.1]{ColPasMar-15} to prove that the Monge state $\bP^*$ in~\eqref{eq:Monge_1D} is the unique optimal plan. By an approximation procedure, we eventually can conclude that $\bP^*$ is optimal also in the case where $w$ is not strictly convex nor strictly positive. For instance we can choose $w_{\eps}(t)=\phi(t)+ \epsilon \cdot e^{-t}$ and then let $\epsilon\to 0$.
\end{proof}

\begin{remark}[Local optimality in 1D]
In Section~\ref{sec:localization}, we have discussed the concept of localization and mentioned that, in general, some information is lost under this procedure. In this respect, the 1D convex case is very special. It satisfies the unusual property that any subsystem is automatically optimal for its corresponding density. Conversely, there is a natural way to glue two separate optimal subsystems and still obtain an optimal plan for the union.
To be more precise, let $\Omega_1, \Omega_2$ be two disjoint intervals in $\mathbb{R}$, and let $\rho_1, \rho_2$ be such that $\rho_i \in \mathcal{M}(\Omega_i)$ for $i=1,2$. Let us then consider $\rho=\rho_1+\rho_2$. Let $\bP^*$ be the optimal grand-canonical Monge plan~\eqref{eq:Monge_1D} for $\rho$ and $\bP_i^*$ the ones for $\rho_i$. A calculation shows that $\bP^*_{|\Omega_i} = \bP^*_i$, that is, localizing optimal grand-canonical plans gives rise to optimal plans again. Conversely, if we start from two optimal plans $\bP^*_i$ and assume strict convexity of $w$, we know that $\bP^*$ is the unique optimal plan and it must have $\bP^*_i$ as localization to $\Omega_i$. This gives a natural and unique way of gluing the two plans $\bP^*_1$ and $\bP^*_2$ into a plan which is still optimal.

Another feature may be of interest. Suppose $\rho_i \in \mathcal{M}(\Omega_i)$ for $i=1,2$ with $\Omega_i$ bounded, and let $\rho^R=\rho_1+\rho_2(\cdot-R)$ obtained by translating $\Omega_2$ far away. If $\rho_1(\Omega_1)+\rho_2(\Omega_2)=N \in \mathbb{N}$, we can consider a \emph{canonical} optimizer $\bP^R$ problem for the measure $\rho^R$. In the limit $R \to \infty$ one could guess that the two localized plans $\bP^R_{|\Omega_1}$ and $\bP^R_{|\Omega_2+R}(\cdot+R)$ will converge to \emph{grand-canonical} optimal plans $\bP^*_i$ on $\Omega_i$. But this cannot be true in general. Some long range correlations have to persist in the limit in some cases. In fact, the canonical plan $\bP^R$ satisfies the property $\bP^R_{|\Omega_1,n}(\Omega_1^n) = \bP^R_{|\Omega_2,N-n} (\Omega_2^{N-n})$, since the probability that $n$ points be in $\Omega_1$ coincides with that to have $N-n$ points in $\Omega_2$. Thus, if for our grand-canonical optimal plans we have $\bP_{1,n}(\Omega_1^n) \neq \bP_{2,N-n} (\Omega_2^{N-n})$, the expected limit cannot hold. As also follows from the previous discussion, one can verify that our Monge plan~\eqref{eq:Monge_1D} does satisfy the matching property $\bP_{1,n}(\Omega_1^n)= \bP_{2,N-n} (\Omega_2^{N-n})$.

Notice that in \cite{MirSeiGor-13} the optimal grand canonical state is defined as a localization of a canonical known state.
\end{remark}

\section{Entropic regularization}\label{sec:entropy}

Here we discuss the entropic regularization of $\cC(\rho)$ which, from a statistical mechanics point of view, amounts to adding temperature in the system.
In the Physics literature, the problem of finding the external potential for which the associated grand-canonical equilibrium classical (Gibbs) state at a temperature $T>0$ has the given density $\rho$, plays an important role in the density functional theory of classical inhomogeneous systems~\cite{Evans-79,Evans-92}. Developed in the 1960-70s~\cite{MorHir-61,DeDominicis-62,StiBuf-62,LebPer-63,Evans-79}, this theory is for instance used to describe the liquid-gas interface~\cite{EbnSaaStr-76,YanFleGib-76}, or the solid-fluid interface~\cite{RamYus-77,Baus-87}. Most Physics papers on the subject use the grand-canonical framework defined in this section. The canonical model, which is just the entropic regularization of the multi-marginal problem $\cC_N(\rho)$~\cite{BenCaCuNenPey-15,CarDuvPeySch-17,CarLab-20,MarGer-20}, has been used quite lately, e.g. in~\cite{WhiGonRomVel-00}.

On the mathematical side, this section will rely on the fundamental works of Chayes, Chayes and Lieb~\cite{ChaChaLie-84,ChaCha-84} in 1984 who, to our knowledge, were the first to prove the existence of the dual potential $V$ for such systems with entropy, both in the $N$-marginal and grand-canonical cases.

\subsection{Entropy}
In order to have a correct behavior it is necessary to place the right `Boltzmann $n!$' in the definition of the entropy, an issue which does not occur for the multi-marginal problem at fixed $n$.  The entropy of a grand-canonical probability $\bP=(\bP_n)_{n\geq0}$ is defined by~\cite{RobRue-67,ChaChaLie-84,ChaCha-84}
\begin{equation}
\cS(\bP):=-\sum_{n\geq0}\int_{\Omega^n}\bP_n\log\big(n!\bP_n\big),
 \label{eq:def_entropy}
\end{equation}
with the condition that each $\bP_n$ is absolutely continuous with respect to the Lebesgue measure on $\Omega^n$. Otherwise, $\cS(\bP)$ is taken equal to $-\ii$. Note the minus sign in the definition (the usual convention in statistical mechanics).

The reason for adding the $n!$ is the following. Our system describing independent agents, we should in principle not work on $\Omega^n$ but, rather, in the simplex of volume $1/n!$ obtained by moding out the permutations. It is therefore more natural to consider $\bQ_n:=n!\,\bP_n$ which is such that
$$\sum_{n\geq0}\bQ_n\big(\Omega^n/\gS_n\big)=\sum_{n\geq0}\frac{\bQ_n\left(\Omega^n\right)}{n!}=\sum_{n\geq0}\bP_n\left(\Omega^n\right)=1.$$
Then we have simply
$\cS(\bP):=-\sum_{n\geq0}\int_{\Omega^n/\gS_n}\bQ_n\log\big(\bQ_n\big).$

The entropy in~\eqref{eq:def_entropy} has no sign. This is because we have used the Lebesgue measure as a reference and $\Omega$ was not assumed to have a finite measure. If $|\Omega|<\ii$, then we can use
$e^{-|\Omega|}\sum_{n\geq0}\frac{|\Omega|^n}{n!}=1$
and Jensen's inequality to obtain as in~\cite{RobRue-67,Wehrl-78}
\begin{align*}
\cS(\bP)&=-\sum_{n\geq0}\frac1{n!}\int_{\Omega^n}(n!\bP_n)\log\big(n!\bP_n\big)\\
&\leq -e^{|\Omega|}\left(e^{-|\Omega|}\sum_{n\geq0}\bP_n(\Omega^n)\right)\log\left(e^{-|\Omega|}\sum_{n\geq0}\bP_n(\Omega^n)\right)=|\Omega|.
\end{align*}
In other words, we obtain a negative entropy after replacing $n!\bP_n$ by $e^{|\Omega|}n!\bP_n$ in the logarithm. The following is a replacement when $\Omega$ is arbitrary but $\int_\Omega\rho|\log\rho|<\ii$ and uses that the entropy is maximized for Poisson states.

\begin{lemma}[Maximal entropy]\label{lem:max_entropy}
Let $\rho\in L^1(\Omega,\R_+)$ be such that $\int_\Omega \rho|\log\rho|<\ii$. Then we have
\begin{equation}
\cS(\bP)\leq \int_\Omega \rho-\int_{\Omega}\rho\log\rho
\label{eq:max_entropy}
\end{equation}
for every $\bP\in\PiGC(\rho)$. There is equality only for the Poisson state $\bG_\rho=(\bG_{\rho,n})_{n\geq0}$ from Example~\ref{ex:Poisson}.
\end{lemma}

\begin{proof}
A calculation shows that the entropy $\cS(\bG_\rho)$ of the Poisson state is equal to the right side of~\eqref{eq:max_entropy}. Another calculation shows that
\begin{align*}
\cS(\bG_\rho)-\cS(\bP)&=\sum_{n\geq0}\int_{\Omega^n}\bP_n\log\left(\frac{\bP_n}{\bG_{\rho,n}}\right) +\sum_{n\geq0}\int_{\Omega^n}(\bP_n-\bG_{\rho,n})\log(n!\bG_{\rho,n})\\
&=\sum_{n\geq0}\int_{\Omega^n}\bP_n\log\left(\frac{\bP_n}{\bG_{\rho,n}}\right) +\int_\Omega (\rho_\bP-\rho_{\bG_{\rho}})\log\rho\\
&=\sum_{n\geq0}\int_{\Omega^n}\bP_n\log\left(\frac{\bP_n}{\bG_{\rho,n}}\right)=:\cH(\bP,\bG_{\rho}),
\end{align*}
for every $\bP\in\PiGC(\rho)$. The relative entropy $\cH(\bP,\bG_\rho)$ of the two grand-canonical probabilities is non-negative and vanishes only when $\bP=\bG_{\rho}$.
\end{proof}

In the case of the harmonic cost, it is natural to assume that $\rho$ has a finite second moment, $\int_{\R^d}|x|^2\rho<\ii$. Then $\int_{\R^d}\rho\log\rho$ is well defined in $\R\cup\{+\ii\}$, without the absolute value in the integrand~\cite{CarDuvPeySch-17}. This is because we can use the Gaussian $g(x):=(\pi)^{-d/2}e^{-|x|^2}$ as reference and rewrite
\begin{multline}
 \int_{\R^d}\rho\log\rho=\rho(\R^d)\cH\left(\frac{\rho}{\rho(\R^d)},g\right)-\frac{d}{2}\log\pi \rho(\R^d)\\-\int_{\R^d}|x|^2\rho+\rho(\R^d)\log\rho(\R^d).
 \label{eq:entropy_gaussian}
\end{multline}
The first term on the right is the relative entropy (also called Kullback-Leibler divergence~\cite{KullLei-51}) $\cH(f,g)=\int_{\R^d}f\log(f/g)$ of $f,g$ which is well defined in $[0,+\ii]$. The proof of Lemma~\ref{lem:max_entropy} is a bit in the same spirit. Under the assumptions on $\rho$ in the statement, our entropy equals
\begin{equation}
\cS(\bP)=\int_\Omega\rho-\int_\Omega\rho\log\rho-\cH(\bP,\bG_{\rho})
\label{eq:formula_entropy_Poisson}
\end{equation}
for every $\bP\in\PiGC(\rho)$, where $\bG_{\rho}$ is the Poisson state and $\cH$ is the (grand canonical) relative entropy, which is non-negative and vanishes only at $\bP=\bG_{\rho}$. In the following it will be convenient to remove the two constants and replace $\cS(\bP)$ by $-\cH(\bP,\bG_{\rho})$ throughout. This corresponds to taking the Poisson state $\bG_\rho$ as reference and was called a renormalization procedure in~\cite[Eq.~(2.12)]{ChaCha-84}. This will in fact allow us to remove the condition $\int\rho|\log\rho|<\ii$.

Next we derive a useful estimate on the growth of the number of particles for states with finite entropy relative to $\bG_\rho$.

\begin{lemma}[Estimate on the growth in $n$]\label{lem:large_n_control}
Let $\rho$ be a finite non-negative measure over $\Omega$. Then, we have for every $\bP\in\PiGC(\rho)$
\begin{equation}
 \sum_{n\geq0}\log(n!)\,\bP_n(\Omega^n)\leq \cH(\bP,\bG_\rho)+\log2+\rho(\Omega)\log\rho(\Omega).
 \label{eq:estim_n_entropy}
\end{equation}
\end{lemma}

The result says that when $\rho(\Omega)=\sum_{n\geq0}n\,\bP_n(\Omega^n)$ and $\cH(\bP,\bG_\rho)$ are both finite, we have the slightly better control $\sum_{n\geq0}n\log(n)\,\bP_n(\Omega^n)\leq C$ for large $n$.

\begin{proof}
We use that
$$\int f\log\frac{f}{g}=\left(\int f\right)\log\frac{\int f}{\int g}+\left(\int f\right)\cH\left(\frac{f}{\int f}\,,\, \frac{g}{\int g}\right)\geq \left(\int f\right)\log\frac{\int f}{\int g}$$
for any non-negative functions $f,g$. We obtain
\begin{align*}
\cH(\bP,\bG_\rho)&=\sum_{n\geq0}\int_{\Omega^n}\bP_n\log\left(\frac{\bP_n n!e^{\rho(\Omega)}}{\rho^{\otimes n}}\right)\\
&\geq \sum_{n\geq0}\bP_n(\Omega^n)\log\left(\frac{\bP_n(\Omega^n) n!e^{\rho(\Omega)}}{\rho(\Omega)^n}\right)\\
&= \sum_{n\geq0}\bP_n(\Omega^n)\log\left(\bP_n(\Omega^n) 2^{n+1}\right)+\sum_{n\geq0}\bP_n(\Omega^n)\log(n!)\\
&\qquad-\rho(\Omega)\log\rho(\Omega)-\log2+(1-\log2)\rho(\Omega).
\end{align*}
The first term on the right is the relative entropy of the two discrete probabilities $\big(\bP_n(\Omega^n)\big)_{n\geq0}$ and $(2^{-1-n})_{n\geq0}$, hence it is non-negative.
\end{proof}

\subsection{Grand-canonical problem at $T>0$}
For any finite non-negative measure $\rho$ on $\Omega$, the grand-canonical optimization problem reads
\begin{equation}
\boxed{\cF_T(\rho):=\inf_{\bP\in\PiGC(\rho)}\Big\{\bP(\bc)+T\cH(\bP,\bG_{\rho})\Big\}.}
\label{eq:def_cF_T}
\end{equation}
We will assume that the cost $\bc$ is stable, $c_n\geq -A-Bn$, so that $\bP(\bc)\geq -A-B\rho(\Omega)$. The infimum is thus finite for one $T>0$ if and only if there exists a $\bP\in\PiGC(\rho)$ such that $\bP(\bc)$ and $\cH(\bP,\bG_{\rho})$ are simultaneously finite. In this case $\cF_T(\rho)$ is actually finite for all $T>0$. Otherwise, we have $\cF_T(\rho)=+\ii$ for all $T>0$. For instance we can simply assume that the cost  $\bG_{\rho}(\bc)$ is finite for the Poisson state. Note also that $\cF_T(\rho)$ is non-decreasing and concave in $T$. The following is a consequence of the known properties of the relative entropy $\cH$.

\begin{theorem}[Existence for the positive temperature problem]\label{thm:existence_temp}
Let $\Omega\subset\R^d$ be any Borel set. Let $\bc=(c_n)_{n \geq 0} $ be a stable family of lower semi-continuous costs. Let $\rho$ be a finite measure such that $\cF_T(\rho)<\ii$ for one (hence all) $T>0$.

\medskip

\noindent $(i)$ $\cF_T(\rho)$ admits a \emph{unique minimizer $\bP^{(T)}$} for all $T>0$.

\medskip

\noindent $(ii)$ In the limit $T\to0^+$ we have
\begin{equation}
\lim_{T\to0^+}\cF_T(\rho)=\inf_{\substack{\bP\in\PiGC(\rho)\\ \cH(\bP,\bG_\rho)<\ii}}\bP(\bc)=:\cF_0(\rho).
\label{eq:liimt_zero_temp}
\end{equation}
If $\bc=(c_n)_{n\geq0}$ is superstable and the right side is equal to $\cC(\rho)$, then $\bP^{(T)}$ converges, up to subsequences, to a minimizer $\bP^*\in\PiGC(\rho)$ for $\cC(\rho)$, in the sense that  $\bP^{(T)}_n\wto\bP^*_n$ narrowly for all $n\geq0$.

\medskip

\noindent $(iii)$ If $\bG_\rho(\bc)<\ii$, then in the limit $T\to\ii$ we have
$$\lim_{T\to\ii}\cF_T(\rho)=\bG_\rho(\bc),$$
$$\lim_{T\to\ii}T\,\cH(\bP^{(T)},\bG_\rho)=\lim_{T\to\ii} \sum_{n\geq0}\int_{\Omega^n}|\bP^{(T)}_n-\bG_{\rho,_n}|=0.$$
\end{theorem}

We emphasize that for the existence of $\bP^{(T)}$ the cost $\bc=(c_n)_{n\geq0}$ is only assumed to be \emph{stable} and not \emph{superstable} as we required at $T=0$ in Theorem~\ref{thm:existence}. This is because the entropy gives us an additional control on the large-$n$ behavior due to Lemma~\ref{lem:large_n_control}.

\begin{proof}
From the stability $c_n\geq -A-Bn$, we have for a minimizing sequence $\bP^k=(\bP^k_n)_{n\geq0}$
$$\cF_T(\rho)+o(1)=\bP^k(\bc)+T\cH(\bP^k,\bG_{\rho})\geq -A-B\rho(\Omega)+T\cH(\bP^k,\bG_{\rho})$$
and therefore $\cH(\bP^k,\bG_{\rho})$ is bounded. By Lemma~\ref{lem:large_n_control} this tells us that
\begin{equation*}
\sum_{n\geq0} n\log(n)\, \bP_n^k(\Omega^n)=\sum_{n \geq 0} \log(n)\,\rho_{\bP_n^k}(\Omega)\leq C
\end{equation*}
and hence
\begin{equation}\label{eqn:estiss2}
\sum_{n \geq M} \rho_{\bP_n^k}(\Omega) \leq  \frac{C}{\log(M)}.
\end{equation}
From the proof of Theorem~\ref{thm:existence} we conclude that there exists a $\bP\in\PiGC(\rho)$ so that, up to extraction of a subsequence,
$\liminf_{k\to\ii}\bP^k(\bc)\geq \bP(\bc)$
and each $\bP^k_n$ converges narrowly to $\bP_n$. We use~\eqref{eqn:estiss2} as a replacement for~\eqref{eqn:estiss} which we had obtained from the superstability. The grand-canonical relative entropy is weakly lower semi-continuous for the narrow convergence,
$\liminf_{k\to\ii}\cH(\bP^k,\bG_\rho)\geq \cH(\bP,\bG_\rho)$
and therefore $\cH(\bP,\bG_\rho)<\ii$. We conclude that $\bP$ is a minimizer for $\cF_T(\rho)$. It is unique because $\bP\mapsto\cH(\bP,\bG_{\rho})$ is strictly convex.

Next we study its behavior when $T$ tends to 0 or infinity. We have the uniform bound $\cF_T(\rho)\geq\cF_0(\rho)$.
On the other hand, let $\bP$ be any state such that $\bP(\bc)$ and $\cH(\bP,\bG_{\rho})$ are both finite. Then we have
$\cF_T(\rho)\leq \bP(\bc)+T\cH(\bP,\bG_\rho).$
Taking first $T\to0$ gives
$\limsup_{T\to0}\cF_T(\rho)\leq  \bP(\bc).$
Optimizing over $\bP$ provides the missing upper bound. We conclude that
$$\lim_{T\to0^+}\bP^{(T)}(\bc)=\cF_0(\rho),\qquad \lim_{T\to0^+}T\,\cH(\bP^{(T)},\bG_\rho)=0.$$
If the cost is superstable, then $\bP^{(T)}$ is tight by the proof of Theorem~\ref{thm:existence} and converges up to a subsequence to some $\bP^*\in\PiGC(\rho)$ with $\cC(\rho)\leq \bP^*(\bc)\leq \cF_0(\rho)$. If $\cC(\rho)=\cF_0(\rho)$ then $\bP^*$ must be an optimizer for  $\cC(\rho)$. If $\cC(\rho)<\cF_0(\rho)$ but $\cH(\bP^*,\bG_\rho)<\ii$ then $\bP^*$ is an optimizer for $\cF_0(\rho)$.

For the limit $T\to\ii$, we simply use that
$$-A-B\rho(\Omega)+T\,\cH(\bP^{(T)},\bG_\rho)\leq \bP^{(T)}(\bc)+T\,\cH(\bP^{(T)},\bG_\rho)\leq \bG_\rho(\bc)$$
where the right side is finite by assumption. This proves that $T\cH(\bP^{(T)},\bG_\rho)$ is bounded, hence $\cH(\bP^{(T)},\bG_\rho)\to0$. Pinsker's inequality~\cite{Pinsker-64,Csiszar-67,Kullback-67}
\begin{equation}
\sum_{n\geq0}\int_{\Omega^n}|\bP_n-\bP'_n|\leq \sqrt{2\cH(\bP,\bP')}
 \label{eq:Pinsker}
\end{equation}
gives the convergence
\begin{equation}
\lim_{T\to\ii} \sum_{n\geq0}\int_{\Omega^n}|\bP^{(T)}_n-\bG_{\rho,_n}|=0.
 \label{eq:CV_L1}
\end{equation}
But then
$\liminf_{T\to\ii}\bP^{(T)}(\bc)\geq \bG_\rho(\bc)$
by the proof of Theorem~\ref{thm:existence}. Therefore $\bP^{(T)}(\bc)$ converges to $\bG_\rho(\bc)$ and $T\,\cH(\bP^{(T)},\bG_\rho)$ tends to 0. In particular,~\eqref{eq:CV_L1} is actually a $o(T^{-1/2})$.
\end{proof}

In most practical examples, minimizers for $\cC(\rho)$ will concentrate on small sets and they will therefore have an infinite relative entropy with respect to the Poisson state $\bG_\rho$. However, we expect that the equality $\cC(\rho)=\cF_0(\rho)$ will hold if these minimizers can be regularized so as to have a finite entropy without generating a too large error in the cost. For the harmonic cost in the canonical case, this was done for instance in~\cite{Leonard-12,CarDuvPeySch-17}. As an illustration, we provide an adaptation of~\cite{CarDuvPeySch-17} for a pairwise repulsive cost under appropriate assumptions on $c_2$ and $\rho$. We need the definition of the two-agent distribution $\rho^{(2)}_\bP$ on $\Omega\times\Omega$ which is given by
$$\rho_\bP^{(2)}(A_1\times A_2)=\sum_{n\geq2}\frac{n(n-1)}{2}\bP_n\big(A_1\times A_2\times(\Omega)^{n-2}\big)$$
and satisfies
$\bP(\bc)=\iint_{\Omega\times\Omega}c_2(x,y)\,{\rm d}\rho^{(2)}_\bP(x,y)$
for a pairwise cost.

\begin{theorem}[Limiting zero-temperature problem for locally bounded pairwise costs]
Let $\Omega\subset\R^d$ be any Borel set. Let $\bc=(c_n)_{n\geq0}$ be a pairwise cost as in~\eqref{eq:pairwise} with a non-negative lower semi-continuous cost $c_2$ on $\Omega\times\Omega$. Let $\rho$ be a finite non-negative measure on $\Omega$ such that $\cC(\rho)<\ii$ and
$\int \log(2+|x|)\,{\rm d}\rho(x)<\ii$.
 Assume in addition that $\cC(\rho)$ admits a minimizer $\bP$ such that
\begin{equation}
\sum_{n\geq0}n^2\,\bP_n(\Omega^n)<\ii.
\label{eq:cond_N2}
\end{equation}
If $c_2$ is uniformly continuous in a neighborhood of the support of $\rho_\bP^{(2)}$, then we have $\cF_0(\rho)=\cC(\rho)$.
\end{theorem}

The logarithm growth condition on $\rho$ will be used to control the negative part of the entropy of a regularization of $\rho$ a bit like in~\eqref{eq:entropy_gaussian}. The condition~\eqref{eq:cond_N2} that the second moment in $n$ of the minimizer $\bP$ is finite amounts to requiring  that $\rho^{(2)}_\bP\in L^1(\Omega\times\Omega)$. In Section~\ref{sec:support} we have given many examples for which minimizers all have a compact support in $n$, and in these cases~\eqref{eq:cond_N2} obviously holds. Finally, in Theorem~\ref{thm:support_asymp_doubling} we have seen that for an asymptotically doubling $c_2$ which diverges on the diagonal,  minimizers for $\cC(\rho)$ are all supported outside of the diagonal. This implies that $\rho^{(2)}$ concentrates on $\{|x-y|\geq\eps\}$ for some $\eps>0$. Thus, if the two-agent cost $c_2$ is regular outside of the diagonal, our last condition is satisfied. For instance, the theorem applies to costs of the form $c_2(x,y)=f(|x-y|)$ with $f$ positive and continuous on $(0,\ii)$ satisfying $f(r)\to+\ii$ when $r\to0^+$. This covers Riesz costs $f(r)=r^{-s}$, $s>0$, in any dimension $d\geq1$.

\begin{proof}
Let $\bP$ be a minimizer for $\cC(\rho)$, satisfying the properties of the statement. We follow~\cite{CarDuvPeySch-17} and cover our space with cubes $(C_z)_{z\in h\Z^d}$ of side length $h$, that is, $C_z=z+[-h/2,h/2)^d$. We then define the block approximation $\bP^h$ of $\bP$ by $\bP^h_0=\bP_0$ and
\begin{align}
\bP^h_n&:=\sum_{\substack{z_1,...,z_n\\ \bP_n(C_{z_1}\times\cdots\times C_{z_n})>0}}
\frac{\bP_n(C_{z_1}\times\cdots\times C_{z_n})}{\rho(C_{z_1})\cdots \rho(C_{z_n})}\rho_{|C_{z_1}}\otimes\cdots\otimes \rho_{|C_{z_n}}\nn\\
&=\rho^{\otimes n}\sum_{\substack{z_1,...,z_n\\ \bP_n(C_{z_1}\times\cdots\times C_{z_n})>0}}\frac{\bP_n(C_{z_1}\times\cdots\times C_{z_n})}{\rho(C_{z_1})\cdots \rho(C_{z_n})}\1_{C_{z_1}}\otimes\cdots\otimes \1_{C_{z_n}}
\label{eq:block_approx}
\end{align}
This is a symmetric measure, since $\bP_n(C_{z_1}\times\cdots\times C_{z_n})$ is symmetric with respect to exchanges of the indices. Note that
$$\rho(C_z)=\sum_{n\geq1}n\,\sum_{z_2,...,z_n}\bP_n(C_z\times C_{z_2}\times\cdots\times C_{z_n})\geq n\bP_n(C_z\times C_{z_2}\times\cdots\times C_{z_n})$$
and therefore none of the $\rho(C_z)$ vanishes in the denominator when $\bP_n(C_{z_1}\times\cdots\times C_{z_n})>0$. A simple calculation shows that
$\rho_{\bP^h}=\rho$ and therefore $\bP^h\in\PiGC(\rho)$. In addition, we have
\begin{equation}
 \rho^{(2)}_{\bP^h}=\rho^{\otimes 2}\sum_{z,z'}\rho_\bP^{(2)}(C_z\times C_{z'})\frac{\1_{C_z}\otimes \1_{C_{z'}}}{\rho(C_z)\rho(C_{z'})}.
 \label{eq:approx_two-particle_block}
\end{equation}
The relative entropy of the block approximation~\eqref{eq:block_approx} equals
\begin{multline*}
\cH(\bP^h,\bG_\rho)= \rho(\Omega)+\bP_0\log\big(\bP_0\big)\\
+\sum_{n\geq1}\sum_{\substack{z_1,...,z_n\\ \bP_n(C_{z_1}\times\cdots\times C_{z_n})>0}}\bP_n(C_{z_1}\times\cdots\times C_{z_n})\log\left(\frac{\bP_n(C_{z_1}\times\cdots\times C_{z_n})}{\rho(C_{z_1})\cdots \rho(C_{z_n})}\right)
\end{multline*}
and since $\log \bP_n(C_{z_1}\times\cdots\times C_{z_n})\leq0$ it can be estimated by
$$\cH(\bP^h,\bG_\rho)\leq \rho(\Omega)+\bP_0\log\big(\bP_0\big)-\sum_{z}\rho(C_z)\log\rho(C_z).$$
To estimate the last term we need some decay on the density $\rho$. We use that
$$\sum_{z}\rho(C_z)\log\left(\frac{N(h)\rho(C_z)}{h^d\rho(\Omega)}(1+|z|)^{d+1}\right)\geq0$$
(this is a relative entropy in $\ell^1$), where
$$N(h):=h^d\sum_{z\in h\Z^d}\frac{1}{(1+|z|)^{d+1}}\underset{h\to0}\sim \int_{\R^d}\frac{{\rm d}x}{(1+|x|)^{d+1}}.$$
Hence
\begin{align*}
-\sum_{z}\rho(C_z)\log\rho(C_z)\leq& \big(\log N(h)-d\log h-\log\rho(\Omega)\big)\rho(\Omega)\\
&\qquad +(d+1)\sum_{z}\rho(C_z)\log(1+|z|).
\end{align*}
For $h$ small enough $\log(1+|z|)\leq \log(2+|x|)$ for $x\in C_z$, and thus
$$\sum_{z}\rho(C_z)\log(1+|z|)\leq \int_\Omega \rho(x)\log(2+|x|)\,{\rm d}x.$$
Therefore $\cH(\bP^h,\bG_\rho)<\ii$. Next we estimate the cost
\begin{align}
\bP^h(\bc)&=\iint_{\Omega\times\Omega}c_2(x,y)\,{\rm d}\rho_{\bP^h}^{(2)}(x,y)\nn\\
&=\sum_{\rho_\bP^{(2)}(C_z\times C_{z'})>0}\frac{\rho_\bP^{(2)}(C_z\times C_{z'})}{\rho(C_z)\rho(C_{z'})}\iint_{C_z\times C_{z'}}c_2(x,y)\rho(x)\,\rho(y)\,{\rm d}x\,{\rm d}y.
\label{eq:cost_block}
\end{align}
The condition $\rho_\bP^{(2)}(C_z\times C_{z'})>0$ means that we work on ${\rm supp}(\rho^{(2)}_\bP)+C_0\times C_0$, a neighborhood of size $h$ of the support of $\rho^{(2)}_\bP$. This is where the cost was assumed to be uniformly continuous. We obtain
\begin{align}
&|\bP^h(\bc)-\bP(\bc)|\nn\\
&\qquad\leq  \sum_{z,z'}\frac{\rho_\bP^{(2)}(C_z\times C_{z'})}{\rho(C_z)\rho(C_{z'})}\iint_{C_z\times C_{z'}}|c_2(x,y)-c_2(z,z')|\,{\rm d}\rho(x)\,{\rm d}\rho(y)\nn\\
&\qquad\qquad+\sum_{z,z'}\iint_{C_z\times C_{z'}}|c_2(x,y)-c_2(z,z')|{\rm d}\rho^{(2)}_{\bP}(x,y)\nn\\
&\qquad \leq2\sum_{\rho_\bP^{(2)}(C_z\times C_{z'})>0}\|c_2-c_2(z,z')\|_{L^\ii(C_z\times C_{z'})}\;\rho^{(2)}_{\bP}(C_z\times C_{z'})\label{eq:estim_cost_block}\\
&\qquad\leq2\rho^{(2)}_{\bP}(\Omega\times\Omega)\;\sup_{\rho_\bP^{(2)}(C_z\times C_{z'})>0}\|c_2-c_2(z,z')\|_{L^\ii(C_z\times C_{z'})}.\nn
\end{align}
We have assumed that $\rho^{(2)}_\bP(\Omega\times\Omega)<\ii$ and the supremum is small under our assumption that $c_2$ is uniformly continuous in a neighborhood of the support of $\rho^{(2)}$.
\end{proof}

\begin{remark}
If $c_2$ grows then we need a further assumption on $\rho^{(2)}$ to control~\eqref{eq:estim_cost_block}. In the case of the harmonic cost $c_2(x,y)=|x-y|^2$, the error in~\eqref{eq:estim_cost_block} can simply be estimated by a constant times
$$h\int_{\Omega\times\Omega}(h+|x-y|)\,{\rm d}\rho^{(2)}_{\bP}(x,y)\leq C\Big(h^2\rho^{(2)}_\bP(\Omega\times\Omega)+\sqrt{h}\bP(\bc)\Big).$$
\end{remark}

\subsection{Duality}
Next we discuss the dual problem. Consider a measurable potential $\psi:\Omega\to\R$ such that
$\int_{\Omega}e^{\frac{\psi(x)}{T}}\,{\rm d}\rho(x)<\ii.$
With $\Psi_n=\sum_{j=1}^n\psi(x_j)$ we define
\begin{equation}
\boxed{F_{T,\rho}(\psi):=\inf_{\bP}\big\{\bP(\bc-\Psi)+T\cH(\bP,\bG_\rho)\big\}=-T\log Z_{T,\rho}(\psi)}
\label{eq:dual_temp}
\end{equation}
with the partition function
\begin{multline}
Z_{T,\rho}(\psi):=e^{-\frac{c_0}{T}-\rho(\Omega)}\\+\sum_{n\geq1}\frac{e^{-\rho(\Omega)}}{n!}\int_{\Omega^n}\exp\left(\frac{-c_n(x_1,...,x_n)+\sum_{j=1}^n\psi(x_j)}{T}\right){\rm d}\rho^{\otimes n}.
\end{multline}
Recalling that $c_n\geq -A-Bn$ we find
\begin{equation}
 Z_{T,\rho}(\psi)\leq \exp\left(\frac{A}T-\rho(\Omega)+e^{\frac{B}{T}}\int_\Omega e^{\frac{\psi}{T}}{\rm d}\rho\right)
 \label{eq:partition_stability}
\end{equation}
and since $Z_{T,\rho}(\psi)\geq e^{-c_0/T-\rho(\Omega)}$, we obtain
\begin{equation}
 -A+T\rho(\Omega)-Te^{\frac{B}{T}}\int_\Omega e^{\frac{\psi}{T}}{\rm d}\rho\leq F_{T,\rho}(\psi)\leq c_0+T\rho(\Omega).
 \label{eq:estim_stability_temp}
\end{equation}
The formula $F_{T,\rho}(\psi)=-T\log Z_{T,\rho}(\psi)$ is well known. The corresponding unique minimizer is the Gibbs state $\bP_\psi$ given by
\begin{equation}
\bP_{\psi,0}=\frac{e^{-\frac{c_0}{T}-\rho(\Omega)}}{Z_{T,\rho}(\psi)},\qquad \bP_{\psi,n}=\frac{e^{\frac{-c_n+\sum_{j=1}^n\psi(x_j)}{T}-\rho(\Omega)}\rho^{\otimes n}}{Z_{T,\rho}(\psi)n!}.
\label{eq:Gibbs_psi}
\end{equation}
To prove this claim, just notice that
$$\bP(\bc)+T\cH(\bP,\bG_\rho)-\bP_\psi(\bc)-T\cH(\bP_\psi,\bG_\rho)=T\cH(\bP,\bP_\psi)$$
is positive and vanishes only at $\bP_\psi$. Next we discuss the existence of a dual potential. The following is a small adaptation of results proved in~\cite{ChaChaLie-84,ChaCha-84}.

\begin{theorem}[Duality at $T>0$~\cite{ChaChaLie-84,ChaCha-84}]\label{thm:duality_temp}
Let $\Omega\subset\R^d$ be any Borel set. Let $\bc=(c_n)_{n \geq 0} $ be a stable family of lower semi-continuous costs with $c_1\in L^\ii(\Omega,\rd\rho)$. Let $\rho$ be a finite measure such that $\cF_T(\rho)<\ii$ for one (hence all) $T>0$. Then we have
\begin{equation}
\cF_T(\rho)=\sup_{\int_\Omega e^{\psi/T}{\rm d}\rho<\ii}\left\{\int_{\Omega}\psi\,{\rm d}\rho+F_T(\psi)\right\}.
\label{eq:duality_temp}
\end{equation}
If there exists $\eps>0$ such that for one (hence all) $T>0$
\begin{equation}
\cF_T\big((1+\eps)\rho\big)<\ii,
\label{eq:condition_dual_potential}
\end{equation}
then there exists a unique potential $\psi$ realizing the supremum in~\eqref{eq:duality_temp} and it is such that $\psi\in L^1(\Omega,{\rm d}\rho)$. The corresponding unique minimizer $\bP^{(T)}$ from Theorem~\ref{thm:existence_temp} is equal to the Gibbs state $\bP_\psi$ in~\eqref{eq:Gibbs_psi}.
\end{theorem}

Here are some important remarks about this result.

\medskip

\noindent $(i)$ We have already encountered the condition~\eqref{eq:condition_dual_potential} in Theorems~\ref{thm:truncation} and~\ref{thm:duality}. Similar comments apply to $T>0$. If $\bP\in \PiGC(\rho)$ is any grand-canonical probability such that $\cH(\bP,\bG_\rho)<\ii$, we can introduce the same state $\bP_\eps$ as in~\eqref{eq:deformed_Chayes} for $0\leq\eps\leq\bP_0/(1-\bP_0)$ and we have
\begin{multline*}
\cH(\bP_\eps,\bG_\rho)=(1+\eps)\cH(\bP,\bG_\rho)-(\bP_0-\eps-\eps\bP_0)\log\big(1+\eps-\eps\bP_0^{-1}\big)\\-\eps\log(\bP_0e^{\rho(\Omega)})+(1+\eps)(1-\bP_0)\log(1+\eps),
\end{multline*}
which is therefore finite. Like for $T=0$ we deduce that $\cF_T(\eta\rho)$ is finite for every $0\leq \eta<\frac{1}{1-\bP_0}$. Since the Gibbs state~\eqref{eq:Gibbs_psi} manifestly satisfies  $\bP_{\psi,0}>0$, we conclude that the condition~\eqref{eq:condition_dual_potential} is in fact a \emph{necessary condition} for a density $\rho$ to arise from a Gibbs state. The argument shows that for any $\rho$ there exists a critical $\eta_\rho$ such that
$$\cF_T(\eta\rho)\begin{cases}<\ii&\text{for all $\eta<\eta_\rho$}\\
=+\ii&\text{for all $\eta>\eta_\rho$.}
\end{cases}$$
If $\cF_T(\eta_\rho\rho)<\ii$ then we must have $\bP^{(T)}_0=0$ and there cannot exist a dual potential for $\eta_\rho\rho$. It is in fact well known in optimal transport that the location of the density $\rho$ in the convex set $\{\rho\ :\ \cF_T(\rho)<\ii\}$ plays a role for the existence of dual potentials. See for instance~\cite{Leonard-01,Leonard-14} which uses the concept of `intrinsic core' similar to~\eqref{eq:condition_dual_potential}.

\medskip

\noindent $(ii)$ To understand the meaning of the condition~\eqref{eq:condition_dual_potential}, think of a family of costs so that the $(c_n)_{n\geq N_c}$ are very different from the $(c_n)_{n< N_c}$. Then we may really have $\cF_T(\eta\rho)=+\ii$ for $\eta$ large enough if the density is more adapted to the costs for $n<N_c$ than it is for $n\geq N_c$. For instance, take $c_n=+\ii$ for all $n\geq N_c$, which amounts to truncating the grand-canonical problem. Then it is clear that $\eta_\rho\leq N_c/\rho(\Omega)$.

\medskip

\noindent $(iii)$ In~\cite{ChaChaLie-84}, the condition~\eqref{eq:condition_dual_potential} was replaced by the stronger assumption that there exists $N>\rho(\Omega)$ such that $\rho^{\otimes N}(c_N)<\ii$. By considering the state
$$\bP_0=1-\eps,\qquad \bP_N=\frac{\eps}{\rho(\Omega)^N}\rho^{\otimes N},\qquad0\leq\eps\leq1,$$
we see that $\cF_T(\eta\rho)<\ii$ for all $0\leq \eta\leq N/\rho(\Omega)$ hence~\eqref{eq:condition_dual_potential} holds.

\medskip

\noindent $(iv)$ In~\cite{ChaCha-84}, the condition~\eqref{eq:condition_dual_potential} was not mentioned. But~\cite[Lem.~3.3]{ChaCha-84} (which contains the argument described after Theorem~\ref{thm:truncation} and in $(i)$) is wrong since there it was tacitly assumed that $\bP_0>0$. Thus the condition~\eqref{eq:condition_dual_potential} should in fact be present in~\cite{ChaCha-84}.

\begin{proof}
The duality~\eqref{eq:duality_temp} follows from the lower semi-continuity shown in the proof of Theorem~\ref{thm:existence_temp}, similarly as in the proof of Theorem~\ref{thm:duality}. We only discuss the existence of the dual potential $\psi$ following~\cite{ChaChaLie-84,ChaCha-84}. Let $(\psi_n)$ be a maximizing sequence for the right side of~\eqref{eq:duality_temp},
$$\lim_{n\to\ii} \left(-T\log Z_{T,\rho}(\psi_n)+\int_\Omega\psi_n{\rm d}\rho\right)=\cF_T(\rho).$$
Using~\eqref{eq:condition_dual_potential}, we have
\begin{align*}
&-T\log Z_{T,\rho}(\psi_n)+\int_\Omega\psi_n{\rm d}\rho\\
&\qquad\qquad=-T\log Z_{T,\rho}(\psi_n)+(1+\eps)\int_\Omega\psi_n{\rm d}\rho-\eps\int_\Omega\psi_n{\rm d}\rho \\
&\qquad\qquad\leq \sup_{\psi}\left\{-T\log Z_{T,\rho}(\psi)+(1+\eps)\int_\Omega\psi{\rm d}\rho\right\}-\eps\int_\Omega\psi_n{\rm d}\rho\\
&\qquad\qquad=\cF_T\big((1+\eps)\rho\big)-\eps\int_\Omega\psi_n{\rm d}\rho.
\end{align*}
Thus we have shown that
\begin{equation}
\int_\Omega\psi_n{\rm d}\rho\leq\frac{\cF_T\big((1+\eps)\rho\big)-\cF_T(\rho)}{\eps}+o(1)_{n\to\ii}
\label{eq:estim_Chayes}
\end{equation}
and $\int_\Omega\psi_n{\rm d}\rho$ is bounded from above.
On the other hand, retaining only the $n=1$ term in the sum we have
$$Z_{T,\rho}(\psi_n)\geq e^{-\rho(\Omega)}\int_\Omega e^{\frac{\psi_n-c_1}{T}}\,{\rm d\rho}\geq e^{-\rho(\Omega)-\frac{\|c_1\|_{L^\ii}}{T}}\int_\Omega e^{\frac{\psi_n}{T}}\,{\rm d\rho}.$$
Since
\begin{equation}
Z_{T,\rho}(\psi_n)=e^{-\frac{\cF_T(\rho)}T+\frac1T\int_\Omega \psi_n{\rm d}\rho+o(1)}
\label{eq:asymp_Z}
\end{equation}
this proves by~\eqref{eq:estim_Chayes} that $e^{\psi_n/T}$ is bounded in $L^1(\Omega,{\rm d}\rho)$. Since $e^{\psi_n/T}\geq 1+(\psi_n)_+/T$, this property immediately implies that $(\psi_n)_+$ is bounded in $L^1(\Omega,{\rm d}\rho)$. Finally, retaining only the term $n=0$ we have
$Z_{T,\rho}(\psi_n)\geq e^{-\frac{c_0}{T}-\rho(\Omega)}.$
Hence $Z_{T,\rho}(\psi_n)$ cannot tend to $0$ and we conclude again from~\eqref{eq:asymp_Z} that $\int_\Omega(\psi_n)_-\rho$ is bounded. We have proved that $(\psi_n)$ is bounded in $L^1(\Omega,{\rm d}\rho)$.
Now we rephrase a beautiful argument of~\cite{ChaChaLie-84} in the form of a lemma.

\begin{lemma}[Weak convergence of the logarithm~\cite{ChaChaLie-84}]\label{lem:Lieb}
Assume that $\rho$ is a finite measure on $\Omega$. Let $(f_n)$ be a sequence of positive functions which converges weakly in $L^2(\Omega,{\rm d}\rho)$ to some $f\geq0$. Write $f_n=e^{\psi_n}$ and $f=e^\psi$ and assume in addition that $(\psi_n)$ is bounded in  $L^1(\Omega,{\rm d}\rho)$. Then $\psi$ belongs to $L^1(\Omega,{\rm d}\rho)$ and we have
\begin{equation}
 \limsup_{n\to\ii}\int_\Omega \psi_n\,{\rm d}\rho\leq \int_\Omega \psi\,{\rm d}\rho.
 \label{eq:Lieb_limsup}
\end{equation}
\end{lemma}

\begin{proof}[Proof of Lemma~\ref{lem:Lieb}]
Since $f$ is in $L^2(\Omega,{\rm d}\rho)$, it is also in $L^1(\Omega,{\rm d}\rho)$, and therefore $\int_\Omega e^\psi{\rm d}\rho$ is finite. This implies $\psi_+\in L^p(\Omega,{\rm d}\rho)$ for all $1\leq p<\ii$.
Following~\cite[Eq.~(2.21)--(2.28)]{ChaChaLie-84} we now introduce the truncated function $\psi^k:=\psi\1(|\psi|\leq k)$ and use Jensen's inequality in the form
$$ \int_\Omega\psi_n{\rm d}\rho-\int_\Omega\psi^k{\rm d}\rho\leq \log\left(\int_\Omega e^{\psi_n-\psi^k}{\rm d}\rho\right)=\log\left(\int_\Omega f_ne^{-\psi^k}{\rm d}\rho\right).$$
We pass to the limit $n\to\ii$, noticing that on the right side $e^{-\psi^k}$ is bounded, hence belongs to $L^2(\Omega,{\rm d}\rho)$. We find
\begin{equation}
  \limsup_{n\to\ii}\int_\Omega\psi_n{\rm d}\rho\leq \int_\Omega\psi^k{\rm d}\rho+\log\left(\int_\Omega e^{\psi-\psi^k}{\rm d}\rho\right).
 \label{eq:estim_k}
\end{equation}
Using that $e^{\psi-\psi^k}\leq 1+f\1(f\geq e^k)$ we find
$\int_\Omega e^{\psi-\psi^k}{\rm d}\rho\leq \rho(\Omega)+\int_\Omega f\,{\rm d}\rho$
which is independent of $k$. Since $\psi_+\in L^1(\Omega,{\rm d}\rho)$ and the limsup in~\eqref{eq:estim_k} is finite, this gives after passing to the limit $k\to\ii$ that $\psi_-\in L^1(\Omega,{\rm d}\rho)$. In fact, the logarithm on the right of~\eqref{eq:estim_k} converges to 0 by the monotone convergence theorem and we obtain~\eqref{eq:Lieb_limsup}.
\end{proof}

We apply the lemma to $f_n:=e^{\frac{\psi_n}{2T}}\wto f=e^{\frac{\psi}{2T}}$. We have
\begin{multline*}
\liminf_{n\to\ii}\int_{\Omega^N} \prod_{j=1}^N f_n(x_j)^2e^{-\frac{c_N(x_1,...,x_N)}{T}}{\rm d}\rho^{\otimes N}(x_1,...,x_N)\\
\geq \int_{\Omega^N} \prod_{j=1}^N f(x_j)^2e^{-\frac{c_N(x_1,...,x_N)}{T}}{\rm d}\rho^{\otimes N}(x_1,...,x_N)
\end{multline*}
since $(f_n)^{\otimes N}\wto f^{\otimes N}$ weakly in $L^2(\Omega^N,{\rm d}\rho^{\otimes N})$ and $e^{-c_N/T}\in L^\ii(\Omega^N,{\rm d}\rho^{\otimes N})$ from the stability of $\bc$. Thus
$\liminf_{n\to\ii} Z_{T,\rho}(\psi_n)\geq Z_{T,\rho}(\psi).$
By Lemma~\ref{lem:Lieb}, we obtain
$$\limsup_{n\to\ii}\left\{-T\log Z_{T,\rho}(\psi_n)+\int_\Omega \psi_n{\rm d}\rho\right\}\leq -T\log Z_{T,\rho}(\psi)+\int_\Omega \psi{\rm d}\rho$$
where $e^{\psi/T},\psi \in L^1(\Omega,{\rm d}\rho)$. Thus $\psi$ is a maximizer and this concludes the existence of an optimal potential. Computing the variations for $\psi+\eps\eta$ with $\eta\in L^\ii(\Omega,{\rm d}\rho)$ gives the equation that $\rho$ is the density of the Gibbs state $\bP_\psi$ in~\eqref{eq:Gibbs_psi}.
It is then a classical duality argument that $\bP_\psi$ minimizes the primal problem, hence is equal to $\bP^{(T)}$. This concludes the proof of Theorem~\ref{thm:existence_temp}.
\end{proof}

\bigskip

\noindent{\textbf{Acknowledgement.}} This project has received funding from the European Research Council (ERC) under the European Union's Horizon 2020 research and innovation programme (grant agreement MDFT No 725528 of M.L.), from the Agence Nationale de la Recherche (project MAGA ANR-16-CE40-0014 of L.N.), as well as from Ministero dell'Università e della Ricerca (PRIN project 202244A7YL of S.D.M. and MIUR Excellence Department Project 2023-2027 awarded to the DIMA of the University of Genova, CUP D33C23001110001), Air Force (AFOSR project FA8655-22-1-7034 of S.D.M.) and Istituto Nazionale di Alta Matematica (S.D.M. is a member of GNAMPA). S.D.M. wants to thank Lucia De Luca and Marcello Ponsiglione for suggesting ideas that led to the construction of Theorem \ref{thm:large_support_Coulomb}. The authors also thank the anonymous referee for useful comments.

\bigskip

\noindent{\textbf{Data statement.}} No datasets were generated or analyzed during the current study.


\end{document}